\documentclass[11pt,twoside,a4paper]{article}
\usepackage[english]{babel}
\usepackage{a4wide}
\usepackage{graphicx}
\usepackage{amsmath,amssymb,amsthm, amsgen}
\usepackage{listings}
\usepackage{color}
\usepackage[usenames,dvipsnames]{xcolor}
\usepackage{rotating}
\usepackage{hhline}
\usepackage{multirow}
\usepackage{enumerate}
\usepackage[bookmarks]{}
\usepackage{amsfonts}   
\usepackage{dsfont}
\usepackage{tikz}
\usepackage{longtable}
\usepackage{booktabs}
\usepackage{enumitem}
\usepackage{url}

\newtheorem{thm}{Theorem}[section]
\newtheorem{prop}[thm]{Proposition}
\newtheorem{prob}[thm]{Problem}
\newtheorem{lem}[thm]{Lemma}

\newtheorem{cor}[thm]{Corollary}
\newtheorem{defn}[thm]{Definition}

\newtheorem{rem}[thm]{Remark}

\newcommand{\essinf}{\operatorname*{ess\,inf}}
\newcommand{\esssup}{\operatorname*{ess\,sup}}

\newcommand{\indicatornoacc}[1]{ \mathds{1}_{ #1 } }

\newcommand{\N}{\mathbb N}

\newcommand{\plog}{p_0}

\newcommand{\R}{\mathbb R}

\newcommand{\sign}{\operatorname{sign}}
\newcommand{\supp}{\operatorname{supp}}

\newcommand{\xiext}{\zeta}
\newcommand{\xto}[1]{\xrightarrow{#1}}

\DeclareFontFamily{U}{mathx}{\hyphenchar\font45}
\DeclareFontShape{U}{mathx}{m}{n}{
      <5> <6> <7> <8> <9> <10>
      <10.95> <12> <14.4> <17.28> <20.74> <24.88>
      mathx10
      }{}
\DeclareSymbolFont{mathx}{U}{mathx}{m}{n}
\DeclareFontSubstitution{U}{mathx}{m}{n}
\DeclareMathAccent{\widecheck}{0}{mathx}{"71}

\setlist{noitemsep} 

\begin{document}

\newcommand{\Vext}{U}
\newcommand{\Vreg}{V_{ \operatorname{reg} }}
\newcommand{\Vextt}{\tilde U}
\newcommand{\kextt}{\ell}
\newcommand{\Vregt}{\tilde V_{ \operatorname{reg} }}
\newcommand{\rhot}{\tilde \rho}
\newcommand{\Vt}{\tilde V}
\newcommand{\Et}{\tilde E}
\newcommand{\Ct}{\tilde C}
\newcommand{\phit}{\tilde \phi}

\title{Regularity of the minimiser of one-dimensional interaction energies}

\author{M.~Kimura, P.~van Meurs}

\date{}

\maketitle 

%

\begin{abstract}
We consider both the minimisation of a class of nonlocal interaction energies over non-negative measures with unit mass and a class of singular integral equations of the first kind of Fredholm type. Our setting covers applications to dislocation pile-ups, contact problems, fracture mechanics and random matrix theory. Our main result shows that both the minimisation problems and the related singular integral equations have the same unique solution, which provides new regularity results on the minimiser of the energy and new positivity results on the solutions to singular integral equations.
\end{abstract}


\section{Introduction}
\label{sec:intro}

We consider the minimisation problem of the energy
\begin{equation} \label{for:defn:ER}
   E : \mathcal P (\R) \to [0, \infty], \qquad
   E (\rho) = \frac12 \iint_{\R \times \R} V ( t -  s) \, d (\rho \otimes \rho) ( s,  t) + \int_\R \Vext (t) \, d\rho (t), 
\end{equation}
where $\mathcal P (\R)$ is the space of probability measures, 
$V$ is an interaction potential which describes repulsive, nonlocal interactions,
and $\Vext$ is a confining potential. Figure \ref{fig:V} illustrates typical examples of $V$ and $\Vext$. The main assumptions on $V$ are that $V \in L^1(\R)$ is even on $\R$, that $V(r) \to +\infty$ as $r \to 0$, and that $V$ is non-increasing and convex on $(0,\infty)$. The main assumptions on $\Vext$ are that it is convex and $[0, \infty]$-valued. The precise assumptions on $V$ and $\Vext$ are given in Section \ref{ss:extn:R}.

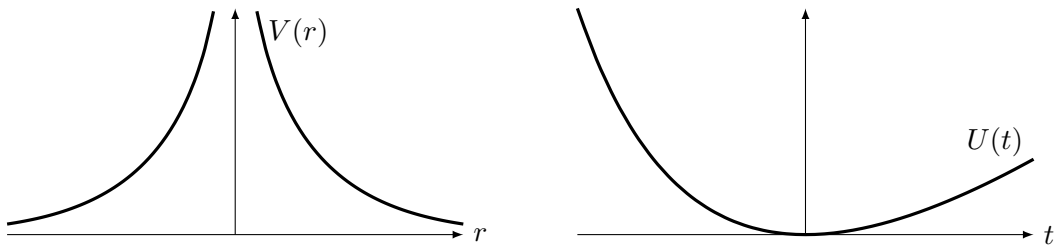
\begin{figure}[h]
\centering
\begin{tikzpicture}[scale=1.5, >= latex]    
\def \w {2}
        
\draw[->] (0,0) -- (0,\w);
\draw[->] (-\w,0) -- (\w,0) node[right] {$r$};
\draw[domain=0.19:\w, smooth, very thick] plot (\x,{ \x * cosh(\x) / sinh(\x) - ln( 2 * sinh(\x) ) });
\draw[domain=-\w:-0.19, smooth, very thick] plot (\x,{ \x * cosh(\x) / sinh(\x) - ln( -2 * sinh(\x) ) });
\draw (0.2, \w) node[anchor = north west]{$V(r)$};

\begin{scope}[shift={(2.5*\w,0)},scale=1]
\draw[->] (0,0) -- (0,\w);
\draw[->] (-\w,0) -- (\w,0) node[right] {$t$};
\draw[domain=-\w:\w, smooth, very thick] plot (\x,{ 4*(4/(\x + 4) + \x/4 -1) });
\draw (\w, .6) node[anchor = south east]{$\Vext(t)$};
\end{scope}
\end{tikzpicture} \\
\caption{Typical examples of $V$ and $\Vext$.}
\label{fig:V}
\end{figure}

Due to the convexity assumptions on $V$ and $\Vext$, uniqueness of the minimiser ${\overline \rho}$ of $E$ can be proven by standard methods (see, e.g., Proposition \ref{prop:P61:R}). Our main interest is in the regularity properties of ${\overline \rho}$. We will prove (Theorem \ref{t} and \ref{t:R}) that the support of ${\overline \rho}$ is a bounded interval $[t_1, t_2]$, that ${\overline \rho}$ is characterised as the solution to a singular integral equation, and that ${\overline \rho}$ on $(t_1, t_2)$ is as regular as $V'$ (away from $0$) and $\Vext'$. 

The regularity of ${\overline \rho}$ around $t_1$ and $t_2$ is more subtle, and depends on whether $\Vext$ attains the value $+\infty$. Since $\Vext$ is convex, this jump can occur at most twice, say at $-\infty \leq s_1 < s_2 \leq \infty$. We prefer to think of $s_1$ and $s_2$ as barriers which result in the hard constraint that the support of ${\overline \rho}$ is contained in $[s_1, s_2]$. Figure \ref{fig:rho:3} shows typical profiles of ${\overline \rho}$ in the case of two, one or no barriers, where $\Vext$ is chosen such that the support of ${\overline \rho}$ stretches until the barriers.

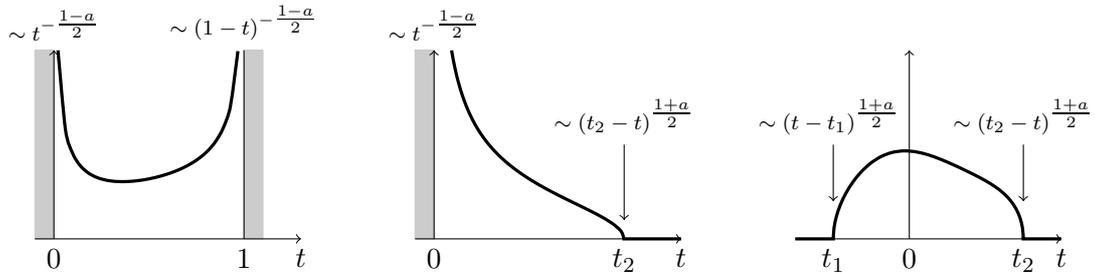
\begin{figure}[h]
\centering
\begin{tikzpicture}[scale=2.5]
\def \a {0.1}
\def \b {1}
\def \c {0.2}
\def \d {0.05}
\def \e {0.3}
\def \f {2}
\def \g {.009}
\def \lgray {black!20!white}

\begin{scope}[shift={(0,0)},scale=1]
  \filldraw[draw=\lgray,fill=\lgray] (-\a,0) rectangle (0, \b);
  \filldraw[draw=\lgray,fill=\lgray] (1,0) rectangle (1+\a, \b); 
  \draw[->] (-\a,0) -- (1+\e,0) node[below] {$t$};
  \draw[->] (0,0) node[below] {$0$} -- (0, \b) node[above] {\scriptsize $\sim t^{-\tfrac{1-a}2}$};
  \draw (1,0) node[below] {$1$} -- (1, \b)  node[above] {\scriptsize $\sim (1-t)^{-\tfrac{1-a}2}$};
  
  \draw[domain=.02156:.96815, smooth, very thick] plot (\x,{ 1/( 2*pi*sqrt(\x*(1-\x)) ) + \c*(\x - .5)});
\end{scope}

\begin{scope}[shift={(\f,0)},scale=1]
  \filldraw[draw=\lgray,fill=\lgray] (-\a,0) rectangle (0, \b); 
  \draw[->] (-\a,0) -- (1+\e,0) node[below] {$t$};
  \draw[->] (0,0) node[below] {$0$} -- (0, \b) node[above] {\scriptsize $\sim t^{-\tfrac{1-a}2}$};
  \draw (1,0) node[below] {$t_2$};
  \draw[->] (1,.5) node[above] {\scriptsize $\sim (t_2 - t)^{\tfrac{1+a}2}$} --++ (0,-.4);
  \draw[very thick] (1 - \g,0) -- (1+\e,0);
  
  \draw[domain=0:1, smooth, very thick] plot ({ 1/( 1 + (pi*\x)^2 ) }, \x);
\end{scope}

\begin{scope}[shift={(2*\f + .1, 0)},scale=1] 
  \draw[->] (-.2,0) -- (1.2,0) node[below] {$t$};
  \draw[->] (.4,0) node[below] {$0$} -- (.4, \b);
  \draw (0,0) node[below] {$t_1$};
  \draw (1,0) node[below] {$t_2$};
  \draw (.4,.5) node[anchor = south east] {\scriptsize $\sim (t - t_1)^{\tfrac{1+a}2}$};
  \draw[->] (0,.5) --++ (0,-.3);
  \draw[->] (1,.5) node[above] {\scriptsize $\sim (t_2 - t)^{\tfrac{1+a}2}$} --++ (0,-.3);
  \draw[very thick] (1 - \g,0) -- (1.2,0);
  \draw[very thick] (-.2,0) -- (\g,0);
  
  \draw[domain=0:pi, smooth, very thick] plot ({ ( 1+cos(180*\x/pi) )/2 },{ 1.4*sin(180*\x/pi)/pi - \d*sin( 180*cos(180*\x/pi) ) });
\end{scope}
\end{tikzpicture} \\
\caption{Typical profiles of the minimiser ${\overline \rho}$ of $E$ in the case of two, one, and no barriers (i.e., the points $s_i$ where $\Vext$ jumps from a finite value to $+\infty$). The value $a \in [0,1)$ is related to the singularity of $V$ around $0$; see \eqref{Va}. }
\label{fig:rho:3}
\end{figure}

The remainder of the introduction is organised as follows. After describing the applications of the regularity of ${\overline \rho}$, we separate two cases. In the first case, we consider $\Vext$ to have two barriers, and impose a rather artificial condition on $\Vext$ to ensure that ${\overline \rho}$ is as in the first of the three plots in Figure \ref{fig:rho:3}. The statement of the regularity results (Theorem \ref{t}) and the analysis turns out to be the easiest in this case. In the second case we consider $\Vext$ to be $[0, \infty)$-valued (i.e., no barriers). The treatment of this case builds further on the previous case, and requires additional arguments for the regularity of ${\overline \rho}$ around the endpoints $t_i$ of its support (Theorem \ref{t:R}). In the discussion afterwards we outline how the results of both cases can be applied to the general setting in which no conditions on $-\infty \leq s_1 < s_2 \leq \infty$ are put. Relying on this generalisation of our regularity results, we demonstrate how several limitations in the applications can be lifted.
%
%
%

\subsection{Applications}
\label{ss:applAi}

There is a wide range of applications for the minimisation problem of $E$ and the regularity of its minimiser ${\overline \rho}$. Most noteworthy is the application to interacting particle systems, where $E$ is the mean-field limit of a discrete interaction energy and $\rho$ is the (non-negative) particle density. See, e.g., \cite{SimioneSlepcevTopaloglu15} and the references therein for applications in statistical mechanics, models of collective behaviour of many-agent systems, granular media, self-assembly of nanoparticles, crystallization, and molecular dynamics simulations of matter.
While the one-dimensional scenario in \eqref{for:defn:ER} does not encompass the full complexity of many such particle system, it does capture, for instance, the log-gases studied in \cite{SandierSerfaty15} and the pile-ups of dislocations studied in \cite{GeersPeerlingsPeletierScardia13,HallChapmanOckendon10}. Since little is known on the regularity properties of ${\overline \rho}$, several results in these papers were not extended:
\begin{itemize}
  \item[(A1)] the applicability of the result in \cite{SandierSerfaty15} on crystallisation phenomena in log-gases; 
  \item[(A2)] the extension of the discrete-to-continuum convergence results in \cite{GeersPeerlingsPeletierScardia13} of dislocation pile-ups to convergence \emph{rates};
  \item[(A3)] the computation of the asymptotic expansions in \cite{HallChapmanOckendon10} of particle pile-ups to characterise boundary layers.
\end{itemize}
We show in Section \ref{s:appl} how our main results (Theorem \ref{t} and \ref{t:R}) lift the obstacles that impeded the progress on applications (A1)--(A3).

Outside of the setting of particle systems, and as the main application of \cite{SandierSerfaty15}, $E$ captures the setting in Random Matrix Theory (see \cite{AndersonGuionnetZeitouni10,Forrester10,Mehta04}) by putting $V(r) = -\log |r|$ and interpreting $\Vext$ as the so-called `external field'. In this setting, (the density of) ${\overline \rho}$ is called the 'density of states'. If ${\overline \rho}$ satisfies certain regularity properties (see, e.g., \cite{KuijlaarsMcLaughlin00}), then several statistics of the eigenvalues can be described as the size of the matrix tends to infinity. This poses the question of sufficient requirements on $\Vext$ such that ${\overline \rho}$ satisfies the desired regularity properties. The two main results on such requirements on $\Vext$ are stated in \cite{DeiftKriecherbauerMcLaughlin98,MhaskarSaff85}. In \cite{DeiftKriecherbauerMcLaughlin98} it is stated that, for \emph{analytic} $\Vext$ with appropriate growth conditions, the support of ${\overline \rho}$ is a finite union of closed intervals, that ${\overline \rho}$ is positive on the interior of those intervals, and that ${\overline \rho}$ `behaves like the square root' at the endpoints of each such interval. In \cite{MhaskarSaff85} the authors prove that for convex $\Vext$, the support of ${\overline \rho}$ is a single interval. As one application of our main result, we show that, for convex $\Vext$, the statements in \cite{DeiftKriecherbauerMcLaughlin98} on the regularity of ${\overline \rho}$ described above also hold when $\Vext$ is not analytic.

We also consider a less common application of $E$. As we will show in our analysis, ${\overline \rho}$ satisfies the singular integral equation of the first kind of Fredholm type given by
\begin{equation} \label{SIE}
  \text{p.v.}\int_{t_1}^{t_2} V'(t-s) {\overline \rho}(s) \, d s + \Vext' (t)  = 0
  \quad \text{for all } t_1 < t < t_2,
\end{equation}
where `p.v.' denotes the principle value integral. In fact, we show that ${\overline \rho}$ can be completely characterised as the solution to a singular integral equation with free boundaries. We refer to \cite{ChanFannjiangPaulino03,LifanovPoltavskiiVainikko03} and the references therein for applications of \eqref{SIE} to fracture mechanics, and to (\cite[\S 102]{Muskhelishvili53}) for applications to contact problems between two elastic bodies. 

\subsection{Case 1: two barriers and a fully supported minimiser}
\label{ss:Case1}

Let $\Vext$ have two barriers at $s_1 < s_2$. By using an affine variable transformation, we set $s_1 = 0$ and $s_2 = 1$. Then, we rewrite the minimisation problem of the energy in \eqref{for:defn:ER} as
\begin{equation} \label{for:defn:E:first}
   E (\rho) = \frac12 \iint_{[0,1]^2} V(t-s) \, d (\rho \otimes \rho) (s,t) + \int_{[0,1]} \Vext (t) \, d\rho (t)
\end{equation}
over the space $\mathcal P ([0,1])$. \medskip

\emph{Assumptions on the potentials $V$ and $\Vext$ in Case 1}. We assume that $V = V_a + \Vreg : [-1,1] \to [0, \infty]$ for some fixed $a \in [0,1)$, where $V_a$ is the Riesz potential given by
\begin{equation} \label{Va}
  V_a (r) := \left\{ \begin{aligned}
  &|r|^{-a} && \text{if } 0 < a < 1, \\
  &- \log |r| \: && \text{if } a = 0 
  \end{aligned}  \right. 
\end{equation}
and $\Vreg$ is the regular part which satisfies
\begin{subequations} \label{for:V:props}
\begin{gather} \label{for:V:props:lambda:psi}
  \Vreg \in \left\{ \begin{aligned}
    &W^{\ell+1,1} (-1,1)
    &&\text{if } 0<a<1, \\
    &W^{\ell+1,1} (-1,1) \cap W^{2,\plog} (-1,1)
    &&\text{if } a = 0
  \end{aligned} \right.
\end{gather}
for some integer $\ell \geq 1$ and some $1 < \plog < 2$.
We further assume that
\begin{gather} 
  V \text{ is even,}
  \quad V''(r) \geq 0 \ \text{for a.e.~} r \in (0,1), 
  \quad V(1) = -V'(1) \geq 0,  
  \\\label{for:V:props:lambda}
  \exists \, c, \varepsilon > 0 \ \forall \, r \in (0, \varepsilon) : \essinf_{0 < s < r} V''(s) \geq \frac c{ r^{2 + a} }.
\end{gather}
\end{subequations}
For the external potential $\Vext : [0,1] \to \R$ we assume that
\begin{equation} \label{for:Vext:assns}
  \Vext \in \left\{ \begin{aligned}
    &W^{\ell+1,1} (0,1)
    &&\text{if } 0<a<1 \\
    &W^{\ell+1,\plog} (0,1)
    &&\text{if } a = 0
  \end{aligned} \right\},
  \quad \Vext(t) \geq 0
  \quad \text{and} \quad \Vext''(t) \geq 0 \ \text{for a.e.~} t \in (0,1),
\end{equation}
where $\ell$ and $\plog$ are the same as in \eqref{for:V:props:lambda:psi}. Finally, we impose the artificial condition
\begin{equation} \label{VextV:bd}
  \sup_{(0,1)} |\Vext'| \leq |V'(1)|,
\end{equation}
which is sufficient for the support of the minimiser $\overline \rho$ to reach both barriers at $0$ and $1$.

Next we motivate several of the assumptions. The main assumptions are the regularity, the convexity and the splitting $V = V_a + \Vreg$. We motivate them in the sketch of the proof which follows Theorem \ref{t} below. The values of $\ell$ and $\plog$ regulate the regularity of $\Vreg$ and $\Vext$; for higher values, Theorem \ref{t} states stronger regularity results on $\overline \rho$. The condition $V'(1) \leq 0$, in addition to convexity, ensures that there are no attractive forces between particles. Because of the unit mass constraint, adding a constant to $V$ or $U$ is equivalent to adding a constant to the energy $E$. We choose this constant such that $V$ and $\Vext$ are non-negative. In particular, we tune this constant such that $V(1) = - V'(1)$ for convenience later on. \medskip

\emph{Proper definition of $E$}. Using assumptions \eqref{for:V:props} and \eqref{for:Vext:assns}, we show that \eqref{for:defn:E:first} is well-defined. Since $\Vext \in C([0,1])$, the second term in \eqref{for:defn:E:first} is finite for any $\rho \in \mathcal P([0,1])$. Since $V$ is lower semi-continuous and non-negative, the first term in \eqref{for:defn:E:first} is well-defined with values in $[0, \infty]$ for any $\rho \in \mathcal P([0,1])$\footnote{To prove this, take any continuous approximation of $V$ from below, and pass to the limit by using the Monotone Convergence Theorem.}. Likewise, the convolution 
\begin{equation} \label{convo:Vrho}
    (V*\rho)(t) := \int_{[0,1]} V(t-s) \, d\rho(s)
  \quad \text{with } \rho \in \mathcal P([0,1])  
\end{equation}
is well-defined as a lower semi-continuous function on $[0,1]$. In particular, we note that $E(\rho) = \infty$ whenever $\rho$ has an atom (i.e., a delta-peak). Neglecting measures with atoms, the notation simplifies to
\begin{equation} \label{for:defn:E}
   E (\rho) 
   = \frac12 \int_0^1 \int_0^1 V(t-s) \, d \rho(s) \, d \rho(t) + \int_0^1 \Vext (t) \, d\rho (t)
   = \frac12 \int_0^1 (V*\rho) \, d \rho + \int_0^1 \Vext \, d\rho.
\end{equation}
\medskip


\emph{Main result of Case 1}. First, we prepare the setting for stating the main result in this section, Theorem \ref{t}. Besides the regularity properties of $\overline \rho$, Theorem \ref{t} also states that $\overline \rho$ is completely characterised as either the minimiser of $E$, as the solution to a variational inequality, or as the solution to a weakly singular integral equation. We introduce the related three problems as Problems \ref{prob:minz}--\ref{prob:wSIE}, which we consider of independent interest on their own. In particular, Problem \ref{prob:wSIE} is the integrated version of \eqref{SIE}. 

Let $0 \leq a < 1$ be fixed, and $E$ be as in \eqref{for:defn:E} with the related potentials $V$, $\Vreg$ and $\Vext$ as in \eqref{for:V:props}, \eqref{for:Vext:assns} and \eqref{VextV:bd}.

\begin{prob}[Minimisation] \label{prob:minz}
  Find the minimiser of $E$ in $\mathcal P ([0,1])$.
\end{prob}

In view of \eqref{convo:Vrho}, we define for any $\rho \in \mathcal P ([0,1])$ the lower semi-continuous function 
\begin{equation*} 
  h_\rho := V*\rho + \Vext : [0,1] \to [0, \infty].
\end{equation*}

\begin{prob}[Variational inequality] \label{prob:VI}
  Find $\rho \in \mathcal P ([0,1])$ such that
  \begin{equation} \label{for:prop:E:on:P:EL:weak}
  \int_{[0,1]} h_{\rho} \, d \mu 
  \geq \int_{[0,1]} h_{\rho} d \rho
  \quad \text{for all } \mu \in \mathcal P ([0,1]) \text{ with } E(\mu) < \infty.
\end{equation}
\end{prob}    

To introduce the weakly singular integral equation, we first set up the functional framework. We define the fractional Sobolev space for $s > 0$ by
\begin{equation} \label{Hs}
  H^{-s} (\R) 
  := \big\{ \zeta \in \mathcal S'(\R) : {\textstyle \int_\R} (1 + \omega^2)^{-s} \big| \widehat \zeta (\omega) \big|^2 \, d\omega < \infty \big\},
\end{equation}
where $\widehat \zeta := \mathcal F \zeta$ is the Fourier transform (defined in Section \ref{s:prel}) and $\mathcal S'(\R)$ is the space of tempered distributions, i.e., the dual of the Schwartz space $\mathcal S(\R)$. With this interpretation, we define the subspace
 \begin{equation*} 
  H^{-s} (0,1) 
  := \big\{ \zeta \in H^{-s} (\R) : \supp \zeta \subset [0,1] \big\}.
\end{equation*}
In Section \ref{ssec:Vip} we show that, for $f, g \in L^2(0,1)$,
\begin{equation} \label{for:defn:Vip}
  (f, g)_V := \int_0^1 (V * f)(t) g (t) \, dt
\end{equation}
defines an inner product on $L^2(0,1)$, and that the norm induced by this inner product  
is equivalent to that of $H^{-(1-a)/2} (0,1)$. Hence, since the interaction term of $E(\rho)$ reads as $\frac12 (\rho, \rho)_V$, $H^{-(1-a)/2} (0,1)$ turns out to be the largest space to seek solutions to the weakly singular integral equation with finite energy. 

\begin{prob}[Weakly singular integral equation] \label{prob:wSIE}
Find the solution $(\rho, C)$ where $C \in \R$ and $\rho \in H^{-(1-a)/2} (0,1)$ with $\int_0^1 \rho = 1$ to
\begin{equation} \label{for:prop:E:on:P:EL:strong}
    h_\rho = C
    \quad \text{a.e.~on } (0,1).
\end{equation}
\end{prob}

We note that, in Problem \ref{prob:wSIE}, non-negativity of $\rho$ is not required. We give in Section \ref{ssec:Vip} a proper meaning to $\int_0^1 \rho = 1$ and $h_\rho$ for $\rho \in H^{-(1-a)/2}(0,1)$. 

\begin{thm}[Properties of the minimiser ${\overline \rho}$ in Case 1] \label{t} 
Let $a \in [0, 1)$. Let $V_a$ be as in \eqref{Va} and assume that $\Vreg$ and $\Vext$ satisfy \eqref{for:V:props}, \eqref{for:Vext:assns} and \eqref{VextV:bd} with corresponding integer $\ell \geq 1$ and $p_0 \in (1, 2)$. Then, all three Problems \ref{prob:minz}--\ref{prob:wSIE} have a unique solution, and all these solutions are equal. Let $(\overline \rho, \overline C)$ be this solution. Then, the support of $\overline \rho$ is $[0,1]$ and 
$$
  \overline C 
  = ( \overline \rho, \overline \rho)_V^2 + \int_0^1 \Vext \, d\overline \rho
  > 0. 
$$
On $[0,1]$, $\overline \rho$ has a density which satisfies
\begin{equation} \label{for:prop:E:on:P:regy:Hol}
  \left. \begin{array}{l}
     \overline \rho(t) = C_1 t^{- \tfrac{1-a}2} + R_1(t)  \\
     \overline \rho(t) = C_2 (1-t)^{- \tfrac{1-a}2} + R_2(1-t)
    \end{array} \right\}  
  \quad \text{for all $t \in [0,1]$}
\end{equation}
for some constants $C_i \geq 0$ and some functions $R_i$ satisfying
\begin{equation*}
  R_i \in \left\{ \begin{aligned}
    \{ f \in C^a ([0, \tfrac12])
    & : f(0) = 0 \}
    && \text{if } 0 < a < 1, \\
    \{ f \in C^{1 - 1/\plog} ([0, \tfrac12])
    & : f(0) = 0 \}
    && \text{if } a=0.
  \end{aligned} \right.
\end{equation*} 
Away from the endpoints of $[0,1]$, $\overline \rho$ satisfies
\begin{equation} \label{for:prop:E:on:P:regy:Wloc}
  \left\{ \begin{aligned}
    \overline \rho &\in W^{\kextt, p}_{\operatorname{loc}} (0,1)
    && \text{if } 0 < a < 1, \\
    \overline \rho &\in W^{\kextt, p_0}_{\operatorname{loc}} (0,1)
    && \text{if } a=0
  \end{aligned} \right.
\end{equation}
for any $1 \leq p < (1-a)^{-1}$. 
Finally, if $\ell \geq 3$, then $\overline \rho > 0$ on $(0,1)$.
\end{thm}

\emph{Outline of the proof of Theorem \ref{t}}. First, we follow the standard approach in the calculus of variations to show that Problems \ref{prob:minz} and \ref{prob:VI} have a unique solution $\overline \rho$. The main observation to do this is that $E$ is coercive with respect to the norm induced by \eqref{for:defn:Vip}. We rely on the convexity properties of $V$ to show that \eqref{for:defn:Vip} indeed defines an inner product.

Second, we show that $\overline \rho$ satisfies Problem \ref{prob:wSIE}. While it is obvious from \eqref{for:prop:E:on:P:EL:weak} that $h_{\overline \rho}$ is constant a.e.~on $\supp \overline \rho$, it is not clear why $\supp \overline \rho = [0,1]$. The novelty of our proof is the observation that $h_{\overline \rho}$ is as regular as $\Vreg$ and $\Vext$ on $(\supp \overline \rho)^c$. Then, using the convexity properties of $V$ and $\Vext$, we obtain from \eqref{for:prop:E:on:P:EL:weak} by contradiction that $\supp \overline \rho = [0,1]$. 

Third, we use the splitting $V = V_a + \Vreg$ to write \eqref{for:prop:E:on:P:EL:strong} as 
\begin{equation*} 
  V_a * \rho = C - \Vreg * \rho - \Vext =: f_\rho \quad \text{a.e.~on } (0,1).
\end{equation*}
Plugging in $\rho = \overline \rho$ and $C = \overline C$ in the right-hand side, the resulting equation for $\rho$ has been solved explicitly in \cite{Carleman22}, whose solution we denote as
$
  \rho = \mathcal C_a f_{\overline \rho},
$
where $\mathcal C_a$ is a linear operator. Since the expression of $\mathcal C_a$ is rather technical (it relies on a fractional derivative and the Hilbert transform), we postpone its definition to \eqref{def:Ca}. Since we can characterise $f_{\overline \rho}$ only in terms of properties of $\overline \rho$, we require that the formula $\rho = \mathcal C_a f_{\overline \rho}$ is valid for a large enough class of functions $f_{\overline \rho}$. Since such a statement appears to be missing in the literature, we establish it in Theorem \ref{thm:expl:soln:a} for $0 < a < 1$ and Theorem \ref{thm:expl:soln:log} for $a=0$. It is here that we employ the asserted regularity on the potentials $\Vreg$ and $\Vext$. We argue that Carleman's solution has to coincide with $\overline \rho$, which results in the implicit formula
\begin{equation} \label{soln:form:Abs}
   \overline \rho = \mathcal C_a f_{\overline \rho}.
\end{equation}

Fourth, from \eqref{soln:form:Abs} we derive all the regularity properties of $\overline \rho$ as listed in Theorem \ref{t} by combining together several established properties of fractional derivatives and the Hilbert transform. Our proof of the positivity of $\overline \rho$ requires a point-wise evaluation of $\overline \rho''$, which is guaranteed by \eqref{for:prop:E:on:P:regy:Wloc} only if $\ell \geq 3$.

\subsection{Case 2: no barriers}
\label{ss:extn:R}

We return our attention to the general form of the energy $E$ introduced in \eqref{for:defn:ER}. In Case 2, we assume that $\Vext$ is $[0, \infty)$-valued, i.e., $s_1 = -\infty$ and $s_2 = +\infty$. We present the main result of this section, Theorem \ref{t:R}, in a similar manner as in Section \ref{ss:Case1}. To avoid repetition and for the sake of conciseness, we will state in Theorem \ref{t:R} only the regularity result of the minimiser ${\overline \rho}$, and refer to Section \ref{s:FBVP} for the full statement (Theorem \ref{t:R:extended}) in which the counterparts of Problems \ref{prob:minz}--\ref{prob:wSIE} are shown to have ${\overline \rho}$ as their unique solution. \medskip

\emph{Assumptions on the potentials $V$ and $\Vext$ in Case 2}. We assume that $V = V_a + \Vreg : \R \to [0, \infty]$ with $V_a$ as in \eqref{Va} for some fixed $a \in [0,1)$. On $\Vreg$, we assume that
\begin{subequations} \label{for:V:on:R:props}
\begin{equation} \label{for:V:on:R:props:psi}
  \Vreg \in    
  \left\{ \begin{aligned}
    &W^{\ell + 1,1}_{\text{loc}} (\R)
    &&\text{if } 0<a<1, \\
    &W^{\ell + 1,1}_{\text{loc}} (\R) \cap W^{2,\plog}_{\text{loc}} (\R)
    &&\text{if } a = 0
  \end{aligned} \right.
\end{equation}
for some integer $\ell \geq 1$ and some $1 < \plog < 2$. We further assume that
\begin{gather} \label{for:V:on:R:props:genl}
  V \in L^1(\R) \text{ is even}, \quad
  V''(r) \geq 0 \ \text{for a.e.~} r \in \R, \\\label{for:V:on:R:props:lambda}
  \exists \, c, \varepsilon > 0 \ \forall \, r \in (0, \varepsilon) : \essinf_{0 < s < r} V''(s) \geq \frac c{ r^{2 + a} }.
\end{gather}  
\end{subequations}
Note that \eqref{for:V:on:R:props} implies $V \geq 0$. We assume on the external potential $\Vext$ that
\begin{subequations} \label{for:Vext:on:R:assns}
\begin{gather} 
  \Vext \in \left\{ \begin{aligned}
    &W^{\ell + 1,1}_{\text{loc}} (\R)
    &&\text{if } 0<a<1, \\
    &W^{\ell + 1,\plog}_{\text{loc}} (\R)
    &&\text{if } a = 0,
  \end{aligned} \right. \\  
  \Vext(t) \geq \Vext (0) = 0, \quad
  \Vext''(t) \geq 0 \ \text{for a.e.~} t \in \R,
  \quad \exists \, c, C > 0 \ \forall \, t \in \R : \Vext ( t) \geq c (| t| - C), 
\end{gather}
\end{subequations}
where $\ell$ and $\plog$ are the same as in \eqref{for:V:on:R:props:psi}. 

Our motivation for these assumptions is as follows. The condition $V \in L^1 (\R)$ in \eqref{for:V:on:R:props:genl} together with the convexity imply that $V (r) \geq 0 \geq V'(r)$ for all $r > 0$, which is similar to the assumptions on $V$ in Case 1. Given that we require $\Vext$ to be a convex, confining potential, the further assumptions in \eqref{for:Vext:on:R:assns} are minimal. Since the interaction part of $E$ is translation invariant, we choose $\Vext$ to be minimal at $t = 0$ without loss of generality. 

\begin{thm}[Properties of the minimiser ${\overline \rho}$ in Case 2: short version] \label{t:R}
Let $a \in [0, 1)$. Let $E$ be as in \eqref{for:defn:ER} with the related $V$, $\Vreg$, $\Vext$, integer $\ell \geq 1$ and $1 < \plog < 2$ defined in \eqref{for:V:on:R:props} and \eqref{for:Vext:on:R:assns}. 
Then, $E$ has a unique minimiser ${\overline \rho} \in \mathcal P (\R)$. The support of ${\overline \rho}$ is a finite interval $[t_1, t_2]$, and ${\overline \rho}$ has a density which satisfies
\begin{equation} \label{for:thm:R:regy}
  {\overline \rho} \in \left\{ \begin{aligned}
    C_0^a ([t_1, t_2]) &\cap W^{\kextt, p}_{\operatorname{loc}} (t_1, t_2)
    && \text{if } 0 < a < 1, \\
    C_0^{\alpha_0} ([t_1, t_2]) &\cap W^{\kextt, p_0}_{\operatorname{loc}} (t_1, t_2)
    && \text{if } a=0
  \end{aligned} \right.
\end{equation}
for any $1 \leq p < (1-a)^{-1}$, where $\alpha_0 := 1 - 1/\plog$. Moreover, if $\Vext \in C^2([t_1, t_2])$ (in the case $a=0$, we require $\Vext \in C^{2, \alpha_0}([t_1, t_2])$), then 
\begin{equation} \label{for:thm:R:expa}
  {\overline \rho}(t) = C_i |t - t_i|^{\tfrac{1+a}2} + R_i(|t - t_i|) \quad \text{for all $t \in [t_1, t_2]$ and $i = 1,2$}
\end{equation}
for some constants $C_i \geq 0$ and some functions $R_i$ satisfying 
\begin{equation*}
  R_i \in \left\{ \begin{aligned}
    \{ f \in C^{1,a} ([0, t_2 - t_1])
    & : \exists \, C > 0 : |f(s)| \leq C s^{1+a} \}
    && \text{if } 0 < a < 1, \\
    \{ f \in C^{1,\alpha_0} ([0, t_2 - t_1])
    & : \exists \, C > 0 : |f(s)| \leq C s^{1+\alpha_0} \}
    && \text{if } a=0. 
  \end{aligned} \right.
\end{equation*}
Finally, if $\ell \geq 3$, then ${\overline \rho} > 0$ on $(t_1, t_2)$.
\end{thm}

\emph{Comments on Theorem \ref{t:R}}. The main difference between the statements of Theorems \ref{t} and \ref{t:R} is the regularity of ${\overline \rho}$ at $t_1$ and $t_2$ (as expected from Figure \ref{fig:rho:3}). In the setting of Theorem \ref{t:R}, it turns out that the constants in \eqref{for:prop:E:on:P:regy:Hol} equal $0$, which corresponds to ${\overline \rho}$ being H\"older continuous at $t_1$ and $t_2$. Motivated by application (A1) of Section \ref{ss:applAi}, we continue the expansion in \eqref{for:thm:R:expa} to the next order term at the small cost of a stronger regularity requirement on $\Vext$. 

The outline of the proof of Theorem \ref{t:R} is as follows. First, we rely on the linear growth of $\Vext$ to show that $E$ has a unique minimiser ${\overline \rho} \in \mathcal P(\R)$, and that ${\overline \rho}$ has bounded support. Then, on this bounded support, we show that Theorem \ref{t} applies to obtain regularity and positivity of ${\overline \rho}$ on $(t_1, t_2)$. To continue the expansion of ${\overline \rho}(t)$ around $t = t_i$, we rely on the equivalent of Problem \ref{prob:VI} which, in addition to the statements of Theorem \ref{t}, gives a lower bound on
\begin{equation} \label{hrho:t}
  h_{\overline \rho} (t) = \int_\R V ( t -  s) \, d \overline \rho ( s) + \Vext ( t) \in [0, \infty]
  \quad \text{for all } t \in \R
\end{equation}
outside of $[t_1, t_2]$.

\subsection{Discussion}
\label{ss:disc}

In the discussion below, we consider both Cases 1 and 2 at the same time unless mentioned otherwise.
\smallskip

\textit{Singular integral equation}. 
In the setting of Theorem \ref{t} (with $\ell = 2$ if $0 < a < 1$), $\overline \rho$ satisfies the singular integral equation \eqref{SIE}. Formally, \eqref{SIE} is obtained from Problem \ref{prob:wSIE} by differentiating $h_{\overline \rho} (t) = C$ and by exchanging the order of integration and differentiation, i.e., $(V_a * \overline \rho)' = V_a' * \overline \rho$. To justify these steps rigorously, we apply \cite[\S 4.2]{Mangler51}, which requires $\overline \rho$ to be H\"older continuous on compact subsets of $(0,1)$ with exponent greater than $a$; hence the condition $\ell = 2$. 
\smallskip

\textit{Expansion of ${\overline \rho}$ around $t_1$ and $t_2$}. We expect that our proof allows for a continuation of the expansions \eqref{for:prop:E:on:P:regy:Hol} and \eqref{for:thm:R:expa} for a large enough integer $\ell$ in \eqref{for:V:props:lambda:psi} and \eqref{for:Vext:assns}. For the sake of simplicity, we have stopped at the first order.
\smallskip

\textit{Positivity of ${\overline \rho}$}. While Theorems \ref{t} and \ref{t:R} state that $\ell \geq 3$ is sufficient for ${\overline \rho} > 0$ on $(0,1)$, we expect that $\ell \geq 1$ may be sufficient too. Indeed, for special choices of $V$ and $\Vext$ for which we can compute $ {\overline \rho}$ explicitly (see Section \ref{s:appl} for examples), it always turns out that $ {\overline \rho} > 0$ on $(0,1)$.
\smallskip

\textit{The general case of $[0,\infty]$-valued $\Vext$}. 
Consider the extension of the setting of Theorem \ref{t:R} in which $\Vext$ is allowed to jump to $+\infty$ at $- \infty \leq s_1 < s_2 \leq +\infty$. This setting includes, for instance, Theorem \ref{t} without the artificial assumption in \eqref{VextV:bd}. Then, most arguments in the proof of Theorem \ref{t:R} can be repeated with obvious modification. This immediately provides the existence and uniqueness of the minimiser ${\overline \rho}$, which moreover has bounded support (see Figure \ref{fig:rho:3} for typical graphs of ${\overline \rho}$). It also provides the bound on the possible blow up of ${\overline \rho}$ at $t_i$ as given by \eqref{for:prop:E:on:P:regy:Hol}. Since the expansion of ${\overline \rho}$ at $t_i$ in Theorem \ref{t:R} relies on a local argument, it applies whenever $t_i \neq s_i$.

In practice, it can be difficult to compute $t_i$ by minimising $E$, and thus the property `$t_i \neq s_i$' required for \eqref{for:thm:R:expa} can be hard to check. If it is not a priori clear whether $t_i = s_i$, then the conditions on $\Vext$ allow one to replace the jump at $s_i$ by an affine extension of class $C^1$. Then, Theorem \ref{t:R} applies to this altered setting, and its extended version (given by Theorem \ref{t:R:extended}) provides several other characterisations for the possibly altered minimiser $\rho^*$. These characterisations can be used to check whether $\supp \rho^*$ is contained in $[s_1, s_2]$, from which it follows which of the situations in Figure \ref{fig:rho:3} fits to ${\overline \rho}$. 
\smallskip

\textit{Examples and applications}.
The equivalence between Problems \ref{prob:minz}--\ref{prob:wSIE} (and the equivalence between Problems \ref{prob:minz:R}--\ref{prob:sSIE:R}) and the properties of ${\overline \rho}$ are valuable both for developing efficient and accurate numerical solution methods and for proving many-particle limits of related interacting particle systems. In Section \ref{s:appl} we demonstrate the applicability of Theorem \ref{t} and Theorem \ref{t:R} by lifting the limitations in (A1)--(A3) and extending the class of potentials $\Vext$ in the result of \cite{DeiftKriecherbauerMcLaughlin98}. In a future publication, we use Theorems \ref{t} and \ref{t:R} to pursue (A2) and to tackle the open problem on the discrete part of the boundary layer result in \cite{GarroniVanMeursPeletierScardia16}.
\smallskip

\textit{Extension of Theorems \ref{t} and \ref{t:R} to higher dimensional domains}. 
Most applications of nonlocal energies as in \eqref{for:defn:ER} are set on $\R^d$ with $d \geq 2$. The existence and uniqueness of minimisers of $E$ can be extended to higher dimensions (see, e.g., \cite[Thm.~3.1]{MoraRondiScardia16ArXiv}). However, for proving the remaining statements of Theorems \ref{t} and \ref{t:R}, we heavily rely on the one-dimensional setting. While we prove that $\supp {\overline \rho}$ is an interval in the one-dimensional setting, in the higher-dimensional setting it is not even clear whether $\supp {\overline \rho}$ has full dimension. Indeed, for Coulomb interactions on a bounded domain $\Omega \subset \R^2$, it is well-known that $\supp {\overline \rho}$ concentrates on $\partial \Omega$. Moreover, when the domain is $\R^d$, \cite{BalagueCarrilloLaurentRaoul14,MoraRondiScardia16ArXiv} provide examples where the dimension of $\supp {\overline \rho}$ is smaller than $d$ while $\supp {\overline \rho} \subset \subset \Omega$.

Even if $\supp {\overline \rho} \subset \R^d$ can be shown to have full dimension with sufficiently smooth boundary, the next challenge would be to prove regularity properties of ${\overline \rho}$. In this paper we rely on the explicit solution formula for Carleman's equations of the form \eqref{soln:form:Abs}. In higher dimensions such a formula is not available. One exception is thanks to the particular setting in \cite{Kahane81}, where the singular integral equation $V_a * \rho = g$ is solved on the unit ball $B \subset \R^2$ for $g \in C^2(\overline B)$. \smallskip

The remainder of the paper is organised as follows. In  Section \ref{s:prel} we set the notation and recall several results from textbooks. In  Section \ref{ssec:Vip} we define the Hilbert space $H_V (0,1)$ equipped with $(\cdot, \cdot)_V$ and characterise its properties. In  Section \ref{s:Carl} we derive a precise regularity result for the unique solution of $V_a * \rho = f$ on $(0,1)$, which was first constructed as $\rho = \mathcal C_a f$ in \cite{Carleman22}. In Section \ref{s:pfs} we prove Theorem \ref{t}. In Section \ref{s:FBVP} we state and prove Theorem \ref{t:R:extended}; the extended version of Theorem \ref{t:R}. In Section \ref{s:appl} we treat several examples in the literature (including (A1)--(A3)) for which Theorem \ref{t} and Theorem \ref{t:R} lift restrictions or guarantee stronger regularity properties. Appendices \ref{app:pf:lem:props:V} and \ref{app:pf:thm:expl:soln:log} contain computationally heavy proofs.

\section{Preliminaries}
\label{s:prel}

In this section we list the symbols used throughout the paper, and cite several textbook results on which we rely in the subsequent sections.

\newcommand{\specialcell}[2][c]{\begin{tabular}[#1]{@{}l@{}}#2\end{tabular}}
\begin{longtable}{lll}
$b \wedge c$, $b \vee c$ & $\min \{b, c\}$, $\max \{b, c\}$ & \\
$\|f\|_p$ & $L^p$-norm of $f$ on the domain of $f$ & \\
$\widehat f$, $\mathcal{F}(f)$ & \specialcell[t]{Fourier transform of $f$; \\ $\mathcal{F}(f)(\omega) = \widehat f (\omega) := \int_\R f(t) e^{-2\pi i \omega t} \, dt$} & \\
$f (t +)$ & one-sided limit of $f$ at $t$ from above; $f( t+) := \lim_{ s \downarrow  t} f( s)$ \\
$f (t -)$ & one-sided limit of $f$ at $t$ from below; $f( t-) := \lim_{ s \uparrow  t} f( s)$ \\
$\indicatornoacc A$ & $\indicatornoacc A (x)$ equals $1$ if $x \in A$ and $0$ if $x \notin A$ & \\
$B_r(x)$ & ball of radius $r$ centred at $x$ & \\
$\Gamma$ & $\Gamma (\alpha) := \int_0^\infty t^{\alpha-1} e^{-t} \, dt$ for any $\alpha > 0$ & \\
$C^\alpha([0,1])$ & H\"older space; $0 < \alpha < 1$ & \\
$C^{k,\alpha}([0,1])$ & $\{ f \in C^\alpha([0,1]) : f^{(k)} \in C^\alpha([0,1]) \}$; $k \in \N$ & \\
$C_0^\alpha([0,1])$ & $\{ f \in C^\alpha([0,1]) : f(0) = f(1) = 0 \}$ & \\
$C_0^\alpha(\phi_a)$ & weighted H\"older space & \eqref{for:defn:wHolS} \\
$C_b(\R)$ & space of bounded, continuous functions on $\R$ & \\
$H^s (\R)$ & fractional Sobolev space; $s \in \R$ & \eqref{Hs} \\
$H_V (t_1, t_2)$ & Hilbert space of functions $f : (t_1, t_2) \to \R$ & \eqref{for:defn:HVR}, \eqref{for:defn:HV} \\
$\mathcal L$ & Lebesgue measure on $\R$ & \\
$\mathcal M ([0,1])$ & space of finite, signed Borel measures on $[0,1]$ & \\
$\N_+, \N$ & $\N_+ := \{1,2,3,\ldots\}$; $\N = \{0\} \cup \N_+$ &  \\
$\mathcal P ([0,1])$ & space of probability measures; $\mathcal P ([0,1]) \subset \mathcal M ([0,1])$ & \\
$p_0$ & regularity constant of $\Vreg$ and $U$ when $a = 0$; $1 < p_0 < 2$ & \eqref{for:V:props:lambda:psi}, \eqref{for:V:on:R:props:psi} \\
$S$ & singular integral operator & \eqref{for:defn:S} \\
$\mathcal S (\R)$ & space of Schwartz functions & \\
$T_\# \rho$ & push-forward of $\rho \in \mathcal P (\R)$ by $T : \R \to \R$; & \\
& $T_\# \rho (A) := \rho (T^{-1}(A))$ & \\
$wL^p (0,1)$ & weak-$L^p$ space & \eqref{for:defn:wLp} \\
\end{longtable} 
We reserve the symbols $C, C',$ etc. for generic positive constants which we leave unspecified. 
\smallskip

We continue by listing several well-known results. The first one is a basic theorem in the calculus of variations (see, e.g., \cite[Thm.~2.1]{KinderlehrerStampacchia80}). We use it in the proof of Theorem \ref{t} to show that Problems \ref{prob:minz} and \ref{prob:VI} are equivalent.

\begin{thm} [Characterisation of minimiser] \label{thm:CoV}
Let $X$ be a Hilbert space, $K \subset X$ closed and convex, $f \in X$, and $\mathsf E: X \to \R$ be given by $\mathsf E(u) = \tfrac12 \|u\|_X^2 - (f,u)_X$. Then $\mathsf E$ has a unique minimiser in $K$, which is characterised by the unique solution $u \in K$ of the variational inequality:
\begin{equation*}
  0 \leq (u - f, v-u)_X
  \quad \text{for all } v \in K.
\end{equation*}

\end{thm}

Next we introduce several integral operators in preparation for defining $\mathcal C_a$ in \eqref{def:Ca}. For H\"older-continuous functions $f : [0,1] \to \R$, we define the singular integral operator 
\begin{equation} \label{for:defn:S}
  (S f)(t) := \text{p.v.}\int_0^1 \frac{f(s)}{t-s} \, ds
  \qquad \text{for all } 0 < t < 1.
\end{equation}
For later use, we extend $S$ to functions defined on any given bounded interval $[t_1, t_2]$ by
\begin{equation} \label{for:defn:S12}
  \big( S_{t_1}^{t_2} f \big)(t) := \text{p.v.}\int_{t_1}^{t_2} \frac{f( s)}{ t -  s} \, d  s
  \qquad \text{for all } t \in (t_1, t_2).
\end{equation}

The operator $-\tfrac1\pi S$ is known as the \emph{finite Hilbert transform}, which attains a natural extension to $L^p$-spaces (see Proposition \ref{prop:props:S}). For any $a, t_0 \in (0,1)$, we recall the identities
\begin{align} \label{EK00}
 S \Big( \frac{ \indicatornoacc{(0, t_0)} (t) }{t^{1-a} |t_0 - t|^a} \Big)
 &= \frac{ 1 }{ t^{1-a} |t_0 - t|^a } \left\{ \begin{aligned}
     &\tfrac\pi{ \tan (a \pi) }
     &&\text{if } t < t_0, \\
     &\tfrac\pi{ \sin (a \pi) }
     &&\text{if } t > t_0,
   \end{aligned} \right. 
  && 0 < t < 1,
  \\ \label{S:id}
  S \bigg( \Big[ \frac{1-t}t \Big]^{ \tfrac{1-a}2 } \bigg)
  &= \frac\pi{ \cos ( \tfrac{a\pi}2 ) } \bigg( 1 - \Big[ \frac{1-t}t \Big]^{ \tfrac{1-a}2 } \sin ( \tfrac{a\pi}2 ) \bigg),
  && 0 < t < 1,
\end{align}
which can be found, e.g., in \cite[(2.47)]{EstradaKanwal00} and \cite[(12A.19)]{King09II} respectively. Here and in the following, we abuse notation by writing $S(f(t))$ instead of $(Sf)(t)$ whenever convenient.

\begin{prop} [Finite Hilbert transform on $L^p$] \label{prop:props:S} 
Let $f \in L^p(0,1)$ and $g \in L^q(0,1)$ with $\tfrac1p + \tfrac1q = 1$ and $p > 1$. Then
\begin{enumerate}[label=(\roman*)]
  \item \label{prop:props:S:BLO} \cite[\S 4.20]{King09I}: $S$ is a bounded linear operator from $L^p(0,1)$ to itself;
  \item \label{prop:props:S:log} \cite[(4.182)]{King09I}: $\displaystyle (Sf)(t) = \frac d{dt} \int_0^1 \log|t-s| f(s) \, ds$ in the weak sense;
  \item \label{prop:props:S:Fubini} \cite[\S 11.10.8]{King09I}: $\displaystyle \int_0^1 f Sg = - \int_0^1 g Sf$;
  \item \label{prop:props:S:pol} \cite[(11.215)]{King09I}: $S(t f(t)) = t S(f(t)) - \int_0^1 f$;
  \item \label{prop:props:S:der} \cite[(11.223)]{King09I}: If $f \in W^{k,p}(0,1)$, then 
  $$(Sf)^{(k)} (t) =  \big( S f^{(k)} \big) (t) + \sum_{\ell = 0}^{k-1} (k - \ell - 1)! \bigg( \frac{f^{(\ell)} (1)}{ (1-t)^{k-\ell} } - \frac{f^{(\ell)} (0)}{ (-t)^{k-\ell} } \bigg) \qquad \text{for any } k \in \N_+.$$
\end{enumerate}
\end{prop}

In preparation for Proposition \ref{prop:props:S:Hol}, we recall some basic properties of H\"older spaces. For $0 < \alpha < 1$ let $C^\alpha ([0,1])$ be the usual space of H\"older continuous functions, and let 
\begin{align*}
  C_0^\alpha ([0,1]) &:= \big\{ f \in C^\alpha ([0,1]) : f(0) = f(1) = 0 \big\}, \\
  C^{k,\alpha} ([0,1]) &:= \big\{ f \in C^k ([0,1]) : f^{(k)} \in C^\alpha ([0,1]) \big\},
\end{align*}
where $k \in \N_+$. We omit the proof of the following proposition.

\begin{prop} [H\"older continuous functions: expansion at endpoint] \label{prop:Hol:0}
Let $0 < \alpha < 1$ and $f \in C^\alpha ([0,1])$ with $f(0) = 0$. Then
\begin{enumerate}[label=(\roman*)]
  \item \label{prop:Hol:0:prod} $t \mapsto t^\beta f(t) \in C^\alpha ([0,1])$ for any $0 < \beta$;
  \item \label{prop:Hol:0:quot} $t \mapsto f(t)/t^\beta \in C^{\alpha - \beta} ([0,1])$  for any $0 < \beta < \alpha$;
  \item \label{prop:Hol:0:der1} if $f \in C^{1,\alpha} ([0,1])$ and $|f(t)| \leq C t^{1 + \alpha}$, then $t \mapsto f(t)/t \in C^\alpha ([0,1])$;
  \item \label{prop:Hol:0:der} if $f \in C^{1,\alpha} ([0,1])$ and $|f(t)| \leq C t^{1 + \alpha + \beta}$ for some $0 < \beta < 1 - \alpha$, then $t \mapsto f(t)/t^\beta \in C^{1,\alpha} ([0,1])$.
\end{enumerate}
\end{prop}

Next we make use of the weighted H\"older space
\begin{equation} \label{for:defn:wHolS}
  C^\alpha_0 (\varphi^\beta) := \{ f : (0,1) \to \R \ | \ \varphi^\beta f \in C^\alpha_0 ([0,1]) \},
\end{equation}
where $\beta > 0$, $\varphi(t) := t(1-t)$ and $\varphi^\beta(t) := \varphi(t)^\beta$. In \cite[\S 1.6]{GohbergKrupnik92I} it is proven that $C^\alpha_0 (\varphi^\beta)$ is a Banach space with norm $\|f\|_{C^\alpha_0 (\varphi^\beta)} := \| \varphi^\beta f \|_{C^\alpha([0,1])}$. For later use, we further set
\begin{equation} \label{phia}
  \phi_a (t) := \varphi^{\tfrac{1-a}2} (t) = [t(1-t)]^{\tfrac{1-a}2}.
\end{equation} 

\begin{prop} [Finite Hilbert transform on $C^\alpha$] \label{prop:props:S:Hol}
Let $0 < \alpha < 1$. Then
\begin{enumerate}[label=(\roman*)]
  \item \label{prop:props:S:Hol:bdd} \cite[Chap.~1, Thm.~6.2]{GohbergKrupnik92I} For any $\alpha < \beta < 1+\alpha$, the operator $S$ is bounded on $C^\alpha_0 (\varphi^\beta)$;
  \item \label{prop:props:S:Hol:cp} \cite[Chap.~1, Thm.~6.3]{GohbergKrupnik92I} For any $f \in C^\alpha([0,1])$, the operator $g \mapsto S(fg) - f Sg$ is compact on $C^\alpha([0,1])$.
  \item \label{prop:props:S:Hol:bdy:0} Let $a < 1$ be such that $\alpha < \frac{1+a}2$. Let $f \in C^\alpha([0,1])$ with $f(0) = 0$. Then, 
  \begin{equation} \label{pSpf:0}
    \frac1{\phi_a(t)} S (\phi_a f)(t) = \frac{A + Bt}{\phi_a(t)} + R(t), 
  \end{equation}
  where $A, B \in \R$ and $R \in C^\alpha([0,t_2])$ with $R(0) = 0$ for any $0 < t_2 < 1$. 
  \item \label{prop:props:S:Hol:bdy:1} If, in addition to the setting in \ref{prop:props:S:Hol:bdy:0}, $f \in C^{1,\alpha}([0,1])$ with $f'(0) = 0$, then $B$ and $R$ in \eqref{pSpf:0} can be chosen such that $R \in C^{1,\alpha}([0,t_2])$ with $R(0) = R'(0) = 0$ for any $0 < t_2 < 1$. 
\end{enumerate}
\end{prop}

\begin{proof}
First we prove \ref{prop:props:S:Hol:bdy:0}. We observe that $\psi (t) := \phi_a(t) f(t) / t$ satisfies
\begin{equation} \label{pf:psi:regy}
  \psi \in C^\alpha_0 \Big( \varphi^{\tfrac{1+a}2} \Big) \cap L^{\tfrac2{1+a}} (0,1).
\end{equation}
Hence, Proposition \ref{prop:props:S}.\ref{prop:props:S:pol} applies. This yields
\begin{equation*}
  S (\phi_a f) (t)
  = t^{\tfrac{1-a}2} t^{\tfrac{1+a}2} S \psi (t) - \int_0^1 \psi(s) \, ds.
\end{equation*}
The second term corresponds to the constant $A$ in \eqref{pSpf:0}. For the first term, we obtain from \eqref{pf:psi:regy} and Proposition \ref{prop:props:S:Hol}.\ref{prop:props:S:Hol:bdd} that $S \psi \in C^\alpha_0 ( \varphi^{(1+a)/2} )$, and thus $\tilde \psi (t) := t^{(1+a)/2} S \psi (t)$ satisfies $\tilde \psi \in C^\alpha([0,t_2])$ with $\tilde \psi(0) = 0$ for any
$0 < t_2 < 1$. Setting $R(t) := \tilde \psi(t) / (1-t)^{(1-a)/2}$ and $B = 0$, we obtain \eqref{pSpf:0}. 

Next we prove \ref{prop:props:S:Hol:bdy:1}. Since $\phi_a f \in W^{1,p}(0,1)$ for some $p > 1$, we obtain from Proposition \ref{prop:props:S}.\ref{prop:props:S:der} that 
\begin{equation*} 
  (S(\phi_a f))' = S((\phi_a f)') = S(\phi_a' f) + S(\phi_a f').
\end{equation*}
Using \ref{prop:props:S:Hol:bdy:0}, we write the second term as $S(\phi_a f') = B_1 + R_1 \phi_a$ with $B_1$ and $R_1$ as specified in \ref{prop:props:S:Hol:bdy:0}. For the first term, we compute $\phi_a'$ to rewrite it as
\begin{equation} \label{pf:Spapf}
  S(\phi_a' f)
  = \frac{1-a}2 \bigg[ S \Big( \frac{ \phi_a(t) f(t)}t \Big) - S \Big( \frac{ \phi_a(t) f(t)}{1-t} \Big) \bigg].
\end{equation}

In the following, we expand both terms in the right-hand side of \eqref{pf:Spapf} separately. For the first term, we observe from Proposition \ref{prop:Hol:0}.\ref{prop:Hol:0:der1} that, by the given properties of $f$, the function $t \mapsto f(t)/t$ satisfies the hypotheses of \ref{prop:props:S:Hol:bdy:0}. Applying \ref{prop:props:S:Hol:bdy:0}, we write the first term as $S ( \phi_a(t) f(t) / t ) = B_2 + R_2(t) \phi_a(t)$ with $B_2$ and $R_2$ as specified in \ref{prop:props:S:Hol:bdy:0}.

We continue by expanding the second term in \eqref{pf:Spapf}. Using Proposition \ref{prop:props:S}.\ref{prop:props:S:pol}, we find
\begin{equation*}
  S \Big( \frac{ \phi_a(t) f(t)}{1-t} \Big)
  = \frac1{1-t} \int_0^1 \frac{ \phi_a(s) f(s)}{1-s} \, ds + \frac{S ( \phi_a f) (t)}{1-t}.
\end{equation*}
The first term is regular on $[0,1)$. We rewrite the second term by expanding both the enumerator and denominator around $t=0$ up to leading order. To expand the enumerator, we apply \ref{prop:props:S:Hol:bdy:0}. Combining both expansions, we obtain a constant $B_3 \in \R$ and an $R_3 \in \cap_{0 < s < 1} C^\alpha([0,s])$ with $R_3(0) = 0$ such that 
\begin{equation*}
  S \Big( \frac{ \phi_a(t) f(t)}{1-t} \Big)
  = B_3 + R_3(t) \phi_a(t).
\end{equation*}

Next we take $t_2 \in (0,1)$ arbitrary. Collecting all expansions above, we find a constant $B$ and an $R_4 \in C^\alpha([0,t_2])$ with $R_4(0) = 0$ such that 
$
  (S(\phi_a f))' = B + R_4 \phi_a
$.
Integrating from $0$ to $t \in (0,t_2]$, we obtain
\begin{equation*}
  S(\phi_a f)(t) = S(\phi_a f)(0) + Bt + \int_0^t R_4(s) \phi_a (s) \, ds,
\end{equation*}
where we have applied \ref{prop:props:S:Hol:bdy:0} to guarantee that $S(\phi_a f)(0) \in \R$ is well-defined. Since $R_4 \in C^\alpha([0,t_2])$ with $R_4(0) = 0$, we have $|R_4(s) \phi_a (s)| \leq C s^{\alpha + (1-a)/2}$, and thus, by Proposition \ref{prop:Hol:0}.\ref{prop:Hol:0:prod}, $\psi(t) := \int_0^t R_4(s) \phi_a (s) \, ds$ satisfies $\psi \in C^{1,\alpha}([0,t_2])$ and $|\psi(t)| \leq C t^{1 + \alpha + (1-a)/2}$. Hence, by Proposition \ref{prop:Hol:0}.\ref{prop:Hol:0:der}, $R := \psi / \phi_a \in C^{1,\alpha}([0,t_2])$ satisfies $0 = R(0) = R'(0)$.
\end{proof}

The following proposition is a combination of two theorems; the well-posedness statement is a particular case of \cite[Thm.~III]{Widom60}, while the explicit solutions are taken from \cite[(12.150),(12.153),(11.63)]{King09I}.

\begin{prop} [A Cauchy integral equation] \label{prop:IE}
If $0 < a < 1$, then for any $\frac2{1+a} < p < \frac2{1-a}$, it holds for all $f \in L^p (0,1)$ that the integral equation 
\begin{equation} \label{for:prop:IE:IE}
  \pi \tan( \tfrac{a\pi}2 ) u - S u = f 
  \quad \text{a.e.~on } (0,1)
\end{equation}
has a unique solution $u$ in $L^p (0,1)$. 
 The solution is given by
\begin{equation*}
  u(t) 
  = \frac{ \sin(a\pi) }{2 \pi} f (t)
	+ \frac{ \cos^2 (\frac{a\pi}2) }{ \pi^2 } \Big( \frac t{1-t} \Big)^{\tfrac{1-a}2} S \bigg( \Big( \frac{1-t}t \Big)^{\tfrac{1-a}2} f(t) \bigg)
  \quad \text{for a.e.~} 0 < t < 1.
\end{equation*} 

If $a = 0$, then for any $1<p<2$ and any $f \in L^p (0,1)$ all solutions in $L^p (0,1)$ to the equation $-Su = f$ (i.e., \eqref{for:prop:IE:IE} with $a = 0$) are given by
\begin{equation*}
  u
  =  \frac1{\pi^2 \phi_0} \big( S(\phi_0 f) + C \big)
  \quad \text{a.e.~on } (0,1),
\end{equation*} 
where $C \in \R$ is a parameter and $\phi_0 (t) = \sqrt{t (1-t)}$ as in \eqref{phia}.
\end{prop}

Since we are using both Sobolev spaces and H\"older spaces, Morrey's inequality is convenient to deduce H\"older continuity from Sobolev regularity:

\begin{prop} [Morrey's inequality {\cite[Thm.~9.12]{Brezis10}}] \label{prop:Morrey}
Let $k \in \N$ and $1 < p < \infty$. Then $W^{k+1,p}(0,1)$ is continuously embedded in $C^{k, 1 - 1/p} ([0,1])$.
\end{prop}  

In the last part of this section, we introduce the Riemann-Liouville fractional derivative and the related fractional integral. In preparation for this, let $a \in (0,1)$, $\alpha, \beta > 0$, and $\Gamma (t) := \int_0^\infty s^{\alpha-1} e^{-s} \, ds$ be the usual $\Gamma$-function. We recall the relations
\begin{align}  \label{for:id:Gamma}
  \Gamma(a) \Gamma(1-a) &= \frac\pi{ \sin(a \pi) },
  && \\
\label{Zwil}
 \int_0^{t} s^{ \alpha - 1 } (t-s)^{ \beta - 1 } \, ds
 &= \frac{ \Gamma(\alpha) \Gamma(\beta) }{ \Gamma(\alpha + \beta) } t^{\alpha + \beta - 1},
 && t > 0,
\end{align}
which can be found in \cite[\S 8.334.3 and \S 3.191.1]{GradshteynRyzhik14}. 

For any $f : [0,\infty) \to \R$ smooth enough, the Riemann-Liouville fractional integral and the Riemann-Liouville fractional derivative are defined by
\begin{subequations} \label{IaDa}
  \begin{align} 
  (I^a f)(t) 
  &:= \frac1{\Gamma(a)} \int_0^t \frac{f(s)}{(t-s)^{{1-a}}} \, ds
  & t > 0, \\ \label{IaDa:Da}
  (D^a f)(t) 
  &:= \frac1{\Gamma(1-a)} \frac d{dt} \int_0^t \frac{f(s)}{(t-s)^a} \, ds
    = \frac d{dt} (I^{1-a} f)(t) 
  & t > 0,
\end{align}
\end{subequations}
respectively. We further set $D^1 f := f'$. For convenience, we extend the definition of $D^a$ to any starting value $t_1 \in \R$ by
\begin{equation} \label{Dat1}
 (D_{t_1}^a f)(t) 
  := \frac1{\Gamma(1-a)} \frac d{dt} \int_{t_1}^t \frac{f(s)}{(t-s)^a} \, ds,
  \quad t > t_1. 
\end{equation} 
We also recall the following formula for the fractional derivative of polynomials, which is a direct consequence of \eqref{Zwil}:
\begin{equation} \label{Dat2}
  D^a \bigg( \sum_{k=0}^{\ell} \frac{ b_k }{ k! } t^k \bigg) 
  = \sum_{k=0}^{\ell} \frac{ b_k }{ \Gamma(k+1-a) } t^{k - a}.
\end{equation} 

\begin{prop} [] \label{prop:reg:Da} 
Let $0<a<1$ and $\ell \in \N$. Let $f \in C^{\ell}([0,1])$ with $q_{\ell-1}$ its $(\ell-1)$-th order Taylor polynomial at $0$ ($q_{-1} := 0$), and set $R_\ell := f - q_{\ell-1}$.
\begin{enumerate}[label=(\roman*)]
  \item \label{prop:reg:Da:Hol} If $\ell \geq 1$, then $D^a R_\ell \in C^{\ell-1,1-a}([0,1])$ and there exists $C > 0$ such that $|(D^a R_\ell)(t)| \leq C t^{\ell - a}$ for all $t \in [0,1]$;
  \item \label{prop:reg:Da:Sob} {\cite[\S 2]{SamkoKilbasMarichev93}}: If $f \in W^{\ell + 1,1}(0,1)$, then $I^a R_\ell \in W^{\ell + 1,1}(0,1)$ and $D^a R_\ell \in W^{\ell,p}(0,1)$ for all $1 \leq p < \frac1a$.
\end{enumerate}
\end{prop} 

\begin{proof}
Proposition \ref{prop:reg:Da}.\ref{prop:reg:Da:Hol} is a corollary of \cite[Thm.~3.2]{SamkoKilbasMarichev93}, which states that $I^{1-a} R_\ell \in C^{\ell,1-a}([0,1])$. Recalling from \eqref{IaDa:Da} that $D^a = \frac d{dt} I^{1-a}$, we obtain $D^a R_\ell \in C^{\ell-1,1-a}([0,1])$.
To prove the bound on $|(D^a R_\ell)(t)|$, we obtain from \eqref{IaDa:Da} that
\[
(D^a R_\ell)(t) 
  = \frac1{\Gamma(1-a)} \frac d{dt} \int_0^t \frac{R_\ell(t-r)}{r^a} \, dr 
    = \bigg[ \frac{R_\ell(0)}{\Gamma(1-a) t^a} + \int_0^t \frac{R_\ell'(t-r)}{r^a} \, dr \bigg].
\]
Differentiating $\ell - 1$ times and noting that $R_\ell^{(k)}(0) = 0$ for all $k = 0, \ldots, \ell - 1$, we obtain that $(D^a R_\ell)^{(k)}(0) = 0$ for all $k = 0, \ldots, \ell - 1$ and 
$
(D^a R_\ell)^{(\ell-1)}(t) 
  = \int_0^t R_\ell^{(\ell)} (t-r) r^{-a} \, dr
$. Since $R_\ell^{(\ell)}$ is continuous, we obtain $|(D^a R_\ell)^{(\ell-1)}(t)| \leq C t^{1-a}$. We conclude by using $(D^a R_\ell)^{(k)}(0) = 0$ for all $k = 0, \ldots, \ell - 1$.
\end{proof}

\begin{prop} [{\cite{SamkoKilbasMarichev93}}; Thm.~2.4 and Thm.~2.6] \label{prop:FDE}
Let $0<a<1$ and $1 \leq p < \infty$. Then $I^a$ is a bounded linear operator from $L^p(0,1)$ to itself. Moreover, for any $f \in W^{1,1}(0,1)$,
\begin{equation*}
  I^a D^a f = f
  \quad \text{a.e.~on } (0,1).
\end{equation*}
\end{prop}

\section{The Hilbert space $H_V (0,1)$ induced by $(\cdot, \cdot)_V$}
\label{ssec:Vip}

In this section we fix any $0 \leq a < 1$ and any potential $\Vreg$ which satisfies \eqref{for:V:props}, and set $V = V_a + \Vreg$. We prove that the bilinear form $(\cdot, \cdot)_V$ (see \eqref{for:defn:Vip}) defines an inner product, and characterise the Hilbert space $H_V(0,1)$ which it generates as the closure of $L^2(0,1)$ (the precise definition is given in Corollary \ref{cor:props:HV}). The space $H_V(0,1)$ provides a convenient functional framework for Problem \ref{prob:wSIE}. Moreover, we establish several properties of the linear operator $\rho \mapsto V * \rho$ (see Lemmas \ref{l:convo:consy} and \ref{lem:Vastxi:zero:means:xi:zero}). 

Since we will make use of the Fourier transform, we rely on \eqref{for:V:props} to extend $\Vreg$ (and $V = V_a + \Vreg$) to $\R$ in the following manner:
\begin{equation} \label{for:V:props:R}
\begin{gathered} 
  \Vreg = V - V_a \in \left\{ \begin{aligned}
    &W^{2,1}_{\text{loc}} (\R)
    &&\text{if } 0<a<1, \\
    &W^{2,\plog}_{\text{loc}} (\R)
    &&\text{if } a = 0,
  \end{aligned} \right. \\
  \exists \, b > 1 : \supp V \subset [-b,b], 
  \quad V \text{ even},
  \quad V''(t) \geq 0 \ \text{for a.e.~} t \in \R.
\end{gathered}  
\end{equation}
This extension induces the following properties on $V$, which we prove in Appendix \ref{app:pf:lem:props:V}:

\begin{lem} [Properties of the extended $V$] \label{lem:props:V}
The potentials $\Vreg$ and $V$ satisfy
\begin{enumerate}[label=(\roman*)]
  \item \label{lem:props:V:Vreg:regy} $\Vreg \in C^1 (\R)$ and $\left\{ \begin{aligned}
    &\Vreg' \in W^{1,1} (\R) 
    &&\text{if } 0 < a < 1, \\
    &\Vreg' \in W^{1, \plog} (\R), \: \Vreg'' \in L^1(\R)
    &&\text{if } a = 0;
  \end{aligned} \right.$  
  \item \label{lem:props:V:V_k} There exist $(V_k)_{k \in \N_+} \subset C_c(\R)$ with $V_k$ even, $0 \leq V_k \leq V$, $0 \leq V_k'' \leq V''$ and $V_k \uparrow V$ pointwise on $\R \setminus \{0\}$ as $k \to \infty$;
  \item \label{lem:props:V:Vhat:bounds} $\displaystyle \exists \, C \geq c > 0 \ \forall \, \omega \in \R :  c(1 + \omega^2)^{ -\tfrac{1-a}2} 
  \leq \widehat V (\omega) 
  \leq C(1 + \omega^2)^{ -\tfrac{1-a}2}$.
\end{enumerate}

\end{lem}

It is straight-forward to extend $(\cdot, \cdot)_V$ to $L^2(\R)$ by
\begin{equation} \label{for:defn:Vip:R}
  (f, g)_V = \int_{\R} (V  * f) g
\end{equation}
for $f,g \in L^2(\R{})$. 

\begin{lem} \label{l:Vip:L2}
The bilinear form $(\cdot, \cdot)_V$ in \eqref{for:defn:Vip:R} is an inner product on $L^2 (\R)$.
\end{lem}

\begin{proof}
Since $V \in L^1(\R)$, it follows from Young's inequality that $f \mapsto V * f$ defines a bounded linear operator on $L^2(\R)$ to itself, and thus \eqref{for:defn:Vip:R} is well-defined. Except for positivity, it is readily checked that \eqref{for:defn:Vip:R} satisfies all other properties of an inner product. To show positivity, i.e., $\|f\|_V \geq 0$ and $f = 0 \Leftrightarrow \|f\|_V = 0$, we use $\widehat V > 0$ (Lemma \ref{lem:props:V}.\ref{lem:props:V:Vhat:bounds}) to estimate
\begin{equation} \label{for:Vnorm:Fversion}
  \| f \|_V^2 = \int_\R (V * f) f = \int_\R \widehat V \big| \widehat f \big|^2 \geq 0
\end{equation}
and to deduce that $\|f\|_V = 0 \Leftrightarrow \widehat f = 0 \Leftrightarrow f = 0$.
\end{proof}


Lemma \ref{l:Vip:L2} shows that
\begin{equation} \label{for:defn:HVR}
  H_V(\R) := \overline{ L^2 (\R) }^{ \| \cdot \|_V }
\end{equation}
is a Hilbert space. Proposition \ref{prop:HV} characterises $H_V(\R)$ as a fractional Sobolev space defined in \eqref{Hs}, and Corollary \ref{cor:props:HV} lists further properties of $H_V(\R)$.

\begin{prop} \label{prop:HV}
The inner products on $H_V(\R)$ and $H^{- (1-a)/2} (\R)$ are equivalent.
\end{prop}

\begin{proof} 
It is enough to show that the norms on $H_V(\R)$ and $H^{- (1-a)/2} (\R)$ are equivalent for all $f \in L^2(\R)$. With $c, C > 0$ as in Lemma \ref{lem:props:V}.\ref{lem:props:V:Vhat:bounds} and the characterisation \eqref{for:Vnorm:Fversion}, we show equivalence of the norms by
\begin{multline*} 
  c\| f \|_{ H^{-(1-a)/2} (\R) }^2
  = \int c (1 + \omega^2)^{- { \tfrac{1-a}2}} \big| \widehat f(\omega) \big|^2 \, d\omega
  \leq \int \widehat V (\omega) \big| \widehat f(\omega) \big|^2 \, d\omega
  = \|f\|_V^2 \\
  \leq \int C (1 + \omega^2)^{- { \tfrac{1-a}2}} \big| \widehat f(\omega) \big|^2 \, d\omega
  = C \| f \|_{ H^{-(1-a)/2} (\R) }^2. \qedhere
\end{multline*}
\end{proof}

\begin{cor}[Properties of $H_V(\R)$] \label{cor:props:HV}
For $\zeta, \eta \in H_V (\R)$, we have
\begin{equation} \label{for:Vip:on:HV}
  ( \zeta, \eta )_V = \int_\R (T \zeta) ( T \eta ),
\end{equation}
where $T$ is given by
\begin{equation} \label{for:defn:v}
  T \zeta := \mathcal F^{-1} (v \widehat \zeta ),
  \quad \text{where} \quad
  v := \sqrt{ \widehat V }.
\end{equation}
The linear operator $T$ is an isometry from $H_V(\R)$ to $L^2(\R)$, which is self-adjoint in $L^2(\R)$.
Moreover, for any bounded interval $(t_1, t_2)$, the following space is a Hilbert space
\begin{equation} \label{for:defn:HV}
  H_V(t_1, t_2) := \{ \zeta \in H_V (\R) : \supp \zeta \subset [t_1, t_2] \},
\end{equation}
which is characterised by
\begin{equation} \label{for:HV:densy:Ccinf}
  H_V(t_1, t_2) = \overline{ C_c^\infty(t_1, t_2) }^{\| \cdot \|_V}.
\end{equation}
\end{cor}

\begin{rem} \label{r:T2:eqs:Vast}
The operator $T$ is also examined in \cite{GarroniVanMeursPeletierScardia16}, but for different assumptions on $V$. We refer to \cite[Lem.~A.2]{GarroniVanMeursPeletierScardia16} in several steps of the proof of Corollary \ref{cor:props:HV}. 

From \eqref{for:defn:v} we observe that $\mathcal F (T^2 \zeta) = \widehat V \widehat \zeta$, which motivates us to interpret the operator $T^2$ as `convolution with $V$'. In Lemma \ref{l:convo:consy} we prove that this interpretation is consistent with the extension of the classical convolution which we introduce in \eqref{convo:Vrho}.

However, we refrain from interpreting $T$ as `convolution with $\mathcal F^{-1} v$'. The reason for this is as follows. From Lemma \ref{lem:props:V}.\ref{lem:props:V:Vhat:bounds} we obtain $v \notin L^1(\R) + L^2(\R)$, and thus we cannot apply the theory of the Fourier transform on $L^1(\R) + L^2(\R)$ to identify $\mathcal F^{-1} v$ as a function. 
\end{rem}

\begin{rem}
We note that $H_V (0,1)$ is independent of the extension of $V$ from $[-1,1]$ to $\R$. Indeed, by using \eqref{for:HV:densy:Ccinf} to approximate $\zeta \in H_V (0,1)$ by $(\varphi_n) \subset C_c^\infty (0,1)$, we observe from \eqref{for:defn:Vip} that $\| \varphi_n \|_V$ is independent of the extension of $V$ for all $n$, and hence $\| \zeta \|_V$ is also independent of the extension of $V$.
\end{rem}

\begin{proof}[Proof of Corollary \ref{cor:props:HV}]
We start by proving the asserted properties of $T$. By Proposition \ref{prop:HV} the Fourier transform is well-defined on $H_V(\R)$. We recall from \cite[Lem.~A.2]{GarroniVanMeursPeletierScardia16} that $T\zeta$ is real-valued for $\zeta \in H_V (\R)$ (this follows from \eqref{for:defn:v} by the Hermitian symmetry of the Fourier transform and $v$ being real-valued and even). It is readily seen that $T$ is self-adjoint in $L^2(\R)$ by computing it in Fourier space.

By \eqref{for:Vnorm:Fversion} and Proposition \ref{prop:HV} we characterise
\begin{equation*}
  H_V(\R) 
  \cong \bigg\{ \zeta \in \mathcal S' (\R) : \| \zeta \|_V^2 = \int_\R \widehat V \big| \widehat \zeta \big|^2 < \infty \bigg\}.
\end{equation*}
Hence, for any $\zeta \in H_V(\R)$, $\widehat \zeta \in L^2 (\widehat V)$ (weighted $L^2$-space) can be treated as a complex-valued function. Moreover
\begin{equation*}
  \| \zeta \|_V^2
  = \int_\R \widehat V \big| \widehat{\zeta} \big|^2 
  = \int_\R \big| v \widehat \zeta \big|^2
  = \int_\R \big| \widehat{ T \zeta } \big|
  = \| T \zeta \|_{L^2(\R)}^2,
\end{equation*}
which shows that $T$ is isometric from $H_V(\R)$ to $L^2(\R)$. Thus, the right-hand side of \eqref{for:Vip:on:HV} is well-defined, and from Lemma \ref{lem:props:V}.\ref{lem:props:V:Vhat:bounds} we easily see that $T^{-1} f = \mathcal F^{-1} (\widehat f / v)$.

Next we show that $H_V(t_1, t_2)$ is a Hilbert space by showing that it is a closed subspace of $H_V(\R)$. It is trivial that $H_V(t_1, t_2)$ is a subspace, and closedness follows from
\begin{equation*}
  \{ \zeta \in \mathcal S' (\R) : \supp \zeta \subset [t_1, t_2] \}
\end{equation*}
being closed in $\mathcal S' (\R)$, and $\mathcal S' (\R) \supset H^{- (1-a)/2} (\R) \supset H_V(t_1, t_2)$.

Next we prove \eqref{for:HV:densy:Ccinf} by using standard approximation arguments. Without loss of generality we set $(t_1, t_2) = (0,1)$. Let 
\begin{equation*}
  H_{V,c}(0,1) := \{ \zeta \in H_V (0,1) : \supp \zeta \subset (0,1) \}.
\end{equation*}
We show that
\begin{subequations}
\begin{align} \label{fp:dens:Cci:HVc}
  \forall \, \zeta \in H_{V,c}(0,1)\
  \exists \, (f_n) \subset C_c^\infty(0,1)
  &: \| f_n - \zeta \|_V \xto{n \to \infty} 0, \quad \text{and}\\ \label{fp:dens:HVc:HV}
  \forall \, \zeta \in H_V(0,1) \
  \exists \, (\zeta_n) \subset H_{V,c}(0,1)
  &: \| \zeta_n - \zeta \|_V \xto{n \to \infty} 0, 
\end{align}
\end{subequations}
from which \eqref{for:HV:densy:Ccinf} follows by a diagonal argument.

To prove \eqref{fp:dens:Cci:HVc}, we set $\eta_n$ as the usual mollifier, and take $f_n := \eta_n * \zeta$. Since $\zeta$ is a distribution with compact support in $(0,1)$, it follows from basic theory on distributions and the Fourier transform (see, e.g., \cite[Thm.~2.7 and Lem.~2.10]{RichardsYoun07}) that $f_n \in C_c^\infty(0,1)$ for all $n \in \N$ large enough. Moreover, $\widehat{\eta_n} \to 1$ as $n \to \infty$ uniformly on bounded sets, and $\sup_{\omega \in \R} | \widehat{\eta_n} (\omega) | = \sup_{\omega \in \R} | \widehat{\eta_1} (\omega) | < \infty$. Using these properties, we compute for any $R > 0$
\begin{align} \notag
  &\| f_n - \zeta \|_V^2
  = \int_{\R} \widehat V \Big| \widehat{\eta_n * \zeta} - \widehat \zeta \Big|^2
  = \int_{\R} \widehat V \big| \widehat{\eta_n} - 1 \big|^2 \big| \widehat \zeta \big|^2 \\\notag
  &\leq \Big( \max_{[-R,R]} \big| \widehat{\eta_n} - 1 \big|^2 \Big) \int_{-R}^R \widehat V \big| \widehat \zeta \big|^2 + \sup_{\omega \in \R} \big( \widehat{\eta_n} (\omega) - 1 \big)^2 \int_{[-R,R]^c} \widehat V \big| \widehat \zeta \big|^2 \\\label{fp:dens:Cci:HVc:est}
  &\leq \Big( \max_{[-R,R]} \big| \widehat{\eta_n} - 1 \big|^2 \Big) \| \zeta \|_V^2 + C \int_{[-R,R]^c} \widehat V \big| \widehat \zeta \big|^2.
\end{align}
Since $\zeta \in H_{V,c}(0,1) \subset H_V(\R)$, it holds that $\widehat V \big| \widehat \zeta \big|^2 \in L^1(\R)$, and thus the second term in the right-hand side of \eqref{fp:dens:Cci:HVc:est} converges to $0$ as $R \to \infty$. We conclude \eqref{fp:dens:Cci:HVc} by first passing to the limit $n \to \infty$ in \eqref{fp:dens:Cci:HVc:est}, and then $R \to \infty$.

In preparation for proving \eqref{fp:dens:HVc:HV}, we introduce the sequence of dilation operators $\tau_\varepsilon : H_V(\R) \to H_V(\R)$ parametrised by $0 \leq \varepsilon < \tfrac12$ and given by
\begin{equation*} 
  \tau_\varepsilon \zeta (\omega) 
  = \mathcal F^{-1} \big( (1 - 2 \varepsilon) e^{- 2 \pi i \varepsilon \omega } \widehat \zeta \big( (1 - 2 \varepsilon) \omega \big) \big),
\end{equation*}
where we have used the characterisation $\widehat \zeta \in L^2(\widehat V)$. By construction, 
\begin{equation} \label{fp:tau:dil}
  \supp \zeta \subset [0,1]
  \quad \Longrightarrow \quad 
  \supp( \tau_\varepsilon \zeta ) \subset [\varepsilon, 1 - \varepsilon],  
\end{equation} 
and, for $f \in C_c (\R)$, it is easy to see that
\begin{equation} \label{fp:tau:conv}
  \| \tau_\varepsilon f - f \|_{L^2(\R)}^2 
  \xto{ \varepsilon \to 0 } 0.
\end{equation}
We claim that $\tau_\varepsilon$ is a bounded linear operators from $H^s (\R)$ to itself with operator norm bounded by $1$ for all $0 \leq s \leq \tfrac12$. Indeed
\begin{align*} \notag
  \big\| \tau_\varepsilon f \big\|_{H^s (\R)}^2
  &= \int_{\R} (1 + \omega^2)^s \Big| \widehat{ \tau_\varepsilon f } \Big|^2 \, d\omega
  = (1 - 2 \varepsilon)^2 \int_{\R} (1 + \omega^2)^s \Big| \widehat f \big( (1 - 2 \varepsilon) \omega \big) \Big|^2 \, d\omega \\ \notag
  &= (1 - 2 \varepsilon) \int_{\R} \Big( 1 + \frac{k^2}{(1 - 2 \varepsilon )^2} \Big)^s \big| \widehat f ( k ) \big|^2 \, dk \\ 
  &\leq (1 - 2 \varepsilon)^{1 - 2s} \int_{\R} (1 + k^2)^s \big| \widehat f ( k ) \big|^2 \, dk
  \leq \| f \|_{H^s (\R)}^2,
\end{align*}
which proves the claim. Then, by Proposition \ref{prop:HV}, we conclude that the norm of $\tau_\varepsilon$ as an operators from $H_V (\R)$ to itself is bounded by some $\varepsilon$-independent constant.

Finally, we prove \eqref{fp:dens:HVc:HV}. Let $\zeta \in H_V (0,1)$, and set $\zeta_\varepsilon := \tau_\varepsilon \zeta$ for $0 < \varepsilon < \tfrac12$. By \eqref{fp:tau:dil}, $\zeta_\varepsilon \in H_{V,c} (0,1)$ for all $\varepsilon > 0$. Regarding the convergence, we take any $\delta > 0$, and observe from $\widehat \zeta \in L^2(\widehat V)$ that there exists $\varphi \in \mathcal S (\R)$ with $\widehat \varphi \in C_c^\infty(\R)$ such that 
\begin{equation*}
  \| \zeta - \varphi \|_V^2 
  = \int_{\R} \widehat V \big| \widehat \zeta - \widehat \varphi \big|^2 
  < \delta^2.
\end{equation*}
Then, we estimate
\begin{equation*}
  \| \tau_\varepsilon \zeta - \zeta \|_V
  \leq \| \tau_\varepsilon (\zeta - \varphi) \|_V + \| \tau_\varepsilon \varphi - \varphi \|_V + \| \varphi - \zeta \|_V
  \leq C\| \varphi - \zeta \|_V + \| \tau_\varepsilon \varphi - \varphi \|_V.
\end{equation*}
The first term is bounded by $C \delta$. For the second term, we use $\| \cdot \|_V \leq C \| \cdot \|_{L^2(\R)}$ and \eqref{fp:tau:conv} to show that it is abitrarily small in $\varepsilon$. Since $\delta > 0$ is abitrary, we conclude \eqref{fp:dens:HVc:HV}, which completes the proof of \eqref{for:HV:densy:Ccinf}.
\end{proof}

Using \eqref{for:defn:HV}, we define the integral of $\rho \in H_V (0,1)$ as
\begin{equation} \label{int01rho}
  \int_0^1 \rho 
  := \langle \rho, \varphi \rangle, 
\end{equation}
where $\varphi \in C_c^\infty (\R)$ is any test function which satisfies $\varphi |_{(0,1)} = 1$. Since $\supp \rho \subset [0,1]$, this definition does not dependent on $\varphi$.

\begin{lem} \label{l:convo:consy}
For any $\rho, \mu \in H_V(0,1) \cap \mathcal P ([0,1])$ it holds that
\begin{gather} \label{Vrho:is:Vrhot}
  V * \rho = T^2 \rho
  \quad \text{a.e.~on } (0,1) \text{, and} \\ \label{for:Vnorm:on:HVnP}
  (\mu, \rho)_V
  = \int_0^1 (V * \mu) \, d\rho.
\end{gather}

\end{lem}

\begin{rem} While any $\rho \in H_V (0,1) \cap \mathcal P([0,1])$ has no atoms,  we cannot exclude Cantor parts. Indeed, \cite[Thm.~4.13]{Falconer03} guarantees for all $0 \leq a < 1$ the existence of a Cantor measure $\rho \in \mathcal P([0,1])$ with $\| \rho \|_{H^{- (1-a)/2}(\R)} < \infty$. Hence, we cannot treat the elements of $H_V (0,1) \cap \mathcal P([0,1])$ as functions.
\end{rem}

\begin{proof}[Proof of Lemma \ref{l:convo:consy}]
Let $\rho, \mu \in H_V(0,1) \cap \mathcal P ([0,1])$ be arbitrary. We extend $\rho, \mu$ to $\R$ with zero extension, and extend $V$ to $\R$ as in \eqref{for:V:props:R}. We first introduce regularisations of $V$, $\rho$ and $\mu$. Let $V_k$ as given by Lemma \ref{lem:props:V}.\ref{lem:props:V:V_k}, for any $k \in \N$, and let $\rho_\varepsilon := \eta_\varepsilon * \rho \in \mathcal S (\R) \cap \mathcal P(\R)$ and $\mu_\delta := \eta_\delta * \mu \in \mathcal S (\R) \cap \mathcal P(\R)$, where $\eta_\varepsilon$ is the usual mollifier. It is easy to verify that
\begin{equation*}
   \int_{\R} \varphi \, d\rho_\varepsilon 
   \xto{\varepsilon \to 0} \int_{\R} \varphi \, d\rho
   \quad \text{for all } \varphi \in C_b(\R) 
   \text{, and}
   \quad v \widehat{\rho_\varepsilon}
   \xto{\varepsilon \to 0} v \widehat{\rho}
   \quad \text{in } L^2(\R),
 \end{equation*} 
 where $v$ is defined in \eqref{for:defn:v}. $\mu_\delta$ satisfies analogous convergence properties. As a consequence of Lemma \ref{lem:props:V}.\ref{lem:props:V:V_k} and \eqref{for:V:props:R}, we obtain that $V_k \in C_c(\R)$. Moreover, since both $V_k$ and $V - V_k$ are convex on $(0,\infty)$ and even on $\R$, we obtain from \eqref{Vhat:nonneg} that $0 \leq \widehat{ V_k } \leq \widehat V$. Moreover, applying the Dominated Convergence Theorem, we obtain
 \begin{equation} \label{VmVkhat:conv}
   \big\| \widehat V - \widehat{ V_k } \big\|_\infty
   = \sup_{\omega \in \R} \bigg| \int_\R (V - V_k)(t) \, \cos(2\pi t \omega) \, dt \bigg|
   = \int_\R (V - V_k) 
   \xto{k \to \infty} 0.
 \end{equation}
Since $V_k$, $\rho_\varepsilon$ and $\mu_\delta$ are all continuous and integrable, we obtain from the classical definition of convolution that
\begin{equation} \label{pf:regzd:version}
  V_k * \rho_\varepsilon = \mathcal F^{-1}( \widehat{V_k} \widehat{\rho_\varepsilon})
  \quad \text{on } \R \text{, and} \quad
  \int_{\R} (V_k * \mu_\delta) \rho_\varepsilon
  = \int_{\R} \widehat{ V_k } \widehat{ \mu_\delta } \overline{ \widehat{ \rho_\varepsilon }}
\end{equation}
where the bar denotes complex conjugation.

Next we pass to the limit in \eqref{pf:regzd:version}; first $\delta \to 0$, then $\varepsilon \to 0$, and finally $k \to \infty$. Regarding $V_k * \rho_\varepsilon$, we note that for all $t \in \R$, the mapping $s \mapsto V_k(t-s)$ is bounded and continuous. Hence,
\begin{equation} \label{pf:Vkrhoe:ptws}
  (V_k * \rho_\varepsilon) (t)
  = \int_{\R} V_k(t-s) \, d\rho_\varepsilon(s)
  \xto{\varepsilon \to 0} \int_{\R} V_k(t-s) \, d\rho(s)
  = (V_k * \rho) (t)
  \quad \text{for all } t \in \R.
\end{equation}
Since $V_k * \rho_\varepsilon$ is bounded, we obtain form the Dominated Convergence Theorem that $V_k * \rho_\varepsilon \to V_k * \rho$ in $L^p(0,1)$ as $\varepsilon \to 0$ for any $1 \leq p < \infty$.
Fixing some $p \in [1, 1/a)$, it follows from $V * \rho \in L^p(0,1)$ and the Monotone Convergence Theorem that
\begin{equation*}
  V_k * \rho
  \xto{k \to \infty} V * \rho
  \quad \text{in } L^p(0,1).
\end{equation*}
Regarding $\mathcal F^{-1}( \widehat{V_k} \widehat{\rho_\varepsilon})$, we take any test function $\varphi \in \mathcal S(\R)$, and compute
\begin{align*}
  \int_\R \mathcal F^{-1}( \widehat{V_k} \widehat{\rho_\varepsilon}) \varphi
  = \int_\R \widehat{ V_k } \widehat{ \rho_\varepsilon } \overline{ \widehat{ \varphi }}
  \xto{\varepsilon \to 0} \int_\R \widehat{ V_k } \widehat{ \rho } \overline{ \widehat{ \varphi }}.
\end{align*}
Using $0 \leq \widehat{ V_k } \leq \widehat V$ and the pointwise convergence of $\widehat{ V_k }$ to $\widehat V$, we conclude from the Dominated Convergence Theorem that
\begin{align*}
  \int_\R \widehat{ V_k } \widehat{ \rho } \overline{ \widehat{ \varphi }}
  \xto{k \to \infty}   \int_\R \widehat{ V } \widehat{ \rho } \overline{ \widehat{ \varphi }}
  = \int_\R (T^2 \rho) \varphi.
\end{align*}
Since $\varphi$ is arbitrary, we conclude by the uniqueness of limits that \eqref{Vrho:is:Vrhot} holds.

Similarly, we prove \eqref{for:Vnorm:on:HVnP} starting from \eqref{pf:regzd:version}. From the same argument as in \eqref{pf:Vkrhoe:ptws} we obtain from the Dominated Convergence Theorem that 
\begin{equation*}
  \int_{\R} (V_k * \mu_\delta) \rho_\varepsilon
  \xto{\delta \to 0} \int_{\R} (V_k * \mu) \rho_\varepsilon.
\end{equation*}
Since $V_k * \mu \in C_b(\R)$ and $\rho$ has no atoms, we further obtain
\begin{equation*}
  \int_{\R} (V_k * \mu) \rho_\varepsilon
  \xto{\varepsilon \to 0} \int_0^1 (V_k * \mu) \, d\rho.
\end{equation*}
Since $0 \leq V_k \nearrow V$ as $k \to \infty$, we conclude by the Monotone Convergence Theorem that
\begin{equation*}
  \int_0^1 (V_k * \mu) \, d\rho
  \xto{k \to \infty}    \int_0^1 (V * \mu) \, d\rho.
\end{equation*}
Regarding the right-hand side of \eqref{pf:regzd:version}, we rewrite it as
\begin{equation*}
  \int_{\R} \widehat{ V_k } \widehat{ \mu_\delta } \overline{ \widehat{ \rho_\varepsilon }}
  = \int_{\R} \frac{ \widehat{ V_k } }{ \widehat V } \big( v \widehat{ \mu_\delta } \big) \big( \overline{ v \widehat{ \rho_\varepsilon }} \big)
\end{equation*}
Since $0 \leq \widehat{ V_k } / \widehat V \leq 1$, we obtain
\begin{equation*}
  \int_{\R} \frac{ \widehat{ V_k } }{ \widehat V } \big( v \widehat{ \mu_\delta } \big) \big(  \overline{ v \widehat{ \rho_\varepsilon }} \big)
  \xto{ \substack{ \delta \to 0 \\ \varepsilon \to 0 } }
  \int_{\R} \widehat{ V_k } \widehat{ \mu } \overline{ \widehat{ \rho }}.
\end{equation*}
Since $\widehat V > 0$ and \eqref{VmVkhat:conv} imply that  $\| (\widehat V - \widehat{ V_k }) / \widehat V \|_{L^\infty (K)} \to 0$ as $k \to \infty$ for any compact $K \subset \R$, we obtain
\begin{equation*}
  \int_{\R} \widehat{ V_k } \widehat{ \mu } \overline{ \widehat{ \rho }}
  = \int_{\R} \bigg( 1 - \frac{\widehat V - \widehat{ V_k }}{ \widehat V } \bigg) \widehat{ V } \widehat{ \mu } \overline{ \widehat{ \rho }}
  \xto{k \to \infty} \int_{\R} \widehat{ V } \widehat{ \mu } \overline{ \widehat{ \rho }}
  = (\mu, \rho)_V. \qedhere
\end{equation*}
\end{proof}

Lemma \ref{l:convo:consy} motivates the following notational convention:

\begin{defn}[Convolution with $V$] \label{d:convo}
Let $\rho \in H_V (0,1) \cup \mathcal P([0,1])$. If $\rho \in \mathcal P([0,1])$, then we interpret $V * \rho$ as the lower semi-continuous function defined in \eqref{convo:Vrho}. If $\rho \in H_V (0,1)$, then $V * \rho := T^2 \rho \in L^2(0,1)$.
\end{defn}

We end this section with two further properties of the convolution with $V$:

\begin{lem} \label{lem:Vastxi:zero:means:xi:zero}
If $\zeta \in H_V(0,1)$ satisfies $V * \zeta = 0$ a.e.~on $(0,1)$, then $\zeta = 0$.
\end{lem} 

\begin{proof}
By \eqref{for:HV:densy:Ccinf} there exists $(\varphi_n) \subset C_c^\infty (0,1)$ such that $\varphi_n \to \zeta$ in $H_V(0,1)$ as $n \to \infty$. From Corollary \ref{cor:props:HV}, we then obtain that $T \varphi_n \to T\zeta$ in $L^2 (\R)$. Using these observations, we compute
\begin{equation*}
  0
  = \int_0^1 (V * \zeta) \varphi_n
  = (T^2 \zeta, \varphi_n)_{L^2(\R)}
  = (T \zeta, T \varphi_n)_{L^2(\R)}
  \xto{n \to \infty} (T \zeta, T\zeta)_{L^2(\R)}
  = \| \zeta \|_V^2. \qedhere
\end{equation*}
\end{proof}

Let $\mathcal M ([0,1])$ be the space of finite, signed Borel measures on $[0,1]$.

\begin{lem} \label{lem:convo:with:psi}
It holds for all $\nu \in H_V (0,1) \cap \mathcal M([0,1])$ that 
\begin{equation*}
  \left\{ \begin{aligned}
    \Vreg * \nu &\in W^{2,1} (0,1)
    &&\text{if } 0 < a < 1, \\
    \Vreg * \nu &\in W^{2, \plog} (0,1)
    &&\text{if } a = 0.    
  \end{aligned} \right.
\end{equation*}

\end{lem} 

\begin{proof}
We set $p = 1$ if $0 < a < 1$, and $p = \plog$ if $a = 0$. Since $\nu \in \mathcal M([0,1])$ and $\Vreg^{(k)} \in L^p (-1,1)$ for $k = 0,1,2$, we obtain by the generalised convolution inequality (see, e.g., \cite[Prop.~3.9.9]{Bogachev07}) that
\begin{equation*} 
  \big\| \Vreg^{(k)} * \nu \big\|_{L^p(0,1)} 
  \leq \big\| \Vreg^{(k)} \big\|_{L^p(0,1)} |\nu| ([0,1])
  < \infty,
\end{equation*}
and thus it remains to check that $\Vreg^{(k)} * \nu = (\Vreg * \nu)^{(k)}$ a.e.~on $(0,1)$. To check this, we take any $\varphi \in C_c^\infty (0,1)$, and compute, using that $\Vreg$ is even,
\begin{multline*}
  \big\langle \varphi, \Vreg^{(k)} * \nu \big\rangle
  = (-1)^k \big\langle \Vreg^{(k)} * \varphi, \nu \big\rangle
  = (-1)^k \big\langle \Vreg * \varphi^{(k)}, \nu \big\rangle \\
  = (-1)^k \big\langle \varphi^{(k)}, \Vreg * \nu \big\rangle
  = \big\langle \varphi, (\Vreg * \nu)^{(k)} \big\rangle. 
  \qedhere
\end{multline*}
\end{proof}

\section{Regularity of the solutions to Carleman's equations}
\label{s:Carl}

Carleman \cite{Carleman22} was the first to give an explicit solution formula for the integral equation
\begin{equation} \label{for:Carleman:eqn}
  \int_0^1 V_a (t-s) u(s) \, ds
  = f (t),
  \quad \text{for a.e.~} 0 < t < 1
\end{equation}
for any $0 \leq a < 1$, where $V_a$ is defined in \eqref{Va}. The family of equations \eqref{for:Carleman:eqn} parametrised by $a$, are called \emph{Carleman's equations}. As in \eqref{soln:form:Abs}, we set ${\mathcal C}_a$ as the linear solution operator, i.e., $u = {\mathcal C}_a f$. However, no precise solution concept for $u$ is given, and the minimal requirements on the regularity of the data $f$ are not specified. Since it is not readily verified that ${\mathcal C}_a f$ indeed satisfies the integral equation, the minimal requirements on $f$ are not easily obtained. 
The aim of this section is to find sufficient requirements on $f$ for which \eqref{for:Carleman:eqn} has a unique solution $u$ in a specific function space (see Theorem \ref{thm:expl:soln:a} in Section \ref{ss:Carl:pos} and Theorem \ref{thm:expl:soln:log} in Section \ref{ss:Carl:0}). Our proof reverses the steps of the constructive solution method of \eqref{for:Carleman:eqn} in \cite[\S 2.6]{EstradaKanwal00}, and justifies all these steps for the assumed regularity on $f$. In preparation for the proof and for later use, we define in Section \ref{ss:Carl:prel} the operator ${\mathcal C}_a$ and compute explicitly ${\mathcal C}_a f$ for $f(t) = A + Bt$. 

\subsection{The linear solution operator $\mathcal C_a$}
\label{ss:Carl:prel} 

The operator ${ \mathcal C}_a$ is given by
\begin{equation} \label{def:Ca} 
  \mathcal C_a f = \left\{ \begin{aligned}
    & \frac{\Gamma(a)}{\pi^2}  \cos^2 \Big( \frac{a\pi}2 \Big) \Big( \frac1{\phi_a} S ( \phi_a D^{1-a} f ) + \pi \tan \Big( \frac{a\pi}2 \Big) D^{1-a} f \Big)
    &&\text{if } 0<a<1 \\
    &\frac1{\pi^2 \phi_0} \bigg[ S ( \phi_0 f' ) + \frac{1}{2 \log 2} \int_0^1 \frac{f(s) }{ \phi_0(s) } \, ds \bigg]
    &&\text{if } a=0,
  \end{aligned} \right.
\end{equation}
where $S$, $\phi_a$ and $D^{1-a}$ are defined in Section \ref{s:prel}.
We leave possible choices for functions $f : (0,1) \to \R$ to Theorems \ref{thm:expl:soln:a} and \ref{thm:expl:soln:log}. Here, we compute $\mathcal C_a p$ for polynomials $p$.

\begin{prop} [$\mathcal C_a$ on polynomials] \label{prop:Ca} Let $0 \leq a < 1$ and $p$ be a polynomial of degree $k \in \N$. Then, there exists a polynomial $q$ of degree $k$ such that
$$
  \mathcal C_a p = q/\phi_a.
$$
In particular, for $a > 0$, it holds that  
\begin{equation} \label{Ca:affine}
  \mathcal C_a (A + Bt) = \frac{ \cos \tfrac{a \pi}2 }{\pi \phi_a (t)} \Big( \frac Ba t + \Big[ A - \frac{1-a}{2a} B \Big] \Big).
\end{equation}
\end{prop}

\begin{proof}
Since $\mathcal C_a$ is linear, it suffices to compute $\mathcal C_a (t^k)$ for $k \in \N$. We start with $a = 0$. From \eqref{def:Ca} it follows that
\[
  \mathcal C_0 (t^k) = \frac k{\pi^2 \phi_0(t)} \bigg[ S ( \phi_0(t) t^{k-1} ) + \frac{1}{2 \log 2} \int_0^1 \frac{s^k }{ \phi_0(s) } \, ds \bigg].
\]
To show that $S ( \phi_0(t) t^{k-1} )$ is a polynomial of degree $k$, it suffices to apply Proposition \ref{prop:props:S}.\ref{prop:props:S:pol} $k-1$ times, and use that $(S \phi_0)(t) = \pi (t - \frac12)$ (see \cite[(11.57)]{King09I}).

Next we treat the case $0 < a < 1$. By \eqref{Dat2},
\begin{equation*}
  D^{1-a} (t^k) = t^{k-1+a} / \Gamma(k+a).
\end{equation*}
Then, applying Proposition \ref{prop:props:S}.\ref{prop:props:S:pol} $k$ times and inserting  \eqref{S:id} and \eqref{Zwil}, we compute
\begin{align*}
  \Gamma(k+a) S ( \phi_a(t) D^{1-a} (t^k) )
  &= S \bigg( t^k \Big[ \frac{1-t}t \Big]^{\tfrac{1-a}2} \bigg) \\
  &= t^k S \bigg( \Big[ \frac{1-t}t \Big]^{\tfrac{1-a}2} \bigg) - \sum_{\ell=0}^{k-1} t^\ell \int_0^1 s^{k-1-\ell} \Big[ \frac{1-s}s \Big]^{\tfrac{1-a}2} \, ds \\
  &= t^k \frac{\pi}{  \cos \tfrac{a\pi}2 } \bigg( 1 - \Big[ \frac{1-t}t \Big]^{\tfrac{1-a}2} \sin \frac{a\pi}2 \bigg) + \sum_{\ell=0}^{k-1} b_\ell t^\ell \\
  &= - \pi \tan \Big( \frac{a\pi}2 \Big) t^{k-1+a} \phi_a(t) + \sum_{\ell=0}^{k} b_\ell t^\ell,
\end{align*}
where the coefficients $b_\ell \in \R$ may depend on $a$ and $k$. Inserting this in \eqref{def:Ca}, the first term cancels out with the second term in \eqref{def:Ca}, and thus we obtain that $\mathcal C_a (t^k) = q(t)/\phi_a(t)$ for some polynomial $q$ of degree $k$. By following the track of constants carefully, we obtain from the basic properties of the $\Gamma$-function that
\begin{equation*}
  \mathcal C_a (1) = \frac{  \cos \tfrac{a\pi}2 }{\pi \phi_a(t)}
  \quad \text{and} \quad
  \mathcal C_a (t) = \frac{  \cos \tfrac{a\pi}2 }{\pi \phi_a(t)} \Big( \frac ta - \frac{1-a}{2a} \Big);
\end{equation*}
\eqref{Ca:affine} follows.
\end{proof}

\subsection{Carleman's equation for $0 < a < 1$}
\label{ss:Carl:pos}


We introduce for $1 \leq p < \infty$ the weak $L^p$ space \cite{Grafakos04} by
\begin{equation} \label{for:defn:wLp}
  wL^p(0,1)
  := \Big\{ f : (0,1) \to \R \text{ measurable} : \esssup_{y > 0} y^p \mathcal L ( \{ |f| > y \} ) < \infty \Big\},
\end{equation}
where $\{ |f| > y \} \subset (0,1)$ is the upper-level set of $|f|$. We only use the basic properties
\begin{equation*}
  L^p(0,1)
  \subset wL^p(0,1)
  \subset L^q(0,1)
  \quad \text{for all } 1 \leq q < p
\end{equation*}
and that $wL^p(0,1) \setminus L^p(0,1)$ contains functions with a $|t|^{-1/p}$-type singularity. 

In view of Section \ref{ssec:Vip}, the left-hand side in \eqref{for:Carleman:eqn} equals $(V_a * u)(t)$ when $u$ is extended to $\R$ with value $0$. It is therefore natural to consider the Hilbert space $H_{V_a}(0,1)$ as in \eqref{for:defn:HV}.

\begin{thm} [Explicit solution] \label{thm:expl:soln:a}
Let $0 < a < 1$ and $f \in W^{2,1}(0,1)$. Then \eqref{for:Carleman:eqn} has a unique solution $u \in H_{V_a} (0,1)$. The solution $u$ is given by $u = {\mathcal C}_a f$ (see \eqref{def:Ca}) and satisfies
\begin{align} \label{for:rhobar:expl:a:f:regy} 
  \phi_a u \in C^\beta ([0,1])
  \qquad \text{for any } 0 < \beta < \tfrac{1-a}2 \wedge a.
\end{align}
\end{thm}

\begin{proof} [Proof of Theorem \ref{thm:expl:soln:a}]
First, we prove that \eqref{for:Carleman:eqn} attains at most one solution in $H_{V_a} (0,1)$. Taking any $u_1, u_2 \in H_{V_a} (0,1)$ which satisfy \eqref{for:Carleman:eqn}, we note that the difference $u := u_1 - u_2$ satisfies $V_a * u = 0$ a.e.~on $(0,1)$, and thus, by Lemma \ref{lem:Vastxi:zero:means:xi:zero}, we conclude $u = 0$.

In the remainder of the proof we show that $u := \mathcal C_a f$ satisfies \eqref{for:Carleman:eqn} and \eqref{for:rhobar:expl:a:f:regy}. First we prove \eqref{for:rhobar:expl:a:f:regy}. For constant functions $f \equiv C$ this is a direct consequence of Proposition \ref{prop:Ca}. Since $\mathcal C_a$ is linear, we can use this observation to assume that $f(0) = 0$ without loss of generality. For such $f$ we obtain from Proposition \ref{prop:reg:Da}.\ref{prop:reg:Da:Sob} that 
$$g := D^{1-a} f \in W^{1,p}(0,1) \quad \text{for any } 1 \leq p < \tfrac1{1-a}.$$  
Then, we obtain from Proposition \ref{prop:Morrey} that $g \in C^\alpha([0,1])$ for any $0 < \alpha < a$. To conclude \eqref{for:rhobar:expl:a:f:regy}, we observe from \eqref{def:Ca} that $u$ has the structure $u = C_1 g + C_2 S(\phi_a g) / {\phi_a}$ with $\phi_a \in C^{(1-a)/2} ([0,1])$. Splitting $S (\phi_a g) = (S (\phi_a g) - \phi_a Sg) + \phi_a Sg$, we obtain from Proposition \ref{prop:props:S:Hol}.\ref{prop:props:S:Hol:cp} that $(S (\phi_a g) - \phi_a Sg) \in C^\beta ([0,1])$ for any $0 < \beta < \frac{1-a}2 \wedge a$. Recalling the definition of the weighted H\"older space in \eqref{for:defn:wHolS}, we obtain by Proposition \ref{prop:props:S:Hol}.\ref{prop:props:S:Hol:bdd} and $g \in C^\beta([0,1]) \subset C_0^\beta(\phi_a)$ that $Sg \in C_0^\beta(\phi_a)$, and thus $\phi_a Sg \in C_0^\beta ([0,1])$. In conclusion
\begin{equation*}
  \phi_a u = \underbrace{ C_1 \phi_a g }_{ \in C^\beta([0,1]) } 
  	  + \underbrace{ C_2 (S (\phi_a g) - \phi_a Sg) }_{ \in C^\beta ([0,1]) } 
  	  + \underbrace{ C_2 \phi_a Sg }_{ \in C^\beta_0 ([0,1]) }
  	  \in C^\beta ([0,1])
  \quad \text{for any } 0 < \beta < \tfrac{1-a}2 \wedge a. 
\end{equation*} 

Next we prove that $u = \mathcal C_a f$ satisfies \eqref{for:Carleman:eqn} for any $f \in W^{2,1}(0,1)$. For such $f$, our proof follows and justifies the computation for the solution of \eqref{for:Carleman:eqn} outlined in \cite[\S 2.6]{EstradaKanwal00}.

First, we list several observations on the regularity of $u$. From \eqref{for:rhobar:expl:a:f:regy} we obtain
\begin{equation} \label{for:pf:thm:expl:soln:a:05}
  u \in wL^{\tfrac2{1-a}} (0,1) \cap L^\infty_{\operatorname{loc}}(0,1),
\end{equation}
where $wL^{2/(1-a)} (0,1)$ is defind in \eqref{for:defn:wLp},
\begin{equation} \label{for:pf:thm:expl:soln:a:10}
  \frac{u(t)}{t^a} 
  = \frac{(\phi_a u)(t)}{ \phi_a(t) t^a}
  \in wL^{\tfrac2{1+a}} (0,1),
\end{equation}
and, using Proposition \ref{prop:props:S:Hol}.\ref{prop:props:S:Hol:bdd},
\begin{equation} \label{for:pf:thm:expl:soln:a:11}
  S \Big( \frac{u(t)}{t^a} \Big)
  = S \Big( \frac{ (1-t)^{a} (\phi_a u)(t) }{ [t(1-t)]^{(a+1)/2} } \Big)
  \in wL^{\tfrac2{1+a}} (0,1) \cap L_{\text{loc}}^\infty (0,1).
\end{equation}
Finally, we obtain from Proposition \ref{prop:reg:Da}.\ref{prop:reg:Da:Sob} that 
\begin{equation} \label{for:pf:thm:expl:soln:a:95}
  \tilde g (t)
  := t^{1-a} D^{1-a}f(t) \in C([0,1]).
\end{equation}

Next we start the computation. From \eqref{def:Ca} we obtain
\begin{align*} 
  \tilde u(t) &:= \frac{t^{1-a} u(t)}{\Gamma(a)} 
  = \frac{t^{1-a} \mathcal C_a f(t)}{\Gamma(a)} \\
  &= \frac{ \sin(\frac{a\pi}2)  \cos(\frac{a\pi}2) }{\pi} t^{1-a} D^{1-a}f(t)
	+ \frac{ \cos^2 (\frac{a\pi}2) }{ \pi^2 } \frac{ t^{1-a} }{\phi_a(t)} S (\phi_a D^{1-a}f) (t) \\
  &= \frac{ \sin(a\pi) }{2 \pi} \tilde g (t)
	+ \frac{ \cos^2 (\frac{a\pi}2) }{ \pi^2 } \Big( \frac t{1-t} \Big)^{\tfrac{1-a}2} S \bigg( \Big( \frac{1-t}t \Big)^{\tfrac{1-a}2} \tilde g (t) \bigg).
\end{align*}
Then, we observe from Proposition \ref{prop:IE} and \eqref{for:pf:thm:expl:soln:a:95} that $\tilde u$ satisfies
\begin{equation} \label{for:pf:thm:expl:soln:a:8}
  \pi \tan (\tfrac{a\pi}2) \tilde u (t) - S \tilde u (t)
  = \tilde g (t)
  = t^{1-a} D^{1-a} f(t)
  \quad \text{for a.e.~} 0 < t < 1.
\end{equation}
Relying on \eqref{for:pf:thm:expl:soln:a:10}, we use Proposition \ref{prop:props:S}.\ref{prop:props:S:pol} to rewrite \eqref{for:pf:thm:expl:soln:a:8} as
\begin{multline} \label{for:pf:thm:expl:soln:a:7}
  D^{1-a} f(t)
  = \frac1{\Gamma(a)} \Big( \pi \tan (\tfrac{a\pi}2) u (t) - \frac1{t^{1-a}} S \big( t^{1-a} u (t) \big) \Big)  \\
  = \frac{ \pi ( 1 - \cos (a\pi) ) }{\Gamma(a) \, \sin (a\pi) } u (t) - \frac{ t^a }{ \Gamma(a) } S \Big( \frac{ u (t) }{t^a} \Big) + \frac1{ \Gamma(a) t^{1-a} } \int_0^1 \frac{ u (s) }{s^a} \, ds.
\end{multline}

We observe from \eqref{for:pf:thm:expl:soln:a:05} and \eqref{for:pf:thm:expl:soln:a:11} that all three terms in the right-hand side of \eqref{for:pf:thm:expl:soln:a:7} are in $L^1 (0,1)$. Then, we use Proposition \ref{prop:FDE} to apply $I^{1-a}$ to all four terms in both sides of \eqref{for:pf:thm:expl:soln:a:7}. This yields, using \eqref{Dat2},
\begin{equation*}
  I^{1-a} D^{1-a} f = f
  \quad \text{and} \quad
  I^{1-a} \Big( \frac1{ \Gamma(a) t^{1-a} } \Big) = 1.
\end{equation*}
Regarding the other two terms in \eqref{for:pf:thm:expl:soln:a:7}, we use \eqref{for:id:Gamma} to obtain
\begin{equation*}
  \frac{ \pi }{\Gamma(a) \, \sin (a\pi) } I^{1-a} u (t)
  = \int_0^t \frac{u(s)}{(t-s)^a} \, ds
\end{equation*}
and
\begin{equation*}
  I^{1-a} \Big( \frac{ t^a }{ \Gamma(a) } S \Big( \frac{ u (t) }{t^a} \Big) \Big)
  = \frac{\sin (a\pi) }{ \pi } \int_0^t \frac{s^a}{(t-s)^a} S \Big( \frac{ u (s) }{s^a} \Big) \, ds.
\end{equation*}
Hence, applying $I^{1-a}$ to \eqref{for:pf:thm:expl:soln:a:7} yields
\begin{multline} \label{for:pf:thm:expl:soln:a:6}
   f(t)
   = \int_0^t \frac{u(s)}{(t-s)^a} \, ds 
   + \bigg[ - \cos (a\pi) \int_0^t \frac{u(s)}{(t-s)^a} \, ds 
   - \frac{\sin (a\pi) }{ \pi } \int_0^t \frac{s^a}{(t-s)^a} S \Big( \frac{ u (s) }{s^a} \Big) \, ds \\
   + \int_0^1 \frac{ u (s) }{s^a} \, ds \bigg]
   \quad \text{for a.e.~} 0 < t < 1.
 \end{multline} 

It is left to show that the expression within brackets in \eqref{for:pf:thm:expl:soln:a:6} equals 
\begin{equation} \label{for:pf:thm:expl:soln:a:4}
  \int_t^1 \frac{u(s)}{|t-s|^a} \, ds.
\end{equation}
With this aim, we fix $0 < t < 1$ arbitrarily, and focus on the second term within these brackets. In preparation for applying Proposition \ref{prop:props:S}.\ref{prop:props:S:Fubini}, we regularise the integrand by replacing $(t-s)^{-a}$ by $|t-s|^{-a} \indicatornoacc{(0, t-\varepsilon)} (s)$ for $0 < \varepsilon < t$, which we interpret as a function of $s$ on $(0,1)$. Since $|t-s|^{-a} \indicatornoacc{(0, t-\varepsilon)} (s)$ converges in $L^p(0,1) \cap L^\infty(0, \tfrac t2)$ for any $1 \leq p < \tfrac1a$ to $|t-s|^{-a} \indicatornoacc{(0, t)} (s)$ as $\varepsilon \to 0$, we obtain from \eqref{for:pf:thm:expl:soln:a:11} that
\begin{equation} \label{for:pf:thm:expl:soln:a:55}
  \int_0^1 \frac{s^a \indicatornoacc{(0, t-\varepsilon)} (s) }{|t-s|^a} S \Big( \frac{ u (s) }{s^a} \Big) \, ds
  \xto{ \varepsilon \to 0 } \int_0^t \frac{s^a}{(t-s)^a} S \Big( \frac{ u (s) }{s^a} \Big) \, ds.
\end{equation}
Since $\big[ s \mapsto |t-s|^{-a} \indicatornoacc{(0, t-\varepsilon)} (s) \big] \in L^\infty(0,1)$ for any $\varepsilon \in (0,t)$, we can apply Proposition \ref{prop:props:S}.\ref{prop:props:S:Fubini},\ref{prop:props:S:pol} to obtain
\begin{multline} \label{for:pf:thm:expl:soln:a:5}
  \int_0^1 \frac{s^a \indicatornoacc{(0, t-\varepsilon)} (s) }{|t-s|^a} S \Big( \frac{ u (s) }{s^a} \Big) \, ds
  = -\int_0^1 \frac{ u (s) }{s^a} S \Big( \frac{s^a \indicatornoacc{(0, t-\varepsilon)} (s) }{|t-s|^a} \Big) \, ds \\
  = -\int_0^1 s^{1-a} u (s) S \Big( \frac{ \indicatornoacc{(0, t-\varepsilon)} (s) }{s^{1-a} |t-s|^a} \Big) \, ds + \bigg[ \int_0^1 \frac{ u (s) }{s^a} \, ds \bigg] \bigg[ \int_0^{t-\varepsilon} \frac1{s^{1-a} |t-s|^a} \, ds \bigg].
\end{multline} 

Next we pass to the limit $\varepsilon \to 0$ in the right-hand side of \eqref{for:pf:thm:expl:soln:a:5}. For the second term we obtain with \eqref{for:id:Gamma} and \eqref{Zwil} that
\begin{equation*}
  \int_0^{t-\varepsilon} \frac1{s^{1-a} |t-s|^a} \, ds
  \xto{ \varepsilon \to 0 } \frac\pi{ \sin (a\pi) }.
\end{equation*}
For the first term, we split the integration domain $(0,1)$ in $(0, \tfrac{1+t}2)$ and $( \tfrac{1+t}2, 1)$. On $(0, \tfrac{1+t}2)$, we obtain from \eqref{for:rhobar:expl:a:f:regy} that $s^{1-a} u (s)$ is bounded, and from Proposition \ref{prop:props:S}.\ref{prop:props:S:BLO} that the other term in the integrand converges in $L^p (0, \tfrac{1+t}2)$ for any $1 \leq p < \tfrac1a \wedge \tfrac1{1-a}$. On $( \tfrac{1+t}2, 1)$, we obtain that
\begin{equation*}
  S \Big( \frac{ \indicatornoacc{(0, t-\varepsilon)} (s) }{s^{1-a} |t-s|^a} \Big)
  = \int_0^{t-\varepsilon} \frac{ r^{-(1-a)} (t-r)^{-a} }{ s - r } \, dr
  \xto{ \varepsilon \to 0 } S \Big( \frac{ \indicatornoacc{(0, t)} (s) }{s^{1-a} |t-s|^a} \Big) 
  \quad \text{for } \frac{1+t}2 < s < 1
\end{equation*}
is a regular integral, where the convergence is uniformly in $s \in ( \tfrac{1+t}2, 1)$. Using \eqref{EK00}, we further rewrite
\begin{equation*}
  S \Big( \frac{ \indicatornoacc{(0, t)} (s) }{s^{1-a} |t-s|^a} \Big)
  = \frac{ 1 }{ s^{1-a} |t-s|^a } \left\{ \begin{aligned}
     &\tfrac\pi{ \tan (a \pi) }
     &&\text{if } s < t \\
     &\tfrac\pi{ \sin (a \pi) }
     &&\text{if } s > t
   \end{aligned} \right\}
   \quad \text{for a.e.~} 0 < s < 1.
\end{equation*}
In conclusion, passing to the limit $\varepsilon \to 0$ in the right-hand side of \eqref{for:pf:thm:expl:soln:a:5} yields, together with \eqref{for:pf:thm:expl:soln:a:55},
\begin{multline*} 
  \int_0^t \frac{s^a}{(t-s)^a} S \Big( \frac{ u (s) }{s^a} \Big) \, ds \\
  =
  - \frac\pi{ \tan (a \pi) } \int_0^t \frac{u(s)}{|t-s|^a} \, ds 
  - \frac\pi{ \sin (a \pi) } \int_t^1 \frac{u(s)}{|t-s|^a} \, ds 
  + \frac\pi{ \sin (a\pi) } \int_0^1 \frac{ u (s) }{s^a} \, ds.
\end{multline*}
Substituting this expression in \eqref{for:pf:thm:expl:soln:a:6}, we obtain that the term within brackets equals \eqref{for:pf:thm:expl:soln:a:4}, which completes the proof.
\end{proof}

\subsection{Carleman's equation for $a = 0$}
\label{ss:Carl:0}

\begin{thm} [Explicit solution] \label{thm:expl:soln:log}
Let $f \in C^{1,\alpha} ([0,1])$ for some $0 < \alpha < 1$. Then
\begin{equation} \label{for:IE:log}
  \int_0^1 -\log|t-s| \, u(s) \, ds 
  = f (t)
  \quad \text{for a.e.~} 0 < t < 1
\end{equation}
has a unique solution $u \in H_{1 - \log|\cdot|} (0, 1)$. It is given by $u = {\mathcal C_0} f$ (see \eqref{def:Ca}) and satisfies
\begin{align} \label{for:rhobar:expl:log:f:regy} 
 \phi_0 u \in C^\alpha ([0, 1]).
\end{align}
\end{thm}

\begin{proof}
We follow a similar proof strategy as for Theorem \ref{thm:expl:soln:a}. In fact, the proof is analogous for the uniqueness of $u$. Also, the regularity property \eqref{for:rhobar:expl:log:f:regy} follows by an analogous argument. Therefore, we omit further details, and focus on proving the existence of a solution $u$ to \eqref{for:IE:log}.

In the remainder of the proof, we show that $u := {\mathcal C_0} f$ satisfies \eqref{for:IE:log}. We observe from \eqref{def:Ca} and Proposition \ref{prop:IE} that $u$ satisfies 
\begin{equation} \label{Suf}
  -S u = f'
   \quad \text{a.e.~on } (0,1).   
\end{equation}
Using Proposition \ref{prop:props:S}.\ref{prop:props:S:log} and integrating \eqref{Suf}, we obtain that 
\begin{equation} \label{for:pf:thm:expl:soln:log:4}
  \int_0^1 -\log| t -  s| \,  u( s) \, d  s
  = f ( t) + C \quad \text{for all } t \in (0,1)
\end{equation}
for some constant $C \in \R$. It is left to prove that $C = 0$. This is done by the computation in Appendix \ref{app:pf:thm:expl:soln:log}.
\end{proof}

\section{Proof of Theorem \ref{t}}
\label{s:pfs}

The proof of Theorem \ref{t} is divided in 3 parts:
\begin{enumerate}
  \item both Problems \ref{prob:minz} and \ref{prob:VI} have a unique solution, and both solutions are given by the same measure $\overline \rho$ (Proposition \ref{p:E});
  \item $\overline \rho$ satisfies Problem \ref{prob:wSIE} and has the regularity
     stated in Theorem \ref{t} (Lemma \ref{lem:t});
  \item the uniqueness of solutions to Problem \ref{prob:wSIE} (Lemma \ref{lem:Un:P13}).
\end{enumerate}
The proof of Proposition \ref{p:E} concerns fitting Problem \ref{prob:minz} to Theorem \ref{thm:CoV}. Lemma \ref{lem:Un:P13} is proven at the end, because it relies on the regularity statements in Theorem \ref{t} for an auxiliary problem.

The main part of the proof concerns Lemma \ref{lem:t}, which we prove in seven steps. In the first two steps, we show that $\overline \rho$ is a solution to Problem \ref{prob:wSIE} and satisfies $\supp \overline \rho = [0,1]$. In these steps, we rely on the convexity properties of $V$ and $U$. The sole purpose of the additional assumption in \eqref{VextV:bd} is to provide a sufficient condition for $\supp \overline \rho = [0,1]$. 

In Step 3 we rely on the splitting $V = V_a + \Vreg$ to rewrite \eqref{for:prop:E:on:P:EL:strong} in the form of Carleman's equation, to which Theorems \ref{thm:expl:soln:a} and \ref{thm:expl:soln:log} provide an explicit expression for the solution. From these theorems we obtain that $\overline \rho$ has a continuous representative which satisfies $\overline \rho = \mathcal C_a f_{\overline \rho}$ (see \eqref{def:Ca}). In the remaining Steps 4--7, we examine $\mathcal C_a f_{\overline \rho}$ to establish the remaining regularity properties of $\overline \rho$.

\begin{prop}[Equivalence of Problems \ref{prob:minz}--\ref{prob:VI}] \label{p:E}
Let $0 \leq a < 1$ and let the potentials $V$, $\Vreg$ and $\Vext$ satisfy \eqref{for:V:props} and \eqref{for:Vext:assns}. Then both Problem \ref{prob:minz} and Problem \ref{prob:VI} have a unique solution, and both solutions are equal.
\end{prop}

\begin{proof} 
First we show that any minimiser of $E$ has to be in $H_V(0,1)$. Let $\rho \in \mathcal P ([0,1])$ such that $E(\rho) < \infty$. Since $\Vext \geq 0$, we find that $\int_0^1 (V*\rho) \, d\rho < \infty$. Then, by Proposition \ref{prop:HV}, we conclude that $\rho \in H_V(0,1)$. 
In addition, by Lemma \ref{l:convo:consy}, we can write
\begin{equation} \label{E:Vip}
   E(\rho) = \frac12 ( \rho, \rho)_V + \int_0^1 \Vext \, d\rho. 
 \end{equation} 

Next we show that Theorem \ref{thm:CoV} applies to the energy $E$ with $H_V(0,1)$ as the Hilbert space and $H_V(0,1) \cap \mathcal P ([0,1])$ as the closed convex subset. Convexity of $H_V(0,1) \cap \mathcal P ([0,1])$ is obvious, and closedness follows by interpreting any $H_V(0,1)$-converging sequence $(\zeta_n) \subset H_V(0,1) \cap \mathcal P ([0,1])$ as a converging sequence of distributions on $\R$ with support in $[0,1]$, for which non-negativity is conserved in the limit, and the unit integral condition follows by testing with any $\varphi \in C_c^\infty(\R)$ with $\varphi |_{(0,1)} = 1$. For the linear term in \eqref{E:Vip}, we use \eqref{for:Vext:assns} to extend $\Vext$ to $\R$ such that $\Vext \in H^{(1-a)/2}(\R) \cap C_c(\R)$. Then, $\langle \Vext, f \rangle := \int_\R \Vext f$ extends as a bounded linear functional on $L^2 (\R)$ to a bounded linear functional on $H^{-(1-a)/2}(\R) + \mathcal M(\R)$. Using Corollary \ref{cor:props:HV}, we find in particular that $\langle \Vext, \cdot \rangle$ is a bounded linear functional on $H_V(0,1)$, and that there exists $\xiext \in H_V(0,1)$ such that $\langle \Vext, \eta \rangle = (\xiext, \eta)_V$ for all $\eta \in H_V(0,1)$. We conclude that Theorem \ref{thm:CoV} applies.

By Theorem \ref{thm:CoV} we obtain that both Problem \ref{prob:minz} and the variational inequality given by
\begin{equation} \label{pf:prop:E:on:P:m1}
  0 \leq (\rho + \xiext, \mu - \rho)_V,
  \quad \text{for all } \mu \in \mathcal P ([0,1]) \cap H_V(0,1),
\end{equation}
with solution concept $\rho \in H_V(0,1) \cup \mathcal P ([0,1])$, have the same unique solution $\overline \rho$. 

It is left to show the equivalence with Problem \ref{prob:VI}. We first show that any solution $\rho \in H_V(0,1) \cup \mathcal P ([0,1])$ of \eqref{pf:prop:E:on:P:m1} satisfies \eqref{for:prop:E:on:P:EL:weak}. We expand
\begin{equation} \label{pf:prop:E:on:P:m05}
  0 \leq (\rho + \xiext, \mu - \rho)_V
  = (\rho, \mu)_V - (\rho, \rho)_V + (\xiext, \mu - \rho)_V.
\end{equation}
Interpreting $\rho, \mu$ as measures with no atoms, we use \eqref{for:Vnorm:on:HVnP} to write
\begin{equation*}
  (\rho, \mu)_V - (\rho, \rho)_V
  = \int_0^1 V * \rho \, d\mu - \int_0^1 V * \rho \, d\rho.
\end{equation*}
Moreover, since $\langle \Vext, \cdot \rangle$ is extended to $\mathcal M([0,1])$, its extension to the subset $\mathcal P([0,1])$ is simply given by $\langle \Vext, \tilde \mu \rangle = \int_{[0,1]} U \, d\tilde \mu$. Hence,
\begin{equation*}
  (\xiext, \mu - \rho)_V
  = \langle \Vext, \mu \rangle - \langle \Vext, \rho \rangle 
  = \int_0^1 U \, d\mu - \int_0^1 U \, d\rho.
\end{equation*}
Collecting our results, we obtain that \eqref{pf:prop:E:on:P:m05} implies \eqref{for:prop:E:on:P:EL:weak}. 

To show that any solution $\rho \in \mathcal P ([0,1])$ to Problem \ref{prob:VI} is also a solution to \eqref{pf:prop:E:on:P:m1}, we separate two cases. If $\rho \in H_V(0,1)$, then the argument above readily shows that $\rho$ satisfies \eqref{pf:prop:E:on:P:m1}. If $\rho \notin H_V(0,1)$, then 
\begin{equation*}
  \int_0^1 h_\rho \, d\rho
  \geq \int_0^1 V * \rho \, d\rho 
  = \infty.
\end{equation*}
However, for $\mu = \mathcal L |_{(0,1)}$, we find with $V_k$ as in Lemma \ref{lem:props:V}.\ref{lem:props:V:V_k} that
\begin{equation*}
  \int_0^1 h_\rho \, d\mu 
  = \lim_{k \to \infty} \int_0^1 \int_0^1 V_k(t-s) \, d \rho (s) \, dt + \int_0^1 U
  \leq \int_{-1}^1 V + \int_0^1 U < \infty,
\end{equation*}
which contradicts \eqref{for:prop:E:on:P:EL:weak}.
\end{proof}

\begin{lem}[Regularity properties of Theorem \ref{t}] \label{lem:t}
Let $a$, $\Vreg$, $\Vext$ and $\ell$ be as in Theorem \ref{t}. Let $\overline \rho$ be the solution of Problem \ref{prob:minz} provided by Proposition \ref{p:E}. Then, $\overline \rho$ is a solution to Problem \ref{prob:wSIE} and satisfies all properties listed in Theorem \ref{t}. Moreover, the continuous representative of $\overline \rho$ satisfies $\overline \rho = \mathcal C_a f_{\overline \rho}$, where $\mathcal C_a$ is defined in \eqref{def:Ca},
\begin{equation} \label{forho} 
  f_{\overline \rho} := \overline C - \Vreg * \overline \rho - \Vext
  \quad \text{and} \quad 
  \overline C = \int_0^1 h_{\overline \rho} (t) \overline \rho (t) \, dt.
\end{equation}
\end{lem} 

\begin{proof}
We recall from \eqref{convo:Vrho} and Corollary \ref{cor:props:HV} that $h_{\overline \rho} = V * \overline \rho + \Vext \in L^2(0,1)$ is lower semi-continuous on $[0,1]$. 

\textit{Step 1: $\overline \rho$ satisfies Problem \ref{prob:wSIE} for some $C \geq 0$}. In this step we prove by contradiction that 
\begin{equation} \label{pf:prop:E:on:P:0}
  \exists \, C \geq 0 : h_{\overline \rho} = C 
  \quad \text{a.e.~on } (0,1). 
\end{equation} 
Suppose $h_{\overline \rho}$ is not constant a.e.~on $(0,1)$, i.e.
\begin{equation} \label{pf:prop:E:on:P:1}
  0
  \leq m 
  := \operatorname*{ess\,inf}_{(0,1)} h_{\overline \rho}
  < \sup_{(0,1)} h_{\overline \rho}
  \leq \infty.
\end{equation}
By \eqref{pf:prop:E:on:P:1} and $h_{\overline \rho} \in L^2(0,1)$, we have 
\begin{equation} \label{pf:prop:E:on:P:1a}
  \forall \, \varepsilon > 0 
  \: \exists \, g \in \mathcal P ([0,1]) \cap L^2(0,1) 
  : \int_0^1 h_{\overline \rho} g < m + \varepsilon.
\end{equation}
We reach a contraction between \eqref{pf:prop:E:on:P:1a} and \eqref{for:prop:E:on:P:EL:weak} by showing that 
\begin{equation} \label{pf:prop:E:on:P:2}
  \int_0^1 h_{\overline \rho} d \overline \rho > m.  
\end{equation}
With this aim, we consider the superlevel set 
\begin{equation*}
  A := \big\{ h_{\overline \rho} > m \big\}.
\end{equation*}
By \eqref{pf:prop:E:on:P:1}, $A \neq \emptyset$, and since $h_{\overline \rho}$ is lower semi-continuous, $A$ is open. We take $(r,s) \subset A$ to be any open component\footnote{\label{fn:1} If $r = 0$ or $s = 1$, then also the boundary of the interval may be included in the open component}. If $\overline \rho ((r,s)) > 0$, then \eqref{pf:prop:E:on:P:2} holds, and the contradiction is reached. Hence, we assume $\overline \rho ((r,s)) = 0$, and thus $(r,s) \cap \supp \overline \rho = \emptyset$. By \eqref{for:V:props} and \eqref{for:Vext:assns} we then find that 
\begin{equation} \label{pf:prop:E:on:P:3}
  h_{\overline \rho} \in C([r,s]) \cap W_{\text{loc}}^{2,1}(r, s),
  \quad \text{and} \quad
  h_{\overline \rho}'' = V'' * \overline \rho + \Vext'' \geq 0 
  \quad \text{on } (r,s),
\end{equation} 
where continuity up to the boundary holds by the following argument. For $r < t < s$ we split
\begin{equation*} 
  h_{\overline \rho}(t) 
  = \Vext (t) + \int_0^r V (t - \tau) \, d\overline \rho (\tau) + \int_s^1 V (t-\tau) \, d\overline \rho (\tau),
\end{equation*} 
where the first two terms are continuous for $t > r$, and the third term is increasing for $t \leq s$. Hence, $\limsup_{t \uparrow s} h_{\overline \rho}(t) \leq h_{\overline \rho}(s)$, 
and together with $h_{\overline \rho}$ being lower semi-continuous on $[0,1]$, we obtain that $h_{\overline \rho}$ is left-continuous at $t = s$. A similar argument shows that $h_{\overline \rho}$ is right-continuous at $t = r$, and thus \eqref{pf:prop:E:on:P:3} follows.

We separate three cases to complete the contradiction between \eqref{pf:prop:E:on:P:1a} and \eqref{for:prop:E:on:P:EL:weak}:
\begin{enumerate}
  \item Let $r = 0$ and $s = 1$. This contradicts with $A \cap \supp \overline \rho = \emptyset$.
  \item Let $0 < r < s < 1$. Since $r, s \notin A$, it holds that $(V * \overline \rho) (r), (V * \overline \rho) (s) \leq m$. By definition of $A$, we also have that $h_{\overline \rho} > m$ on $(r,s)$. However, by the convexity of $h_{\overline \rho}$ (see \eqref{pf:prop:E:on:P:3}) and $h_{\overline \rho} \in C(\overline A)$ we find that $h_{\overline \rho} \leq m$  on $[r,s]$, and a contradiction is reached.
  \item Let $r = 0$ and $s < 1$; the case $r > 0$ and $s = 1$ can be dealt with analogously. Given that $\varepsilon > 0$ is as in \eqref{for:V:props:lambda}, we obtain
  \begin{equation} \label{pf:prop:E:on:P:43}
    h_{\overline \rho}'(t) 
    = (V' * \overline \rho)(t) + \Vext' (t)
    > |V'(1)| - \sup_{(0,1)} |\Vext'| \geq 0,
    \quad \text{for all } s - \varepsilon < t < s.
  \end{equation}
   We obtain the desired contradiction similarly to case 2 above.
\end{enumerate}
This concludes the proof of \eqref{pf:prop:E:on:P:0}. 
\smallskip

\textit{Step 2: $\supp \overline \rho = [0,1]$}.
We prove $\supp \overline \rho = [0,1]$ by a small modification to the argument in Step 1. Suppose that the open set $[0,1] \setminus \supp \overline \rho$ is non-empty, and set $(r,s)$ as one of its components\footnote{See footnote \ref{fn:1} on page \pageref{fn:1}}. Then \eqref{pf:prop:E:on:P:3} holds, and $h_{\overline \rho}'' > 0$ close to the endpoints of $(r, s)$ (by \eqref{for:V:props:lambda}) implies that $h_{\overline \rho}$ is not constant a.e.~on $(r, s)$, which contradicts \eqref{pf:prop:E:on:P:0}.
\smallskip

\textit{Step 3: $\overline \rho$ has an integrable, continuous representative on $(0,1)$ which satisfies $\overline \rho = \mathcal C_a f_{\overline \rho}$}. We rewrite \eqref{pf:prop:E:on:P:0} as
\begin{equation} \label{pf:prop:E:on:P:5}
  C 
  = h_{\overline \rho} 
  = V_a * \overline \rho + \Vreg * \overline \rho + \Vext
  \quad \text{a.e.~on } (0,1).
\end{equation}

\textit{Step 3a: $0 < a < 1$}. Since Lemma \ref{lem:convo:with:psi} implies that $\Vreg * \overline \rho + \Vext \in W^{2,1}(0,1)$, Theorem \ref{thm:expl:soln:a} states that 
\begin{equation} \label{pf:prop:E:on:P:6}
  V_a * \rho
  = C - \Vreg * \overline \rho - \Vext
  \quad \text{a.e.~on } (0,1)
\end{equation}
has a unique solution $\tilde \rho$ in $H_{V_a}(0,1) \cong H_V(0,1)$ (cf.~Proposition \ref{prop:HV}).
Subtracting \eqref{pf:prop:E:on:P:5} from \eqref{pf:prop:E:on:P:6}, we find $V_a * (\tilde \rho - \overline \rho) = 0$. Then, by Lemma \ref{lem:Vastxi:zero:means:xi:zero}, we obtain $\overline \rho = \tilde \rho$ a.e.~on $(0,1)$. Hence, by Theorem \ref{thm:expl:soln:a},  $\overline \rho$ has an integrable, continuous representative on $(0,1)$ which satisfies $\overline \rho = \mathcal C_a f_{\overline \rho}$.

\textit{Step 3b: $a = 0$}. Since the argument is similar to Step 3a, we focus on the differences. To prove that \eqref{pf:prop:E:on:P:6} has a unique solution $\tilde \rho$ in $H_{1-\log|\cdot|} \cong H_V(0,1)$, we obtain from Lemma \ref{lem:convo:with:psi} that the right-hand side of \eqref{pf:prop:E:on:P:6} is in $W^{2, \plog} (0,1)$, find by Proposition \ref{prop:Morrey} that $W^{2, \plog} (0,1) \subset C^{1, 1 - 1/\plog} ([0,1])$, and conclude with Theorem \ref{thm:expl:soln:log} that \eqref{pf:prop:E:on:P:6} has a unique solution $\tilde \rho \in H_V(0,1)$. Again, we obtain $\overline \rho = \tilde \rho$ a.e.~on $(0,1)$, and conclude by the further statements of Theorem \ref{thm:expl:soln:log}. 

From here, we denote by $\overline \rho$ the integrable, continuous representative.
\smallskip

\textit{Step 4: $C = \overline C > 0$}. Since $\overline \rho \in L^1(0,1)$, we obtain, using Step 1 and Lemma \ref{l:convo:consy}, that 
$$
  C 
  = \int_0^1 C \overline \rho
  = \int_0^1 h_{\overline \rho} \overline \rho 
  = (\overline \rho, \overline \rho)_V + \int_0^1 U \overline \rho
  > 0.
$$
\smallskip

\textit{Step 5: $\overline \rho$ satisfies \eqref{for:prop:E:on:P:regy:Wloc}}.
We first treat the case $0 < a < 1$. Since $\overline \rho \in L^1(0,1)$ we obtain from \eqref{for:V:props:lambda:psi} that 
\begin{equation} \label{pf:frho:Wl11}
  f_{\overline \rho} = \overline C - \Vreg * \overline \rho - \Vext \in W^{\ell + 1 ,1} (0,1).
\end{equation}
Let $q$ be the $(\ell-1)$-th order Taylor polynomial of $f_{\overline \rho}$ at $0$, and set $g := D^{1-a} ( f_{\overline \rho} - q)$. We expand 
\begin{equation*} 
  \overline \rho 
  = \mathcal C_a f_{\overline \rho}
  = \mathcal C_a q + C g
   + C' \tfrac1{\phi_a} S \big( \phi_a g \big),
\end{equation*}
where $C, C' \in \R$ are constants. By Proposition \ref{prop:Ca}, $\mathcal C_a q \in C^\infty_{\text{loc}} (0,1)$. For the other two terms in the right-hand side, we obtain from Proposition \ref{prop:reg:Da}.\ref{prop:reg:Da:Sob} that $g = D^{1-a} ( f_{\overline \rho} - q) \in W^{\ell, p} (0,1)$ for any $1 \leq p < \tfrac1{1-a}$. 

It remains to show that $S ( \phi_a g ) \in W^{\ell, p}_{\text{loc}} (0,1)$ for any $1 \leq p < \tfrac1{1-a}$. We take $[t_1, t_2] \subset (0,1)$ arbitrary, and choose $\psi \in W^{\ell, p} (0,1)$ such that $\psi |_{(t_1, t_2)} = ( \phi_a g ) |_{(t_1, t_2)}$. Then
\begin{equation} \label{pf:prop:E:on:P:8}
  \big( S ( \phi_a g ) \big) (t) 
  = (S \psi) (t) + \int_0^{t_1} \frac{ (\phi_a g - \psi) (s) }{t-s} \, ds \\
    + \int_{t_2}^1 \frac{ (\phi_a g - \psi) (s) }{t-s} \, ds,
    \quad t_1 < t < t_2.
\end{equation}
Since for any $t_1 < t < t_2$ the map $s \mapsto (t-s)^{-1}$ is in $C^\infty ( [0, t_1] \cup [t_2, 1] )$, the two integrals in the right-hand side of \eqref{pf:prop:E:on:P:8} are in $C^\infty ((t_1, t_2))$ as functions of $t$. By Proposition \ref{prop:props:S}.\ref{prop:props:S:BLO},\ref{prop:props:S:der} we obtain  $S \psi \in W^{\ell, p}_{\text{loc}} (0,1)$. Since $0 < t_1 < t_2 < 1$ are arbitrary, we conclude that $S ( \phi_a g ) \in W^{\ell, p}_\text{loc} (0,1)$, which completes the proof of \eqref{for:prop:E:on:P:regy:Wloc}.

The case $a = 0$ follows by a simplification of the same argument. The only difference is that \eqref{for:V:props:lambda:psi} and $\overline \rho \in wL^2(0,1) \subset L^{\plog}(0,1)$ imply
$
  f_{\overline \rho} \in W^{\ell + 1 ,\plog} (0,1).
$
No Taylor polynomial needs to be subtracted; simply putting
$
  g := f_{\overline \rho}' \in W^{\ell, \plog} (0,1)
$
suffices to repeat the argument used in \eqref{pf:prop:E:on:P:8}.
\smallskip

\textit{Step 6: $\overline \rho$ satisfies \eqref{for:prop:E:on:P:regy:Hol}}. We focus on the expansion of $\overline \rho(t)$ around $t=0$; the proof of the expansion around $t=1$ is analogous.

\textit{Step 6a: $0 < a < 1$}. We expand
$$
  f_{\overline \rho}(t) = A + Bt + R(t),
$$
where we rely on \eqref{pf:frho:Wl11} to choose the constants $A, B \in \R$ such that the remainder term $R \in C^1([0,1])$ satisfies $R(t) = o(t)$. Then, we obtain from Step 3 that
\begin{equation} \label{p:Caf:Td}
  \overline \rho = \mathcal C_a f_{\overline \rho} = \mathcal C_a (A + Bt) + \mathcal C_a R.  
\end{equation}

For the affine part, we apply Proposition \ref{prop:Ca} to obtain 
$$
  \mathcal C_a (A + Bt) = \frac{ C + C' t }{\phi_a (t)}.
$$
Here and in the remainder of Step 6, $C, C' \in \R$ denote some explicit constants that can change from line to line. For the remainder part in \eqref{p:Caf:Td}, we expand
$$
  \mathcal C_a R
  = \frac C{\phi_a} S ( \phi_a D^{1-a} R )
    + C' D^{1-a} R.
$$
From Proposition \ref{prop:reg:Da}.\ref{prop:reg:Da:Hol} we obtain that $R_1 := D^{1-a} R \in C^a([0,1])$ with $R_1(0) = 0$. Then, from Proposition \ref{prop:props:S:Hol}.\ref{prop:props:S:Hol:bdy:0} we obtain
$$
  \frac1{\phi_a} S ( \phi_a R_1 ) 
  = \frac C{\phi_a} + R_2,
$$
where $R_2 \in C^a([0,1-\varepsilon])$ with $R_2(0) = 0$ for any small $\varepsilon > 0$. Inserting all these findings in \eqref{p:Caf:Td}, we obtain
\begin{equation} \label{p:Caf:Td1}
  \overline \rho = \mathcal C_a f_{\overline \rho} = \frac C{\phi_a} + R_3,  
\end{equation}
where $R_3 \in C^a([0,1-\varepsilon])$ with $R_3(0) = 0$. This proves \eqref{for:prop:E:on:P:regy:Hol} for $0 < a < 1$. 

\textit{Step 6b: $a = 0$}. The proof is an easier version of Step 6a. For $a = 0$, \eqref{def:Ca} reads as
\begin{equation*} 
  \overline \rho
  = \mathcal C_0 \overline \rho
  = \frac1{\pi^2} \bigg( \frac1{\phi_0} S ( \phi_0 f_{\overline \rho}' )
    + \frac C{\phi_0} \bigg).
\end{equation*}
From Step 5 we get that $f_{\overline \rho}' \in C^{\alpha_0}([0,1])$ with $\alpha_0 = 1 - 1/\plog \in (0, \tfrac12)$. Hence, the term involving $S$ can be treated similarly as in Step 6a by expanding $f_{\overline \rho}'$ around $0$. This proves \eqref{for:prop:E:on:P:regy:Hol}. 
\smallskip

\textit{Step 7: $\overline \rho > 0$ on $(0,1)$}. 
We take $\ell \geq 3$, and reason by contradiction. Assume that there exists $0 < t_0 < 1$ such that $\overline \rho (t_0) = 0$. We set $0 < r < t_0 \wedge (1 - t_0)$, and compute for any $t \in B_r(t_0) = (t_0 -r, t_0 + r)$
\begin{multline*}
  0
  = h_{\overline \rho}''(t)
  = (V * \overline \rho)''(t) + \Vext''(t) \\
  = \int_{ B_r(t_0)^c } V'' (t - s) \overline \rho (s) \, ds
    + \frac{ d^2 }{ dt^2 } \bigg[ \int_{ t - t_0 -r }^{ t - t_0 +r } V(s) \overline \rho (t - s) \, ds \bigg]  + \Vext''(t).
\end{multline*}
To compute the derivative explicitly, we use \eqref{for:prop:E:on:P:regy:Wloc} and $\ell \geq 3$ to obtain $\overline \rho \in C^2 (0,1)$. This yields
\begin{align*}
  &\frac{ d^2 }{ dt^2 } \bigg[ \int_{ t - t_0 -r }^{ t - t_0 +r } V(s) \overline \rho (t - s) \, ds \bigg] \\
  &= \frac{ d }{ dt } \bigg[ 
      V (t - t_0 + r) \overline \rho (t_0 - r) 
      - V (t - t_0 - r) \overline \rho (t_0 + r)       
      + \int_{ t - t_0 -r }^{ t - t_0 +r } V(s) \overline \rho' (t - s) \, ds \bigg] \\
  &= V' (t - t_0 + r) \overline \rho (t_0 - r) 
    - V' (t - t_0 - r) \overline \rho (t_0 + r) \\
    &\quad + V (t - t_0 + r) \overline \rho' (t_0 - r) 
    - V (t - t_0 - r) \overline \rho' (t_0 + r)     
    + \int_{ t - t_0 -r }^{ t - t_0 +r } V(s) \overline \rho'' (t - s) \, ds.
\end{align*}
Setting $t = t_0$ and using that $V$ is even, we find
\begin{multline} \label{pf:prop:E:on:P:10}
  0
  = h_{\overline \rho}''(t_0)
  = \int_{ B_r(t_0)^c } V'' (t_0 - s) \overline \rho (s) \, ds
    + r^2 V' (r) \frac{ \overline \rho (t_0 + r) + \overline \rho (t_0 - r) }{ r^2 } \\
    + 2r V(r) \frac{ \overline \rho' (t_0 + r) - \overline \rho' (t_0 - r) }{2r}
    + \int_{ -r }^{ r } V(s) \overline \rho'' (t_0 - s) \, ds
    + \Vext''(t_0).
\end{multline}

Next we establish a contradiction by showing that the limit $r \to 0$ in \eqref{pf:prop:E:on:P:10} yields a positive value. By \eqref{for:Vext:assns}, $\Vext''(t_0) \geq 0$. By \eqref{for:V:props:lambda} and Step 2, the integral over $B_r(t_0)^c$ is non-negative and increasing along any sequence $r_k \downarrow 0$ for $k$ large enough. 
For the second and the third term in the right-hand side of \eqref{pf:prop:E:on:P:10}, we observe from $\overline \rho \in C^2 (0,1)$ and $\overline \rho (t_0) = 0$ that
\begin{align*}
  \frac{ \overline \rho (t_0 + r) + \overline \rho (t_0 - r) }{ r^2 }
  = \frac{ \overline \rho (t_0 + r) - 2 \overline \rho (t_0) + \overline \rho (t_0 - r) }{ r^2 }
  &\xto{ r\to 0 } \overline \rho'' (t_0), \\  
  \frac{ \overline \rho' (t_0 + r) - \overline \rho' (t_0 - r) }{2r}
  &\xto{ r\to 0 } \overline \rho'' (t_0).
\end{align*}
Since both $r^2 V' (r)$ and $r V(r)$ converge to $0$ as $r \to 0$, we conclude that the second and the third term in the right-hand side of \eqref{pf:prop:E:on:P:10} also converge to $0$ as $r \to 0$. Lastly, since $V \in L^1 (-1,1)$ and $\overline \rho'' \in C(0,1)$, we obtain that the fourth term in the right-hand side of \eqref{pf:prop:E:on:P:10} converges to $0$ as $r \to 0$.

In conclusion, by taking the limit $r \to 0$ in \eqref{pf:prop:E:on:P:10}, we obtain that the right-hand side is positive, and the contradiction is reached. Hence, we conclude that $\overline \rho (t) > 0$ for any $t \in (0,1)$.
\end{proof}

\begin{lem}[Problem \ref{prob:wSIE} has a unique solution] \label{lem:Un:P13}
  Let $a$, $\Vreg$ and $\Vext$ be as in Theorem \ref{t}. Then, Problem \ref{prob:wSIE} has a unique solution.
\end{lem}

\begin{proof}
The existence is covered by Lemma \ref{lem:t}. To show the uniqueness, let $(\rho_i, C_i)$ for $i = 1, 2$ be two solutions to Problem \ref{prob:wSIE}. Then 
\begin{equation} \label{pf:prop:E:on:P:7}
  V*(\rho_1 - \rho_2) = C_1 - C_2 \quad \text{a.e.~on } (0,1).
\end{equation}
To show that $\rho_1 = \rho_2$, we consider the auxiliary energy $\Et (\rho) := \tfrac12 \| \rho \|_V^2$. We observe that $\Et$ is of the same form as \eqref{for:defn:E}, and that the assumptions \eqref{for:V:props}, \eqref{for:Vext:assns} and \eqref{VextV:bd} are all satisfied. Hence, Proposition \ref{p:E} provides the unique minimiser $\rhot \in \mathcal P ([0,1]) \cap H_V (0,1)$ of $\Et$, and Lemma \ref{lem:t} implies that $V * \rhot = \tilde C$ on $(0,1)$ where $\tilde C > 0$. Setting $\alpha := (C_1 - C_2)/\tilde C$, we obtain that $V * (\alpha \rhot) = C_1 - C_2$ on $(0,1)$. Then, by \eqref{pf:prop:E:on:P:7} and Lemma \ref{lem:Vastxi:zero:means:xi:zero}, we find that $\alpha \rhot = \rho_1 - \rho_2$. Hence $\alpha = \int_0^1 \alpha \, d \tilde \rho = \int_0^1 d(\rho_1 - \rho_2) = 0$, and thus $C_1 = C_2$. By \eqref{pf:prop:E:on:P:7} and Lemma \ref{lem:Vastxi:zero:means:xi:zero} we then also have $\rho_1 = \rho_2$.
\end{proof}

\section{The extended version of Theorem \ref{t:R}: statement and proof}
\label{s:FBVP}

In this section we state (in Section \ref{ss:FBVP:R}) and prove (in Section \ref{ss:pf:R}) the extended version of Theorem \ref{t:R} given by Theorem \ref{t:R:extended}. Since we will often translate between Cases 1 and 2 (see Sections \ref{ss:Case1} and \ref{ss:extn:R}) by an affine change of variables, we alter the notation in Case 2. Instead of \eqref{for:defn:ER}, we set 
\begin{equation} \label{for:defn:ERt}
   \Et (\rho) = \frac12 \iint_{\R \times \R} \Vt (\tilde t - \tilde s) \, d (\rho \otimes \rho) (\tilde s, \tilde t) + \int_\R \Vextt (\tilde t) \, d\rho (\tilde t), 
\end{equation}
and denote its minimiser by $\rhot$.

\subsection{Theorem \ref{t:R:extended}: the extended version of Theorem \ref{t:R}}
\label{ss:FBVP:R}

Next we state the counterpart of Problems \ref{prob:minz}--\ref{prob:wSIE} in the setting on $\R$. For given $0 \leq a < 1$, let $\Vt$, $\Vregt$ and $\Vextt$ be as in \eqref{for:V:on:R:props} and \eqref{for:Vext:on:R:assns}. 

\begin{prob}[Minimisation] \label{prob:minz:R}
  Find the minimiser of $\Et$ (defined in \eqref{for:defn:ERt}) in $\mathcal P (\R)$.
\end{prob}

We note that we cannot construct the space $H_{\Vt} (\R)$ as in \S \ref{ssec:Vip}, because the compact support condition in \eqref{for:V:props:R} is not satisfied. We will side-step the construction of $H_{\Vt} (\R)$ by first showing that the solution $\rhot$ to Problem \ref{prob:minz:R} has finite support. This property allows us to modify the tails of $\Vt$ such that \eqref{for:V:props:R} is satisfied without losing the minimality of $\rhot$. Then, by Proposition \ref{prop:HV} we can identify the related Hilbert space by $H^{-(1-a)/2} (\R)$, which does not depend on the choice of the regularisation of the tails of $\Vt$. This motivates the following problem:

\begin{prob}[Variational inequality] \label{prob:VI:R}
  Find $\rho \in \mathcal P (\R)$ such that
  \begin{equation} \label{for:thm:R:EL:weak}
  \int_{\R} \tilde h_{\rho} \, d \mu 
  \geq \int_{\R} \tilde h_{\rho} d \rho
  \quad \text{for all } \mu \in \mathcal P (\R) \cap H^{-(1-a)/2} (\R),
\end{equation}
where $\tilde h_\rho$ is as in \eqref{hrho:t}.
\end{prob}       

To state Problem \ref{prob:wSIE:R}, we recall \eqref{int01rho} for the extension of the integral to distributions $\rho \in H^{-(1-a)/2} (t_1, t_2)$, Definition \ref{d:convo} for the definition of $\tilde h_\rho$ for such $\rho$, and that $f(t \pm)$ denote the one-sided limits of $f$ at $t$.

\begin{prob}[Weakly singular integral equation with free boundary] \label{prob:wSIE:R}
Find the solution \\
$(\rho, t_1, t_2, C)$ with $C \in \R$, $-\infty < t_1 < t_2 < \infty$, $\rho \in H^{-(1-a)/2} (t_1, t_2)$ and $\int_{t_1}^{t_2} \rho = 1$ to
\begin{equation} \label{for:FBVP}
  \left\{ \begin{aligned}
     &\tilde h_{\rho} (\tilde t) = C
     \quad \text{a.e.~on } (t_1, t_2), \\
     &\tilde h_{\rho}' (t_1-) \leq 0 \leq \tilde h_{\rho}' (t_2+).
  \end{aligned} \right.
\end{equation}  
\end{prob}

We note that it is not restrictive to assume the support of $\rhot$ in Problem \ref{prob:wSIE:R} to be finite. Indeed, if $\rhot \in H^{-(1-a)/2} (\R)$ satisfies \eqref{for:FBVP} (without the boundary conditions), then we observe from the following inclusion of levelsets,
$$
  \supp \rho 
  \subset \overline{ \{ \tilde h_\rho = C \} }
  \subset \overline{ \{\Vextt \leq C\} }, 
$$
that $\supp \rho$ is bounded due to the linear growth of $\Vextt$ (see \eqref{for:Vext:on:R:assns}).

Since we prove the solution $\rhot$ of Problem \ref{prob:wSIE:R} to be in $C_0^\beta ([t_1, t_2])$ for some $\beta > 0$, we can also seek classical solutions to Problem \ref{prob:wSIE:R}. This solution concept coincides with that in \cite{Muskhelishvili53}. The benefit of working with classical solutions is that the boundary conditions in \eqref{for:FBVP} turn into homogeneous Dirichlet boundary conditions (which are already included in the H\"older space $C_0^\beta ([t_1, t_2])$). 

\begin{prob}[Classical weakly singular integral equation with free boundary] \label{prob:sSIE:R}
Find the solution $(\rho, t_1, t_2, C)$ with $C \in \R$, $-\infty < t_1 < t_2 < \infty$, $\rho \in C_0^\beta ([t_1, t_2])$ for some $\beta > 0$ and $\int_{t_1}^{t_2} \rho = 1$ to
\begin{equation} \label{for:FBVP:classical}
  \tilde h_{\rho} (\tilde t) = C
     \quad \text{on } [t_1, t_2].
\end{equation}  
\end{prob}

With Problems \ref{prob:minz:R}--\ref{prob:sSIE:R} being defined, we are finally ready to state the extended version of Theorem \ref{t:R}:

\begin{thm}[Equivalence of Problems \ref{prob:minz:R}--\ref{prob:sSIE:R} and properties of the solution] \label{t:R:extended}
Let $a$, $\Vt$, \\ $\Vregt$, $\Vextt$, $\plog$ and $\ell$ be as in Theorem \ref{t:R}. Then, all four Problems \ref{prob:minz:R}--\ref{prob:sSIE:R} have a unique solution, and all these solutions are equal. The solution $(\rhot, t_1, t_2, \Ct)$ to these problems satisfies all statements listed in Theorem \ref{t:R}. 
Moreover, $(\rhot, t_1, t_2, \Ct)$ satisfies the implicit relation \eqref{IF:R:impl}.
\end{thm}
 
\subsection{Proof of Theorem \ref{t:R:extended}} 
\label{ss:pf:R} 
 
The proof of Theorem \ref{t:R:extended} is divided in 3 parts:
\begin{enumerate}
  \item Problem \ref{prob:minz:R} has a unique solution $\rhot$, which turns out to have bounded support and no atoms (Proposition \ref{prop:P61:R});
  \item Theorem \ref{t:R} holds, and $\rhot$ is a solution to Problems \ref{prob:VI:R}--\ref{prob:sSIE:R} and satisfies an implicit relation similar to \eqref{soln:form:Abs} (Lemma \ref{lem:t:R});
  \item Problems \ref{prob:VI:R}--\ref{prob:sSIE:R} have a unique solution (Lemma \ref{lem:Un:PR}).
\end{enumerate}

The proof of Proposition \ref{prop:P61:R} is quite standard; its result holds for much weaker assumptions on $\Vregt$ and $\Vextt$ than those given by \eqref{for:V:on:R:props} and \eqref{for:Vext:on:R:assns}. Lemma \ref{lem:Un:PR} is proven at the end, because it relies on Lemma \ref{lem:t:R}.

The proof of the main part, Lemma \ref{lem:t:R}, starts from Proposition \ref{prop:P61:R}, which provides the unique solution $\rhot \in \mathcal P (\R)$ to Problem \ref{prob:minz:R}. In five subsequent steps, $\rhot$ is proven to satisfy all properties listed in Theorem \ref{t:R}. In Step 1, we use that $\supp \rhot$ is bounded with endpoints $-\infty < t_1 < t_2 < \infty$ to construct an affine mapping $\mathcal T$ such that the minimum and maximum of $\supp (\mathcal T_\# \rhot)$ are $0$ and $1$. In a separate lemma (Lemma \ref{l:t:extn}), we show that this property is a sufficient substitute for the assumption in \eqref{VextV:bd} under which Theorem \ref{t} applies to $\overline \rho := \mathcal T_\# \rhot$ for shifted versions of $\Vregt$ and $\Vextt$.  

In Step 2, the main step, we apply Theorem \ref{t} to $\overline \rho$. The requirements to apply this Theorem are the main motivation for assumptions \eqref{for:V:on:R:props} and \eqref{for:Vext:on:R:assns}. Then, pulling $\overline \rho$ back along $\mathcal T$, we obtain from Theorem \ref{t} most of the regularity properties stated in Theorem \ref{t:R} on $\overline \rho$. The properties which require further motivation, are that \eqref{for:thm:R:EL:weak} holds on $\R$ instead of $[t_1, t_2]$ (Step 3), the boundary conditions in Problems \ref{prob:wSIE:R} and \ref{prob:sSIE:R} (Steps 4 and 5), and the further expansion of $\rhot$ around $t_1$ and $t_2$ (Step 5). 
\smallskip

We proceed to carry out the plan of the proof. We establish part 1 for weaker assumptions on $\Vt$ and $\Vextt$:

\begin{prop}[Problem \ref{prob:minz:R} has a unique solution] \label{prop:P61:R}
Let $\Vt \in L^1(\R)$ be such that $\Vt$ is even, convex outside of $0$ and $\liminf_{\tilde t \to 0} V(\tilde t) = \infty$. Let $\Vextt \in C(\R)$ satisfy the growth condition $\liminf_{\tilde t \to \pm \infty} \Vextt(\tilde t) = \infty$. Then, Problem \ref{prob:minz:R} has a unique solution $\rhot \in \mathcal P(\R)$. Moreover, $\supp \rhot$ is bounded, and $\rhot$ has no atoms.
\end{prop}

\begin{proof}
We observe that $\Vt \geq 0$. Then, we infer that $\Et$ is well-defined on $\mathcal P(\R)$ with values in $[\min_\R \Vextt, \infty]$. Because of the growth condition on $\Vextt$, any minimising sequence is tight, and therefore weakly converging along a subsequence to some $\rhot \in \mathcal P(\R)$. Since $\Vt$ is lower semi-continuous, $\rhot$ is a minimiser of $\Et$.

We prove uniqueness of the minimiser by relying on the claim that $\Et$ is strictly convexity. The boundedness of  $\supp \rhot$ follows from a straightforward energy estimate that relies on the growth of $\Vextt$ (see, e.g., \cite[\S 2.2]{MoraRondiScardia16ArXiv}), and the absence of atoms is guaranteed by the singularity of $\Vt$ at $0$.

Finally, we prove that $\Et$ is strictly convex. Repeating the computation in \eqref{Vhat:nonneg}, we obtain that $\mathcal F \Vt > 0$ on $\R$. Then, repeating the proof of Lemma \ref{l:Vip:L2}, we obtain that $(f,g)_{\Vt} := \int_{\R} (\Vt*f)g$ defines an inner product on $L^2(\R)$. From this, we conclude that $\Et$ is strictly convex. 
\end{proof}

We proceed with the second part of the proof of Theorem \ref{t:R:extended}. In preparation for Lemma \ref{lem:t:R}, we prove the following lemma.

\begin{lem}[Alternative assumption for Theorem \ref{t}] \label{l:t:extn}
Given the setting of Proposition \ref{p:E}, let $\overline \rho$ be the solution to Problems \ref{prob:minz} and \ref{prob:VI}. If $\{0,1\} \subset \supp \overline \rho$, then Theorem \ref{t} also holds when \eqref{VextV:bd} is not satisfied.
\end{lem}

\begin{proof}
We note that \eqref{VextV:bd} is solely used in \eqref{pf:prop:E:on:P:43} in Step 1 of the proof of Theorem \ref{t}, which treats those intervals $(r,s) \subset A = \{ h_{\overline \rho} > m \}$ in Step 1 for which either $r = 0$ or $s = 1$. Since $\overline \rho$ contains no atoms and $\{0,1\} \in \supp \overline \rho$, we obtain that $\overline \rho ((r,s)) > 0$ if $r = 0$ or $s = 1$. This contradicts with \eqref{pf:prop:E:on:P:2}, and thus the argument in the third out of the three cases containing \eqref{pf:prop:E:on:P:43} can be omitted.
\end{proof}

In preparation for stating Lemma \ref{lem:t:R}, we introduce $\tilde {\mathcal C}_a$ as the counterpart of ${\mathcal C}_a$ acting on functions defined on some bounded interval $(t_1, t_2)$. We set
\begin{align*} 
  \phit_a (\tilde t)
  := \big[ (\tilde t - t_1)(t_2 - \tilde t) \big]^{ \tfrac{1-a}2 },
\end{align*}
and recall the operators $S_{t_1}^{t_2}$ and $D_{t_1}^{1-a}$ from \eqref{for:defn:S12} and \eqref{Dat1}. The linear operator $\tilde {\mathcal C}_a$ is given by
\begin{equation*} 
  \tilde{ \mathcal C}_a f := \left\{ \begin{aligned}
    & \frac{\Gamma(a)}{\pi^2}  \cos^2 \Big( \frac{a\pi}2 \Big) \Big( \frac1{\phit_a} S_{t_1}^{t_2} ( \phit_a D_{t_1}^{1-a} f ) + \pi \tan \Big( \frac{a\pi}2 \Big) D_{t_1}^{1-a} f \Big)
    &&\text{if } 0<a<1 \\
    &\frac1{\pi^2 \phit_0} \bigg[ S_{t_1}^{t_2} ( \phit_0 f' )  + \frac{1}{2 \log 2} \bigg( \int_{t_1}^{t_2} \frac{\tilde f_{\rhot}(\tilde s) }{ \phit_0 (\tilde s) } \, d\tilde s + \pi \log (t_2 - t_1) \bigg) \bigg]
    &&\text{if } a=0,
  \end{aligned} \right.
\end{equation*}  
where $f : (t_1, t_2) \to \R$.

\begin{lem}[Regularity properties of Theorem \ref{t:R:extended}] \label{lem:t:R}
Let $a$, $\Vregt$, $\Vextt$ and $\ell$ be as in Theorem \ref{t:R}. Let $\rhot$ be given by Proposition \ref{prop:P61:R}, and set $t_1 := \min (\supp \rhot)$ and $t_2 = \max (\supp \rhot)$. Then, Theorem \ref{t:R} holds, and $(\rhot, t_1, t_2, \Ct)$ with
\begin{equation*}
   \Ct := \iint_{\R \times \R} \Vt (\tilde t - \tilde s) \, d (\rho \otimes \rho) (\tilde s, \tilde t) + \int_\R \Vextt (\tilde t) \, d\rho (\tilde t) 
\end{equation*}
is a solution to Problems \ref{prob:VI:R}--\ref{prob:sSIE:R}, which moreover satisfies
\begin{equation} \label{IF:R:impl}
  \rhot = \tilde{\mathcal C}_a \tilde f_{\rhot}
    \quad \text{on $(t_1, t_2)$,} \quad \text{where} \quad
    \tilde f_{\rhot} := \Ct - \Vregt * \rhot - \Vextt.
\end{equation}
\end{lem}

\begin{proof} 
The proof is divided in five steps.
\smallskip

\textit{Step 1: translation to Theorem \ref{t}}. 
Let $L := t_2 - t_1$ and $\mathcal T (\tilde t) := (\tilde t - t_1)/L$ be the affine map which maps $[t_1, t_2]$ to $[0,1]$.

\textit{Step 1a: $0 < a < 1$}. Using $\mathcal T$, we rewrite
\begin{multline} \label{Et:p3a1}
  \Et (\rhot) 
  = \frac12 \int_{t_1}^{t_2} \int_{t_1}^{t_2} \Vt (\tilde t - \tilde s) \, d \rhot (\tilde s) \, d \rhot (\tilde t) + \int_{t_1}^{t_2} \Vextt(\tilde t) \, d\rhot(\tilde t) \\
  = \frac12 \int_0^1 \int_0^1 \Vt \big( L (t-s) \big) \, d (\mathcal T_\# \rhot) (s) \, d (\mathcal T_\# \rhot) (t) + \int_0^1 \Vextt \big( \mathcal T^{-1}(t) \big) \, d(\mathcal T_\# \rhot)(t).
\end{multline}
Since 
$\Vt ( L t )
  = V_a ( L t ) + \Vregt ( L t ) 
  = L^{-a} V_a (t) + \Vregt ( L t ) 
$,
we can further rewrite \eqref{Et:p3a1} as 
\begin{multline} \label{Et:p3a2}
  \Et (\rhot) 
  = L^{-a} \bigg[ \frac12 \int_0^1 \int_0^1 \big[ V_a (t-s) + L^a \Vregt ( L (t-s) ) \big] \, d (\mathcal T_\# \rhot) (s) \, d (\mathcal T_\# \rhot) (t) \\
    + \int_0^1 L^a \Vextt ( L t + t_1 ) \, d(\mathcal T_\# \rhot)(t) \bigg] 
  = L^{-a} \big[ E(\mathcal T_\# \rhot) + c \big],
\end{multline}
where $E$ is as in \eqref{for:defn:E} with 
\begin{equation} \label{VrVe:trf}
  \Vreg(t) := L^a \Vregt ( L t ) - 2c, \quad
  \Vext(t) := L^a \Vextt \big( L t + t_1 \big),
\end{equation}
and $c \in \R$ is a constant such that $V(1) = -V'(1)$. 

\textit{Step 1b: $a = 0$}. We argue analogously to Step 3a. 
Since 
$\Vt ( L t )
  = V_0 (t) - \log L + \Vregt ( L t ) 
$,
the corresponding computation in \eqref{Et:p3a1} and \eqref{Et:p3a2} yields
\begin{equation*} 
  \Et (\rhot) 
  = E(\mathcal T_\# \rhot) + c - \tfrac12 \log L,
\end{equation*}
where $E$ is as in \eqref{for:defn:E} with 
\begin{equation*}
  \Vreg(t) := \Vregt ( L t ) - 2c
  \quad \text{and} \quad
  \Vext(t) := \Vextt \big( L t + t_1 \big),
\end{equation*}
and $c \in \R$ is a constant such that $V(1) = -V'(1)$. Since these definitions are consistent with putting $a = 0$ in \eqref{Et:p3a2} and \eqref{VrVe:trf} (except for an additive constant to $\Et$), we refer in the remainder of the proof to \eqref{Et:p3a2} and \eqref{VrVe:trf} for all $0 \leq a < 1$.
\smallskip

\textit{Step 2: Properties of $\rhot$ for given $- \infty < t_1 < t_2 < \infty$.} 
Since $\rhot$ is the unique minimiser of $\Et$, it follows from \eqref{Et:p3a2} that $E$ also has a unique minimiser, which is given by $\overline \rho := \mathcal T_\# \rhot$. It is easy to see that the transformations in \eqref{VrVe:trf} transfer the properties of $\Vt$, $\Vregt$ and $\Vextt$ in \eqref{for:V:on:R:props} and \eqref{for:Vext:on:R:assns} to those listed in \eqref{for:V:props} and \eqref{for:Vext:assns}. Moreover, by Step 1 we obtain that $0, 1 \in \supp \overline \rho$. Hence, Lemma \ref{l:t:extn} implies that Theorem \ref{t} applies to $\overline \rho$. Pulling back along $\mathcal T$ (i.e.~$\rhot = (\mathcal T^{-1})_\# \overline \rho$), all the statements on $\overline \rho$ given in Theorem \ref{t} transfer to $\rhot$. Hence, we obtain that $\supp \rhot = [t_1, t_2]$, $\rhot$ has an integrable, continuous representative on $(t_1, t_2)$, and that $\rhot$ satisfies the local Sobolev regularity in \eqref{for:thm:R:regy}. Furthermore, we obtain that $(\rhot, t_1, t_2, \Ct)$ satisfies \eqref{for:FBVP:classical} and \eqref{IF:R:impl}. Lastly, if $\ell \geq 3$, then $\rhot > 0$ on $(t_1, t_2)$. 
%
%
%
\smallskip

\textit{Step 3: $\rhot$ satisfies Problem \ref{prob:VI:R}.}
We use a different dilation map than $\mathcal T$ on a larger interval $(r_1, r_2) \supset (t_1, t_2)$ whose endpoints satisfy
\begin{equation} \label{choice:s1s2}
  \Vextt(r_1) \wedge \Vextt(r_2) \geq \Ct
  \quad \text{and} \quad
  \Vextt'(r_1) < 0 < \Vextt'(r_2).
\end{equation}
We set $\mathcal R (\tilde t) := (\tilde t - r_1)/(r_2 - r_1)$ as the dilation map. Then, as in Step 3, we introduce
\begin{equation*} 
  \Et (\rhot) 
  = (r_2 - r_1)^{-a} \big( E^* (\mathcal R_\# \rhot) + c \big),
\end{equation*}
where $E^*$ is as in \eqref{for:defn:E} with 
\begin{equation*}
  \Vreg^*(t) := \Vregt ( (r_2 - r_1) t ) - 2c
  \quad \text{and} \quad
  \Vext^*(t) := \Vextt \big( (r_2 - r_1) t + r_1 \big),
\end{equation*}
and $c \in \R$ is such that $V^*(1) = - (V^*)'(1)$.
Again, we find that $\rho^* := \mathcal R_\# \rhot$ is the unique minimiser of $E^*$ and that $\Vreg^*$ and $\Vext^*$ satisfy \eqref{for:V:props} and \eqref{for:Vext:assns}. Hence, Proposition \ref{p:E} applies to $E^*$, and thus $\rho^*$ satisfies the related Problem \ref{prob:VI}. Hence, $\rhot = \mathcal R_\#^{-1} \rho^*$ satisfies
\begin{equation} \label{wEL:s1s2}
  \int_{r_1}^{r_2} \tilde h_{\rhot} \, d \mu 
  \geq \int_{\R} \tilde h_{\rhot} d \rhot
  \quad \text{for all } \mu \in \mathcal P ([r_1, r_2]) \cap H^{-(1-a)/2} (r_1, r_2).
\end{equation}
Since $\supp \rhot = [t_1, t_2] \subset (r_1, r_2)$ and $\tilde h_{\rhot} = \Ct$ on $[t_1, t_2]$, we obtain from \eqref{wEL:s1s2} that $\tilde h_{\rhot} \geq \Ct$ a.e.~on $[r_1, r_2]$. Moreover, by \eqref{choice:s1s2} we have for all $\tilde t \in [r_1, r_2]^c$ that
\begin{equation*}
  \tilde h_{\rhot} (\tilde t)
  \geq \Vextt (\tilde t)
  > \Vextt(r_1) \wedge \Vextt(r_2) 
  \geq \Ct.
\end{equation*}
Hence, $\rhot$ satisfies Problem \ref{prob:VI:R}.

From here, we denote by $\rhot$ the continuous representative.
\smallskip

\textit{Step 4: $(\rhot, t_1, t_2, \Ct)$ satisfies Problem \ref{prob:wSIE:R}}. By Step 2 it is enough to show that $\rhot$ satisfies the boundary conditions in Problem \ref{prob:wSIE:R} at $t_1$ and $t_2$. We focus on proving $\tilde h_{\rhot}' (t_1-) \leq 0$; the boundary condition at $t_2$ follows from a similar argument.

By the argument in Step 1 of the proof of Theorem \ref{t} it follows that $\tilde h_{\rhot} \in C(\R) \cup C^1((-\infty, t_1))$ and $\tilde h_{\rhot}'' \geq 0$ on $(-\infty, t_1)$. Moreover, by the Monotone Convergence Theorem, we obtain that $\tilde h_{\rhot}' (t_1-)$ is well-defined as a value in $(-\infty, \infty]$. Since $\tilde h_{\rhot} \geq \Ct$ on $\R$ and $\tilde h_{\rhot} (t_1) = \Ct$, it must hold that $\tilde h_{\rhot}' (t_1-) \leq 0$. 
\smallskip

\textit{Step 5: $\rhot$ satisfies the H\"older condition in \eqref{for:thm:R:regy} and \eqref{for:thm:R:expa}}. Let $E$, $V$, $\Vreg$, $\Vext$ and $\overline \rho = \mathcal T_\# \rhot$ as in Step 1. We focus on the expansion of $\overline \rho(t)$ around $t=0$. In addition to the properties of $\overline \rho$ stated in Theorem \ref{t}, we obtain from Steps 2 and 4 that $h_{\overline \rho}' (0-) \leq 0$, which implies
\begin{equation*} 
  \liminf_{t \downarrow 0} \overline \rho (t) = 0.
\end{equation*}
Then, the constants $C_i$ in \eqref{for:prop:E:on:P:regy:Hol} must be $0$, and thus $T^\# \overline \rho$ satisfies the H\"older condition in \eqref{for:thm:R:regy}.

For the proof of \eqref{for:thm:R:expa}, we separate two cases:

\textit{Step 5a: $0 < a < 1$}. We use a bootstrap argument. Since $\overline \rho \in C_0^a([0,1])$ and $\Vreg \in W^{2,1}(0,1)$, we have $\Vreg * \overline \rho \in C^2([0,1])$. Hence, we obtain $f_{\overline \rho} \in C^2 ([0,1])$ from \eqref{forho} and the additional regularity imposed on $\Vextt$. Then, repeating the argument in Step 6a of Lemma \ref{lem:t} with a second order expansion in \eqref{p:Caf:Td} and using Proposition \ref{prop:props:S:Hol}.\ref{prop:props:S:Hol:bdy:1} for the regularity of $R_2$, we obtain that 
\begin{align*}
 R &\in C^2([0,1]) 
 &&\text{satisfies} 
 & |R(t)| &\leq Ct^2, \\
 R_1 &\in C^{1,a}([0,1]) 
 &&\text{satisfies} 
 & |R_1(t)| &\leq Ct^{1+a}, \\
 R_2 &\in C^{1,a}([0,1-\varepsilon]) 
 &&\text{satisfies} 
 & |R_2(t)| &\leq Ct^{1+a}, \\
 R_3 &\in C^{1,a}([0,1-\varepsilon]) 
 &&\text{satisfies} 
 & |R_3(t)| &\leq Ct^{1+a},
\end{align*}
and the equivalent of \eqref{p:Caf:Td1} becomes
\begin{equation*} 
  \overline \rho (t) = \mathcal C_a f_{\overline \rho} (t) = \frac{C + C' t}{\phi_a(t)} + R_3(t).  
\end{equation*}
Since $\overline \rho (0) = 0$, we obtain $C = 0$, and \eqref{for:thm:R:expa} follows. 

\textit{Step 5b: $a = 0$}. The proof is analogous to that of Step 5a; the only difference is that from $\overline \rho \in C_0^{\alpha_0}([0,1])$, $\Vreg \in W^{2, 1}(0,1)$ and $\Vext \in C^{2, \alpha_0}([0,1])$ we obtain $f_{\overline \rho}' \in C^{1, \alpha_0}([0,1])$.
\end{proof}

Finally, we focus on the third of the three parts of the proof of Theorem \ref{t:R:extended}:
 
\begin{lem}[Problems \ref{prob:VI:R}--\ref{prob:sSIE:R} have a unique solution] \label{lem:Un:PR}
Let $a \in [0, 1)$ and assume that $\Vregt$ and $\Vextt$ satisfy \eqref{for:V:on:R:props} and \eqref{for:Vext:on:R:assns}. Then, all Problems \ref{prob:VI:R}, \ref{prob:wSIE:R} and \ref{prob:sSIE:R} have a unique solution.
\end{lem} 

\begin{proof} 
Let $\rhot$ be the solution to Problem \ref{prob:minz:R}. In part 2 of the proof of Theorem \ref{t:R:extended} we have proven that $\rhot$ satisfies all the asserted properties of Theorem \ref{t:R:extended}, including the fact that it is a solution to Problems \ref{prob:VI:R}--\ref{prob:sSIE:R}. It remains to show the uniqueness of all three problems.
\smallskip

\textit{Step 1: Problem \ref{prob:VI:R} has a unique solution}. Let $\rho \in \mathcal P (\R) \cap H^{-(1-a)/2} (\R)$ be a solution. We consider a sequence of approximating energies
\begin{equation*}
  \Et_k (\rho) = \frac12 \int_{\R} \Vt_k * \rho \, d \rho + \int_\R \Vextt \, d\rho,
  \quad k \in \N_+,
\end{equation*}
where the potentials $\Vt_k : \R \to [0, \infty]$ are chosen such that $\Vt_k$ satisfies \eqref{for:V:props:R}, $\Vt_k |_{(-k,k)} = \Vt |_{(-k,k)}$, and $\Vt_k \leq \Vt_{k+1} \leq \Vt$ for all $k \in \N_+$. By the Monotone Convergence Theorem, it is easy to see that $\Et_k (\rho) \to \Et (\rho)$ as $k \to \infty$.

Using \eqref{for:thm:R:EL:weak} with $\mu = \rhot$, we obtain by the Monotone Convergence Theorem that
\begin{align} \label{pf:lUn1}
  0 
  \leq \int_\R \tilde h_{\rho} \, d(\rhot - \rho)
  = \lim_{k \to \infty} \int_\R (\Vt_k * \rho + \Vextt) \, d(\rhot - \rho).
\end{align}
Since $\Vt_k$ satisfies \eqref{for:V:props:R} for any $k \in \N_+$, we can apply Corollary \ref{cor:props:HV} to $\Vt_k$ to obtain
\begin{multline*}
  \int_\R (\Vt_k * \rho + \Vextt) \, d(\rhot - \rho)
  = (\rho, \rhot - \rho)_{\Vt_k} + \int_\R \Vextt \, d(\rhot - \rho) \\
  \leq \frac12 (\rhot, \rhot)_{\Vt_k} - \frac12 (\rho, \rho)_{\Vt_k} + \int_\R \Vextt \, d(\rhot - \rho) 
  = \Et_k (\rhot) - \Et_k (\rho) 
  \xto{k \to \infty} \Et (\rhot) - \Et(\rho).
\end{multline*}
Together with \eqref{pf:lUn1}, we obtain from Proposition \ref{prop:P61:R} that $\rho = \rhot$.
\smallskip

\textit{Step 2: Problem \ref{prob:wSIE:R} has a unique solution}.
Let $(\rho, s_1, s_2, C)$ be a solution to Problem \ref{prob:wSIE:R}. We prove that $\rho$ is a solution to Problem \ref{prob:VI:R}, and conclude by the statement of Step 1. We define the affine coordinate transformation $\mathcal R (\tilde t) := (\tilde t - s_1)/(s_2 - s_1)$, and obtain, analogously to Steps 1 and 2 of the proof of Lemma \ref{lem:t:R}, that $(\mathcal R_\# \rho, \, (s_2 - s_1)^a C - 2c)$ 
equals the solution $(\overline \rho, \overline C)$ to Problem \ref{prob:wSIE} for the shifted potentials $V$ and $\Vext$. In particular, from $\mathcal R_\# \rho = \overline \rho \geq 0$ we obtain $\rho \geq 0$. Similar to \eqref{pf:prop:E:on:P:3}, we then deduce that $\tilde h_\rho'' \geq 0$ on $[s_1, s_2]^c$. Together with the boundary condition $\tilde h_{\rho}' (s_1-) \leq 0 \leq \tilde h_{\rho}' (s_2+)$, we conclude that $\tilde h_\rho \geq C$ on $\R$, and thus $\rho$ is a solution to Problem \ref{prob:VI:R}.
\smallskip

\textit{Step 3: Problem \ref{prob:sSIE:R} has a unique solution}. Let $(\rho, s_1, s_2, C)$ be a solution to Problem \ref{prob:sSIE:R}. We reason similarly as in Step 2. The only difference is that we can rely no more on the boundary conditions in Problem \ref{prob:wSIE:R} to prove that $\tilde h_\rho \geq C$ on $\R$. Instead, we obtain from $\rho \in C_0^a ([s_1, s_2])$ that $\tilde h_\rho \in C(\R)$. Moreover, from $\rho \geq 0$ with $\supp \rho \subset [s_1, s_2]$ we obtain, similar to the proof of Step 1 in Theorem \ref{t}, that $\tilde h_\rho' \leq 0$ on $(-\infty, s_1)$. Hence, $\tilde h_\rho \geq C$ on $(-\infty, s_1)$. A similar argument shows that $\tilde h_\rho \geq C$ on $(s_2, \infty)$, and thus $\tilde h_\rho \geq C$ on $\R$. Hence, $\rho$ satisfies Problem \ref{prob:VI:R}.
\end{proof}

\section{Applications of Theorem \ref{t} and Theorem \ref{t:R:extended}}
\label{s:appl}

In this section we apply Theorem \ref{t} and Theorem \ref{t:R:extended} (i.e., the extended version of Theorem \ref{t:R}) to improve previous results in the literature, including all three applications (A1)--(A3) mentioned in Section \ref{ss:applAi}. We consider the general form of the energy $E$ as in \eqref{for:defn:ER} in which $\Vext$ is allowed to jump to $+\infty$ at $- \infty \leq s_1 < s_2 \leq \infty$, and recall the discussion in Section \ref{ss:disc} on how to extend Theorems \ref{t} and \ref{t:R:extended} to cover this general setting. In \S \ref{ss:appl:expl} we compute explicitly the minimiser ${\overline \rho} \in \mathcal P([s_1, s_2])$ in several special cases, including the one from (A3). In \S \ref{ss:impr:reg} we strengthen the results of \cite{SandierSerfaty15} and \cite{GeersPeerlingsPeletierScardia13} to lift the previous limitations in applications (A1) and (A2).

\subsection{Explicit formulas for ${\overline \rho}$ in the case $\Vreg = 0$}
\label{ss:appl:expl}

When $\Vreg = 0$, \eqref{IF:R:impl} provides an explicit expression for ${\overline \rho}$. In this section we simplify this expression for several examples in the literature, including that of (A3).

The case $a = 0$ is studied in \cite{HeadLouat55} in the context of dislocation densities. For $[s_1, s_2] = [0,1]$ and $\Vext(t) = 0$ for $t \in [0,1]$ it is found that $[t_1, t_2] := \supp {\overline \rho} = [0,1]$ and
\begin{equation*} 
  {\overline \rho} (t) = \frac1{ \pi \sqrt{ t(1-t) } }.
\end{equation*}
Moreover, for $[s_1, s_2] = [0, \infty)$ and $\Vext (t) = t$ for $t \geq 0$ it is found that $[t_1, t_2] = [0,2]$ with
\begin{equation*} 
  {\overline \rho} (t) = \frac1{ \pi  } \sqrt{ \frac{2-t}t }.
\end{equation*}
Both formulas can also be computed from \eqref{IF:R:impl}; we omit the details. For treatment of the setting $[s_1, s_2] = \R$ and $\Vext(t) = t^\alpha$ we refer to \cite[\S IV.5]{SaffTotik97}.

In the case $0 < a < 1$, the easiest setting is given by $[s_1, s_2] = [0,1]$ and $\Vext(t) = 0$ for $t \in [0,1]$. From the structure of $E$ it follows a priori that $[t_1, t_2] = [0,1]$. Then, from \eqref{IF:R:impl} with $f_{{\overline \rho}} (t) = {\overline C}$, we apply Proposition \ref{prop:Ca} to obtain
\begin{equation*} 
  {\overline \rho} (t) = \frac{{\overline C} }\pi \cos \Big( \frac{a\pi}2 \Big) \big[ t(1-t) \big]^{ - \tfrac{1-a}2 }.
\end{equation*}
To find ${\overline C}$, we use \eqref{Zwil} to compute
\begin{equation*}
  1
  = \int_0^1 {\overline \rho}(t) \, dt
  = \frac{{\overline C} }\pi \cos \Big( \frac{a\pi}2 \Big) \int_0^1 \big[ t(1-t) \big]^{ - \tfrac{1-a}2 } \, dt
  = \frac{{\overline C} }\pi \cos \Big( \frac{a\pi}2 \Big) \frac{ \Gamma (\tfrac{1+a}2)^2 }{ a \Gamma(a) }.
\end{equation*}
We conclude that
\begin{equation*} 
  {\overline \rho} (t) = \frac{ a \Gamma(a) }{ \Gamma (\tfrac{1+a}2)^2 } [ t(1-t) ]^{ - \tfrac{1-a}2 }
  \quad \text{and} \quad
  {\overline C} = \frac{ a \pi \Gamma(a) }{ \Gamma (\tfrac{1+a}2)^2 \cos ( \tfrac{a\pi}2 ) }.
\end{equation*}

Next we consider the setting in \cite{HallChapmanOckendon10}: $0 < a < 1$, $[s_1, s_2] = [0, \infty)$ and $\Vext (t) = \gamma t$ for $t \geq 0$ with $\gamma > 0$. In this setting, an explicit formula for ${\overline \rho}$ was still missing. Here we derive this formula (see \eqref{rhob:HCO}). Since $\Vext$ is increasing on $[0, \infty)$, we find from the structure of $E$ that $t_1 = 0$.

Leaving $t_2$ and ${\overline C}$ unknown, we compute ${\overline \rho}$ from \eqref{IF:R:impl}. This gives, using Proposition \ref{prop:Ca},
\begin{equation} \label{rho:HCO2} 
  {\overline \rho} (t) 
  = \tilde {\mathcal C}_a f_{{\overline \rho}} (t)
  = \tilde {\mathcal C}_a ({\overline C} - \gamma t)
  = \frac{ \cos ( \tfrac{a\pi}2 )}{ a \pi } \frac{ a {\overline C} + \tfrac{1-a}{2} \gamma t_2 - \gamma t }{\phit_a (t)}. 
\end{equation} 
Next we compute ${\overline C}$ by using the condition ${\overline \rho} (t_2) = 0$ in \eqref{rho:HCO2}. This yields ${\overline C} = \tfrac{1+a}{2a} \gamma t_2$, and thus
\begin{equation*}
    {\overline \rho} (t) 
  = \frac{ \gamma \cos ( \tfrac{a\pi}2 ) }{ a\pi } (t_2-t)^{ \tfrac{1+a}2 } t^{ - \tfrac{1-a}2 }.
\end{equation*}  
Finally, we determine $t_2$ from the unit mass condition. Using \eqref{Zwil}, we compute
\begin{equation*}
  1
  = \int_0^{t_2} {\overline \rho} (t) \, dt 
  = \gamma \frac{ \cos ( \tfrac{a\pi}2 ) }{ 2 \pi a^2 \Gamma (a) } \Gamma \big( \tfrac{1+a}2 \big)^2 t_2^{1+a}.
\end{equation*} 

Gathering our computations, we have
\begin{gather} \label{rhob:HCO}
  {\overline \rho} (t) 
  = \gamma \frac{ \cos ( \tfrac{a\pi}2 ) }{ \pi a } (t_2-t)^{ \tfrac{1+a}2 } t^{ - \tfrac{1-a}2 },
  \quad {\overline C} = \frac{1+a}{2a} \gamma t_2,
  \quad
  t_2 = \bigg[ \frac\gamma{2 \pi} \frac{ \cos ( \tfrac{a\pi}2 ) }{ a^2 \Gamma (a) } \Gamma \Big( \frac{1+a}2 \Big)^2 \bigg]^{ - \tfrac1{1+a} }.
\end{gather}

\subsection{Improved regularity}
\label{ss:impr:reg}

In this section we use Theorem \ref{t} and Theorem \ref{t:R:extended} to strengthen the results of \cite{SandierSerfaty15} and \cite{GeersPeerlingsPeletierScardia13} as part of applications (A1) and (A2). 

\smallskip
The setting in \cite{SandierSerfaty15} is given by $a = 0$, $[s_1, s_2] = \R$, $\Vreg \equiv 0$ and $\Vext$ a confining potential. The energy $E$ and its minimiser ${\overline \rho}$ (possibly with a disconnected support) are the starting point in \cite{SandierSerfaty15} to derive crystallisation phenomena in 1D log-gases. Instead of proving that ${\overline \rho}$ satisfies the following sufficient properties for their further results,
\begin{subequations} \label{SS15}
\begin{align} \label{SS15:Hol}
  &{\overline \rho} \in C^{\tfrac 12} (\R); \\ \label{SS15:supp}
  &\supp {\overline \rho} \text{ is a finite union of $K \in \N$ compact intervals } [t_1^k, t_2^k]; \\ \label{SS15:LB}
  &\exists \, c > 0 \
    \forall \, k = 1,\ldots,K \
    \forall \, t_1^k \leq t \leq  t_2^k :  
    {\overline \rho} (t) \geq c \sqrt{ (t_2^k - t) (t - t_1^k) },
\end{align}
\end{subequations}
the authors of \cite{SandierSerfaty15} assume that $\Vext$ is chosen such that these properties hold. 

Thanks to \cite{DeiftKriecherbauerMcLaughlin98}, these properties are satisfied whenever $\Vext$ is analytic. Theorem \ref{t:R} extends the choice of $\Vext$ significantly by relaxing the analyticity. Indeed, for convex $\Vext$ (for the precise conditions, see \eqref{for:Vext:assns}), it follows from Theorem \ref{t:R} that \eqref{SS15:supp} is satisfied with $K=1$. If, moreover, $\Vext \in C_{\text{loc}}^{2, \alpha}$ for some $\alpha \in (0,\frac12)$, then \eqref{for:thm:R:expa} implies that \eqref{SS15:Hol} and \eqref{SS15:LB} are satisfied too, possibly for $c = 0$. To guarantee that $c > 0$, it is enough to impose $\Vext \in W_{\text{loc}}^{4, p}$ for some $p > 1$. To see this, we rely on the strict positivity of ${\overline \rho}$ on its support to focus on its behaviour at the endpoints. For convenience, we focus on the left endpoint $t_1$. We reason by contradiction. Suppose $c=0$. Then, we obtain from \eqref{for:thm:R:expa} that ${\overline \rho}' \in C^{\alpha}([t_1 - 1, \frac12 (t_1+t_2)])$ for some $\alpha > 0$. Hence, by Proposition \ref{prop:props:S}.\ref{prop:props:S:log} and Proposition \ref{prop:props:S:Hol}.\ref{prop:props:S:Hol:bdd}, the function
\begin{equation*}
  (V_0 * {\overline \rho})''(t)
  = (V_0 * {\overline \rho}')'(t)
  = \frac d{dt} \int_{t_1 - 1}^{t_2} \log |t - s| {\overline \rho}'(s) \, ds
  = \big( S_{t_1 - 1}^{t_2} ({\overline \rho}') \big)(t)
\end{equation*}
is H\"older continuous around $t_1$. Hence, $h_{{\overline \rho}}''$ is H\"older continuous around $t_1$. Moreover, as in Step 2 of the proof of Lemma \ref{lem:t}, we find that $\liminf_{t \uparrow t_1} h_{{\overline \rho}}''(t) > 0$, which contradicts with \eqref{for:FBVP:classical}.
\smallskip

Next we extend \cite[Thm.~2]{GeersPeerlingsPeletierScardia13}. The setting is given by $a = 0$, $[s_1, s_2] = [0, \infty)$, $\Vext(t) = t$ for $t \geq 0$ and 
\begin{equation*} 
  V(t) := t \coth t - \log | 2 \sinh t |.
\end{equation*}
For this setting, \cite[Thm.~2]{GeersPeerlingsPeletierScardia13} states that Problem \ref{prob:minz:R} has a unique solution ${\overline \rho}$ in $\mathcal P ([0,\infty))$, but no further properties of ${\overline \rho}$ were sought.

Here, we provide such properties by applying Theorem \ref{t:R}. With this aim, we first prove that $V$ satisfies \eqref{for:V:on:R:props}. The conditions in \eqref{for:V:on:R:props:genl} can be proven by direct computations (see, e.g., the appendices of \cite{GeersPeerlingsPeletierScardia13,VanMeurs15}). To prove \eqref{for:V:on:R:props:psi}, we
define the analytic function $\psi(t) := \frac{ \sinh t }t \geq 1$, and rewrite
\begin{equation*} 
  \Vreg (t)
  = t \frac{ \cosh t }{ \sinh t } - \log | 2 \sinh t | + \log |t| 
  = \frac{ \cosh t }{ \psi (t) } - \log | 2 \psi (t) |.
\end{equation*}
Hence, $\Vreg \in C^\infty (\R)$. This proves \eqref{for:V:on:R:props:psi}, and \eqref{for:V:on:R:props:lambda} follows readily.

Thus, Theorem \ref{t:R} applies for any $\ell \in \N_+$ and any $1 < p_0 < 2$. In particular, $\supp {\overline \rho} = [0, t_2]$ with $t_2 < \infty$, ${\overline \rho} (t_2) = 0$, ${\overline \rho} \in C^{1/2}([\tfrac12 t_2, t_2]) \cap C^\infty ((0, t_2))$ and ${\overline \rho} > 0$ on $(0, t_2)$. These properties are enough for the ongoing research on (A2).

\section*{Acknowledgements}

MK is supported by MEXT KAKENHI Grant Number JP17K18733. The work of PvM is funded by the International Research Fellowship of the Japanese Society for the Promotion of Science with the related JSPS KAKENHI Grant Number JP15F15019.

\appendix

\section{Proof of Lemma \ref{lem:props:V}}
\label{app:pf:lem:props:V}

\begin{proof}[Proof of Lemma \ref{lem:props:V}]
Lemma \ref{lem:props:V}.\ref{lem:props:V:Vreg:regy} follows from the local regularity property in \eqref{for:V:props:R} and $\Vreg (t) = -V_a(t)$ for all $|t| \geq b$. Lemma \ref{lem:props:V}.\ref{lem:props:V:V_k} is satisfied for the sequence of convex functions given by $V_k (t) := \int_t^\infty k \wedge (-V')(s) \, ds$ for $t > 0$, with even extension to the negative half-line. 
\smallskip

Next we prove Lemma \ref{lem:props:V}.\ref{lem:props:V:Vhat:bounds}. We start with the upper bound on $\widehat V$. Since $V \in L^1(\R)$, we have $\widehat V \in L^\infty (\R)$, and hence it is sufficient to prove the decay of the tails of $|\widehat V (\omega)|$ for all $|\omega| > 1$. We employ the splitting $V = V_a + \Vreg$. By interpreting $V_a$ as a tempered distribution, we obtain from \cite[1.3.(1)]{Erdelyi54} and \cite[\S 5.2 Example 4]{RichardsYoun07} that
\begin{equation} \label{for:pf:lem:props:V:1}
  \widehat{ V_a } (\omega) 
  = \left\{ \begin{aligned}
    &\frac1{2|\omega|}
    &&\text{if } a = 0, \\
    &\frac{2 \sin (\tfrac{a\pi}2) \Gamma (1-a) }{ |2 \pi \omega|^{1-a} }
    &&\text{if } 0 < a < 1
  \end{aligned} \right\}
  \quad \text{for } |\omega| > 0,
\end{equation}
where the distribution $|\cdot|^{-1}$ is defined by
\begin{equation*}
  \langle \varphi, |\cdot|^{-1} \rangle
  := -\int_\R \varphi'(\omega) \sign (\omega) \log |\omega| \, d\omega,
  \quad \text{for all } \varphi \in \mathcal S (\R).
\end{equation*}
From \eqref{for:pf:lem:props:V:1} we find for all $0 \leq a < 1$ a constant $C_a > 0$ such that 
\begin{equation*}
  \widehat{ V_a } (\omega) \leq C_a (1 + \omega^2)^{- { \tfrac{1-a}2}}
  \quad \text{for all } |\omega| \geq 1.
\end{equation*}
Regarding $\Vreg$, we have by Lemma \ref{lem:props:V}.\ref{lem:props:V:Vreg:regy} that $\Vreg'' \in L^1 (\R)$ for any $0 \leq a < 1$, and thus $\mathcal F (\Vreg'') \in L^\infty (\R)$. By applying basic results from the theory on tempered distributions (see, e.g., \cite[Chap.~5]{RichardsYoun07}, we conclude from 
\begin{equation*}
  \big| 4 \pi^2 \omega^2 \widehat \Vreg (\omega) \big|
  = \big| \widehat{ \Vreg'' } (\omega) \big|
  \leq C
\end{equation*}
that
\begin{equation*}
  \big| \widehat{\Vreg} (\omega) \big| 
  \leq C |\omega|^{-2} 
  \leq \tilde C (1 + \omega^2)^{- { \tfrac{1-a}2}}
  \quad \text{for all } |\omega| \geq 1.
\end{equation*}
Together with \eqref{for:pf:lem:props:V:1} and $V = V_a + \Vreg$ we conclude that the upper bound on $\widehat V$ in Lemma \ref{lem:props:V}.\ref{lem:props:V:Vhat:bounds} holds.

Next we prove the lower bound on $\widehat V$ in Lemma \ref{lem:props:V}.\ref{lem:props:V:Vhat:bounds}. We use the identity 
\begin{align} \notag
  \widehat V (\omega) 
  &= \int_{\R} V(t) \cos (2 \pi \omega t) \, dt 
  = 2 \sum_{\ell=0}^\infty \int_{\tfrac \ell \omega}^{\tfrac {\ell+1}\omega } V(t) \cos (2 \pi \omega t) \, dt \\\label{Vhat:nonneg}
  &= \frac2\omega \sum_{\ell=0}^\infty \int_0^1 V \Big( \frac{\ell + \xi}\omega \Big) \cos (2 \pi \xi) \, d\xi 
  = \frac2\omega \sum_{\ell=0}^\infty - \int_0^1 \frac1\omega V' \Big( \frac{\ell + \xi}\omega \Big) \frac1{2\pi} \sin (2 \pi \xi) \, d\xi \\\notag
  &= \frac1{\pi \omega^2} \sum_{\ell=0}^\infty - \int_0^1 \bigg[ V' \Big( \frac{\ell + \xi}\omega \Big) - V' \Big( \frac{\ell + \tfrac12}\omega \Big) \bigg] \sin (2 \pi \xi) \, d\xi \\\notag
  &\begin{aligned}
       = \frac1{\pi \omega^2} \sum_{\ell=0}^\infty \bigg( 
         &\int_0^{\tfrac12} \bigg[ \int_\xi^{\tfrac12} \frac1\omega V'' \Big( \frac{\ell + \eta}\omega \Big) \, d\eta \bigg] \sin (2 \pi \xi) \, d\xi \\
         &- \int_{\tfrac12}^1 \bigg[ \int_{\tfrac12}^\xi \frac1\omega V'' \Big( \frac{\ell + \eta}\omega \Big) \, d\eta \bigg] \sin (2 \pi \xi) \, d\xi \bigg)
     \end{aligned} \\\label{for:pf:lem:props:V:2}
  &= \frac1{\pi \omega^3} \sum_{\ell=0}^\infty \int_0^{\tfrac12} \bigg[ \int_0^{\tfrac12} V'' \Big( \frac{\ell + \xi + \zeta}\omega \Big) \, d\zeta \bigg] \sin (2 \pi \xi) \, d\xi.
\end{align}

We first prove that $\widehat V$ is positive. In view of \eqref{for:V:props:lambda} we set $\lambda(t) := \essinf_{0 < s < t} V''(s)$. Continuing from \eqref{for:pf:lem:props:V:2}, we use \eqref{for:V:props:lambda} to estimate
\begin{align*}
  \widehat V (\omega) 
  &\geq \frac1{\pi \omega^3} \int_0^{\tfrac12} \int_0^{\tfrac12} \lambda \Big( \frac{\xi + \eta}\omega \Big) \sin (2 \pi \xi) \, d\xi d\eta \\ 
  &\geq \frac1{\pi \omega^3} \int_0^{\tfrac12} \int_0^{\tfrac14} \lambda \Big( \frac{\xi + \eta}\omega \Big) 4 \xi \, d\xi d\eta \\
  &\geq \frac4\pi \int_0^{\tfrac1{2\omega}} \int_0^{\tfrac1{4\omega}} \lambda (x+y) x \, dx dy  
  > 0,
\end{align*}
and thus the lower bound in Lemma \ref{lem:props:V}.\ref{lem:props:V:Vhat:bounds} holds on any bounded interval. To bound the tails of $\widehat V$ from below, we set $c$ and $\varepsilon$ as in \eqref{for:V:props:lambda}, and estimate for any $\omega \geq \tfrac2\varepsilon$
\begin{align*} 
  \widehat V (\omega) 
  &\geq \frac1{\pi \omega^3} \sum_{\ell=0}^\infty \int_0^{\tfrac12} \int_0^{\tfrac12} \lambda \Big( \frac{\ell + \xi + \zeta}\omega \Big) \sin (2 \pi \xi) \, d\xi d\zeta \\
  &\geq \frac1{\pi \omega^2} \sum_{\ell=0}^\infty \frac1\omega \lambda \Big( \frac{\ell + 1}\omega \Big) \int_0^{\tfrac12} \int_0^{\tfrac12} \sin (2 \pi \xi) \, d\xi d\zeta \\
  &= \frac1{2 \pi^2 \omega^2} \sum_{\ell=0}^\infty \frac1\omega \lambda \Big( \frac{\ell + 1}\omega \Big) 
  \geq \frac{1}{2 \pi^2 \omega^2} \int_{\tfrac1\omega}^\varepsilon \lambda ( s ) \, ds 
  \geq \frac{c}{2 \pi^2 \omega^2} \int_{\tfrac1\omega}^\varepsilon \frac1{s^{2+a}} \, ds 
  \geq \tilde c (1 + \omega^2)^{-\tfrac{1 - a}2},
\end{align*}
which completes the proof of Lemma \ref{lem:props:V}.\ref{lem:props:V:Vhat:bounds}.
\end{proof}

\section{Computation of the constant in Theorem \ref{thm:expl:soln:log}}
\label{app:pf:thm:expl:soln:log}

Here we prove that the constant $C$ in \eqref{for:pf:thm:expl:soln:log:4} equals $0$. Dividing by $\phi_0$ and integrating over $t$, it is enough to show that
\begin{align} \label{app:1}
  \int_0^1 \int_0^1 \frac{ \log| t -  s| }{ \phi_0(t) } \,  u( s) \, d  s \, dt 
  = -\int_0^1 \frac{ f(s) }{ \phi_0(s) } \, ds
\end{align}
where $u = \mathcal C_0 f$ and $f \in C^{1,\alpha}([0,1])$ for some $0 < \alpha < 1$.

We start by computing the left-hand side of \eqref{app:1}. Using Fubini's Theorem, we rewrite
\begin{equation} \label{app:2}
  \int_0^1 \int_0^1 \frac{ \log| t -  s| }{ \phi_0(t) } \,  u( s) \, d  s \, dt
  = \int_0^1 \bigg( \int_0^1 \frac{ \log| t -  s| }{ \phi_0(t) } \, dt \bigg) \,  u( s) \, d  s.
\end{equation}
With Proposition \ref{prop:props:S}.\ref{prop:props:S:log} and \eqref{EK00} we obtain that the term in parentheses satisfies
\begin{equation*}
  \frac d{ds} \int_0^1 \frac{ \log| t -  s| }{ \phi_0(t) } \, dt
  = \Big( S \frac1{\phi_0} \Big) (s) = 0.
\end{equation*}
Hence, the term in parentheses in \eqref{app:2} is constant in $s$. We evaluate it at $s = \tfrac12$:
\begin{multline*}
  \int_0^1 \frac{ \log| t - \tfrac12 | }{ \phi_0(t) } \, dt
  = 2 \int_{\tfrac12}^1 \frac{ \log| t - \tfrac12 | }{ \sqrt{t(1-t)} } \, dt
  = 2 \int_0^1 \frac{ \log \tfrac r2 }{ \sqrt{1 - r^2} } \, dr \\
  = 2 \int_0^1 \frac{ \log r }{ \sqrt{1 - r^2} } \, dr - 2 \log 2 \int_0^1 \frac{dr}{ \sqrt{1 - r^2} }.
\end{multline*}
Both integrals are evaluated respectively in \cite[4.241.7]{GradshteynRyzhik14} and \eqref{Zwil}. This yields
\begin{align*}
  \int_0^1 \frac{ \log| r - \tfrac12 | }{ \phi_0(r) } \, dr
  = - 2 \pi \log 2.
\end{align*}
Inserting this result in \eqref{app:2} and using $u = \mathcal C_0 f$ and $\int_0^1 \frac1{\phi_0} = \pi$, we get
\begin{multline*} 
  \int_0^1 \int_0^1 \frac{ \log| t -  s| }{ \phi_0(t) } \,  u( s) \, d  s \, dt 
  = - 2 \pi \log 2 \int_0^1 u(s) \, ds \\
  = - \frac{2 \log 2}\pi \int_0^1 \frac{S ( \phi_0 f' )(s)}{\phi_0(s)} \, ds
    - \int_0^1 \frac{ f(s) }{ \phi_0(s) } \, ds.
\end{multline*}
It is left to show that the first term in the right-hand side equals $0$. Since $f' \in C^\alpha([0,1])$, we obtain this from Proposition \ref{prop:props:S}.\ref{prop:props:S:Fubini} and \eqref{EK00} by
\begin{equation*}
  \int_0^1 \frac1{\phi_0} S ( \phi_0 f' )
  = - \int_0^1 S \Big( \frac1{\phi_0} \Big) \, \phi_0 f'
  = 0.
\end{equation*}

\bibliographystyle{alpha}

\end{document}